\documentclass [12pt]{amsart}

\usepackage{mathtools} 
\usepackage[top=100pt,bottom=60pt,left=96pt,right=92pt]{geometry}
\usepackage[mathscr]{eucal}
\usepackage{hyperref} 
\usepackage{amsmath}
\usepackage{bm} 
\usepackage{amsthm}
\usepackage{amssymb}
\usepackage{amsfonts} 
\usepackage{mathrsfs}
\usepackage{amscd}
\usepackage{enumitem} 

\usepackage[pdftex]{color}
\usepackage{xcolor}
\usepackage{layout}
\usepackage{stmaryrd} 
\usepackage{float} 
\usepackage{longtable}
\usepackage[pdftex]{graphicx}
\usepackage{txfonts} 
\usepackage{setspace} 
\usepackage{subcaption} 
\usepackage{verbatim}
\usepackage{nicefrac}
\usepackage{textcomp}

\newtheorem{theorem}{Theorem}[section]
\newtheorem{proposition}[theorem]{Proposition}
\newtheorem{corollary}[theorem]{Corollary}
\newtheorem{lemma}[theorem]{Lemma}

\newtheorem{definition}[theorem]{Definition}
\newtheorem{remark}[theorem]{Remark}

\theoremstyle{plain}
\newtheorem{notation}[theorem]{Notation}

\newcommand{\purge}[1]{} 

\newcommand{\vungoc}{V\~u Ng\d{o}c}

\def\epsilon{\varepsilon}
\def\phi{\varphi}
\def\theta{\vartheta}

\newcommand{\ga}{\gamma}
\newcommand{\de}{\delta}

\newcommand{\ze}{\zeta}

\newcommand{\ka}{\kappa}
\newcommand{\lam}{\lambda}

\newcommand{\si}{\sigma}

\newcommand{\om}{\omega}

\newcommand{\De}{\Delta}

\newcommand{\ati}{{\ti{a}}}
\newcommand{\bti}{{\ti{b}}}
\newcommand{\cti}{{\ti{c}}}

\newcommand{\Fti}{{\ti{F}}}

\newcommand{\Lti}{{\ti{L}}}

\newcommand{\zbar}{\bar{z}}

\newcommand{\zhat}{{\hat{z}}}

\def\C{{\mathbb C}}

\def\N{{\mathbb N}}

\def\R{{\mathbb R}}
\def\mbS{{\mathbb S}} 
\def\T{{\mathbb T}}

\def\Z{{\mathbb Z}}

\newcommand{\mcF}{\mathcal F}

\newcommand{\mcK}{\mathcal K}
\newcommand{\mcL}{\mathcal L}
\newcommand{\mcM}{\mathcal M}

\newcommand{\mcP}{\mathcal P}

\newcommand{\mcX}{\mathcal X}

\newcommand{\mfI}{\mathfrak I}

\newcommand{\mfR}{\mathfrak R}

\newcommand{\mff}{\mathfrak f}
\newcommand{\mfg}{\mathfrak g}

\newcommand{\mfp}{\mathfrak p}

\newcommand{\mfs}{\mathfrak s}

\newcommand{\ti}{\tilde}
\newcommand{\x}{\times}
\newcommand{\del}{\partial}

\newcommand{\beq}{\begin{equation}}
\newcommand{\eeq}{\end{equation}}
\newcommand{\beqs}{\begin{equation*}}
\newcommand{\eeqs}{\end{equation*}}

\DeclarePairedDelimiter{\abs}{\lvert}{\rvert}

\DeclareMathOperator{\Id}{Id}

\DeclareMathOperator{\Mod}{mod}

\DeclareMathOperator{\sign}{sign}
\DeclareMathOperator{\Span}{Span}

\def\slashii#1{\setbox0=\hbox{$#1$}             
\dimen0=\wd0                                 
\setbox1=\hbox{\sl/} \dimen1=\wd1            
\ifdim\dimen0>\dimen1                        
\rlap{\hbox to \dimen0{\hfil\sl/\hfil}}   
#1                                        
\else                                        
\rlap{\hbox to \dimen1{\hfil$#1$\hfil}}   
\hbox{\sl/}                               
\fi}                                         %
\def\slashiii#1{\setbox0=\hbox{$#1$}#1\hskip-\wd0\hbox to\wd0{\hss\sl/\/\hss}}
\newcommand{\refintroToric}{Theorem \ref{introToric}}

\newcommand{\refintroSemitoric}{Theorem \ref{introSemitoric}} 

\newcommand{\refintroPinchedTorus}{Proposition \ref{introPinchedTorus}}

\newcommand{\refthmEliasson}{Theorem \ref{thmEliasson}}

\newcommand{\refpropEigValues}{Proposition \ref{propEigValues}}

\newcommand{\refnondeg}{Lemma \ref{nondeg}}

\newcommand{\refmarsdenWeinstein}{Theorem \ref{marsdenWeinstein} (Marsden-Weinstein)}

\newcommand{\refredRankOne}{Lemma \ref{redRankOne}}

\newcommand{\refnondegRankOne}{Definition \ref{nondegRankOne}}

\newcommand{\refparamU}{Lemma \ref{paramU}}

\newcommand{\refthOctagon}{Theorem \ref{thOctagon}}

\newcommand{\refcoordEEPoints}{Proposition \ref{coordEEPoints}}

\newcommand{\refdegPoint}{Proposition \ref{degPoint}}

\newcommand{\refzeroToOne}{Proposition \ref{zeroToOne}}

\newcommand{\refXY}{Lemma \ref{XY}}

\newcommand{\refmjred}{Lemma \ref{mjred}}

\newcommand{\refXJH}{Lemma \ref{XJH}}

\newcommand{\refHtFt}{Proposition \ref{HtFt}}

\newcommand{\refmainTheorem}{Theorem \ref{mainTheorem}}

\newcommand{\refdoublePinchParam}{Proposition \ref{doublePinchParam}}

\newcommand{\refrankOpoints}{Proposition \ref{rankOpoints}}

\newcommand{\refAeeffee}{Proposition \ref{Aeeffee}}

\newcommand{\refeeFixedPoints}{Proposition \ref{eeFixedPoints}}

\newcommand{\refentriesRkOne}{Lemma \ref{entriesRkOne}}

\newcommand{\refellRegVertical}{Proposition \ref{ellRegVertical}}

\newcommand{\refellRegHorizontal}{Proposition \ref{ellRegHorizontal}}

\begin{document}

\title[Semitoric systems with four focus-focus singularities]{A family of semitoric systems with four focus-focus singularities and two double pinched tori}

\author{Annelies De Meulenaere $\&$ Sonja Hohloch}

\date{\today}

\begin{abstract}
We construct a 1-parameter family $F_t=(J, H_t)_{0 \leq t \leq 1}$ of integrable systems on a compact $4$-dimensional symplectic manifold $(M, \om)$ that changes smoothly from a toric system $F_0$ with eight elliptic-elliptic singular points via toric type systems to a semitoric system $F_t$ for $ t^- < t < t^+$. These semitoric systems $F_t$ have precisely four elliptic-elliptic and four focus-focus singular points. Moreover, at $t= \frac{1}{2}$, the system has precisely two focus-focus fibres each of which contains exactly two focus-focus points, giving these fibres the shape of double pinched tori. We exemplarily parametrise one of these fibres explicitly.
\end{abstract}

\maketitle


\section{{\bf Introduction}}

Integrable systems lie at the intersection of many areas in mathematics and physics like, for example, dynamical systems, ODEs, PDEs, symplectic geometry, Lie theory, classical mechanics, mathematical physics etc. Integrable systems display a `certain amount of order' due to being `not chaotic' and having their flow lines stay in the fibres of the momentum map. Their global behaviour, nevertheless, can be very intricate. When it comes to their singularities, nondegenerate singularities can be classified by means of a local normal form that splits a singularity into hyperbolic, elliptic, regular, and focus components where in fact the latter always comes as a pair, usually thus referred to as focus-focus block. 

In this paper, we will construct a special 1-parameter family of integrable systems of two degrees of freedom on a $4$-dimensional manifold which displays an interesting bifurcation behaviour and where all mentioned types of components will appear except for the hyperbolic ones. Since a completely integrable system gives rise to a Lagrangian fibration this explicit family may also become of interest for the study of low dimensional singular Lagrangian fibrations and, for instance, the (non)displaceability of its fibers under Hamiltonian isotopies, i.e., questions of symplectic rigidity.

Let us now be more precise.
Given a $4$-dimensional, connected, symplectic manifold $(M, \om)$, a \emph{semitoric system} is a completely integrable system $F:=(J, H): M \to \R^2$ where $F$ only has nondegenerate singularities that have no hyperbolic components and where $J: M \to \R$ is proper and induces an effective Hamiltonian $\mbS^1$-action. For instance, coupled spin oscillators and coupled angular momenta are semitoric.

Contrary to toric systems on $4$-dimensional manifolds that give rise to an $\mbS^1 \times \mbS^1$-action and only admit elliptic-elliptic and elliptic-regular singular points, semitoric systems give rise to an $\mbS^1 \times \R$-action and admit in addition focus-focus singularities. 

Whereas toric systems are classified (up to equivariant symplectimorphism) by precisely one invariant, namely by the image of their momentum map (cf.\ Delzant \cite{delzant}), the classification of semitoric systems is more involved: apart from one invariant based on a `straightened' image of the momentum map, Pelayo $\&$ \vungoc\ \cite{pelayoVungocInv, pelayoVungocActa} showed that there are four other invariants necessary for a symplectic classification, namely the total number of focus-focus points, the position of their images in the  straightened image of the momentum map, a Birkhoff normal form type invariant (called `Taylor series invariant') that describes the impact a focus-focus point has on the neighbourhood of its fibre, and, last but not least, the so-called twisting index invariant that describes the dynamical interaction between the different focus-focus fibres. Originally this classification applied only to semitoric systems with maximally one focus-focus point per fibre. Recently, Palmer $\&$ Pelayo $\&$ Tang \cite{ppt} generalised the invariants to allow for more than one focus-focus point per fibre.

During the last decade, semitoric systems gained more and more attention and popularity. This is due to the fact that, although their classification is more involved than the one of toric systems, it is still `doable and constructive': Pelayo $\&$ \vungoc\ \cite{pelayoVungocActa} showed, how to construct for given data that are admissible as invariants, a semitoric system having these data as invariants. They constructed the systems by gluing together coordinate patches containing the necessary singular points etc. But patching together plus smoothing along the seams of the patches is not very suitable for {\em explicit observations} how a semitoric system and its invariants change under variation of parameters. 

For this, it is much more practical to have systems that are defined globally by one explicit formula on an explicitly given symplectic manifold. Recently, there has been progress in this direction:
Alonso $\&$ Dullin $\&$ Hohloch \cite{adh1, adh2} completed the computation of the classifying invariants for coupled spin oscillators with varying radius and coupled angular momenta with two varying radii and a coupling parameter. Hohloch $\&$ Palmer \cite{hohlochPalmer} generalised the coupled angular momenta to an explicit semitoric family with two coupling parameters that has two distinct focus-focus points for certain intervals of the parameters. Le Floch $\&$ Palmer \cite{leFlochPalmer} pushed this further to explicit families of semitoric systems on Hirzebruch surfaces. By using invariant functions to perturb a given toric system, they came up with a method that seemed to work for more general situations than Hirzebruch surfaces.

This technique plus the question how difficult it is to construct {\em explicit} semitoric systems with {\em more than two} focus-focus points motivated the present paper: We construct a 1-parameter family of integrable systems $F_t$ that is at time $t=0$ a toric system, then, as soon as $t>0$, it becomes a system of toric type until some $0< t^- <\frac{1}{2}$. At $t=t^-$, four singular points undergo a Hamiltonian-Hopf bifurcation, changing from elliptic-elliptic to focus-focus. As soon as $t>t^-$, the system is semitoric with four focus-focus points until some $t^+$ with $\frac{1}{2}< t^+<1$ where the focus-focus points undergo again a Hamiltonian-Hopf bifurcation that renders them elliptic-elliptic. Moreover, at time $t= \frac{1}{2}$, the focus-focus points team up in pairs of two that lie in the same focus-focus fibre, effectively creating thus two singular fibres that look like double pinched tori. For $t^+<t \leq 1$, the system is again of toric type. 

The precise construction of $F_t$ goes as follows: First, starting with the octagon in Figure \ref{fig_octagon}, we follow the steps of Delzant's \cite{delzant} construction as described in detail in Cannas da Silva \cite{cannasDaSilva} to obtain a $4$-dimensional, compact, connected, symplectic manifold $(M, \om):=(M_\De, \om_\De)$ by using symplectic reduction by a Hamiltonian $\T^6$-action of the $10$-dimensional preimage of a certain map from $\C^8$ to $\R^6$ (the details are given in Section \ref{section toric}). Points on $(M, \om)$ are usually written as equivalence classes of the form $[z] = [z_1, \dots, z_8]$ with $z_k = x_k +i y_k \in \C$ for $1 \leq k \leq 8$.
The momentum map of the toric system on $(M, \om)$ is in fact surprisingly simple:

\begin{figure}[h]
\centering
\includegraphics[scale=.18]{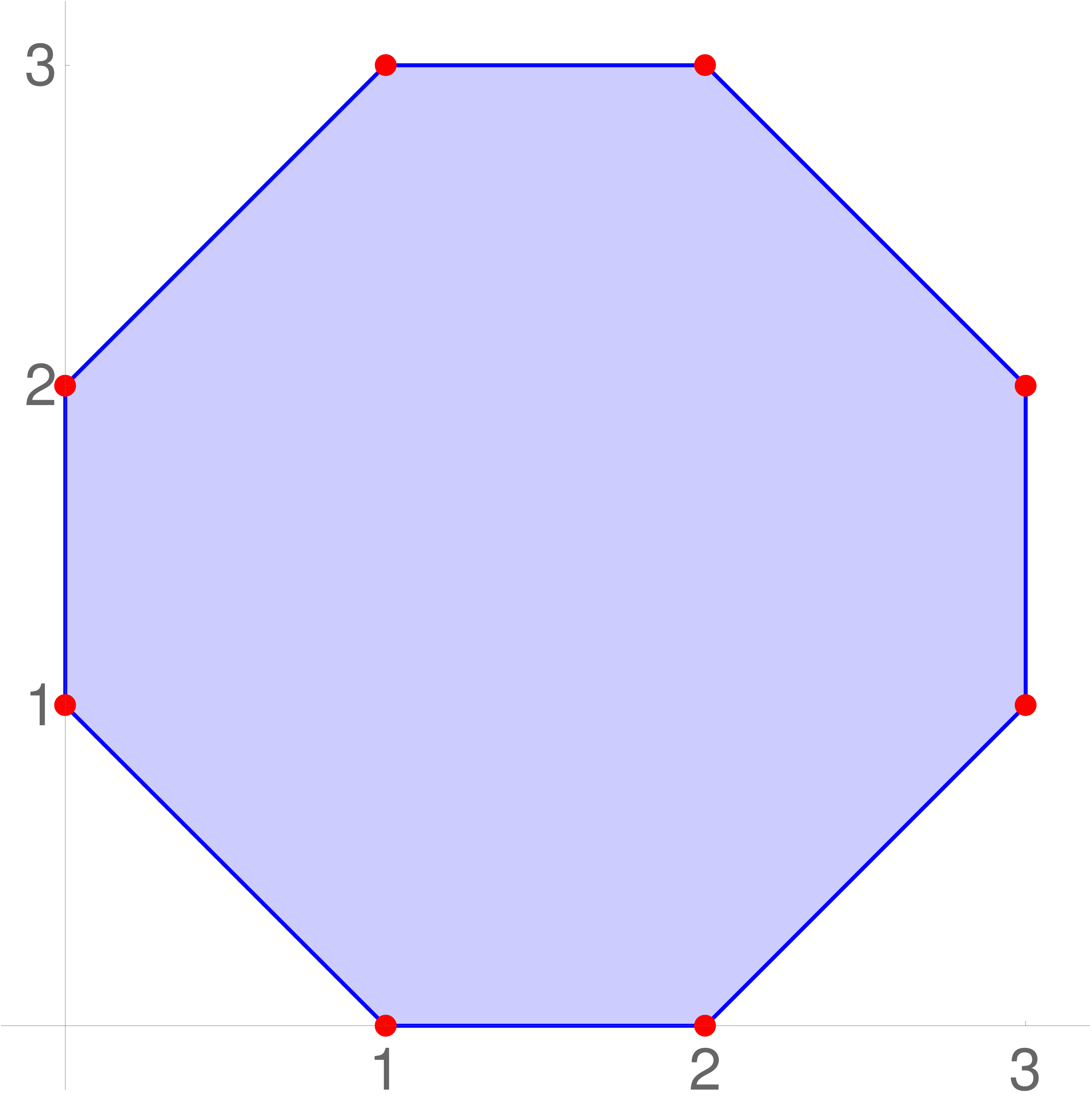}
\caption{The octagon $\De \subset \mathbb{R}^2$, obtained by chopping off the corners of the square $[0,3] \times [0,3]$, is a Delzant polytope.
}
\label{fig_octagon}
\end{figure}

\begin{theorem}
\label{introToric}
Let $\De$ be the octagon displayed in Figure \ref{fig_octagon}.
Then $F = (J,H): (M, \om) \to  \mathbb{R}^2$ given by
\begin{equation*} 
J([z_1, \ldots, z_8]) = \frac{1}{2}\vert z_1 \vert^2, \qquad  H([z_1, \ldots, z_8]) = \frac{1}{2}\vert z_3 \vert^2
\end{equation*}
is a momentum map of an effective Hamiltonian 2-torus action satisfying $F(M)= \De$. Thus, in particular, $F$ has precisely eight elliptic-elliptic singular points. 
\end{theorem}

This theorem is restated as \refthOctagon\ and proven throughout Section \ref{section toric}. In addition, in \refcoordEEPoints, we compute the precise coordinates of the eight elliptic-elliptic fixed points of $F=(J, H)$.

\vspace{3mm}

Now we look for a suitable function to perturb the integral $H$ of $F=(J, H)$ with. To this aim, consider the function $Z: \C^8 \to \C$ given by $Z(z_1, \ldots, z_8) :=\overline{z_2} \: \overline{z_3} \:\overline{z_4}\: z_6 \: z_7 \: z_8$. We will see in \refXY\ that it descends to a function $Z:(M, \om) \to \C$ and so do its real part $\mfR(Z)=:X $ and imaginary part $\mfI(Z)=:Y$. Both functions are in addition invariant under $J$ so that they also pass to the reduced spaces $M^{red, j}$ for $j \in J(M)$. Now we vary $H$ via linear combination $(1-2t) H + t\ga X $ for $\ga>0 $ sufficiently small and obtain 

\begin{theorem}
\label{introSemitoric}
 Let $(M, \om, F=(J, H))$ be the toric system from \refintroToric\ and let $ 0 < \gamma < \frac{1}{48}$ and set $F_t :=(J, H_t):= (J, \ (1-2t) H + t\ga X) : (M, \om) \to \R^2$. 
Then $(M, \omega, F_t)_{0 \leq t \leq 1}$ is toric for $t=0$, of toric type for $0<t<t^-$, semitoric for $t^- < t < t^+$, and again of toric type for $t^+< t\leq 1$ where 
$$ 0  \ < \ t^- := \frac{1}{2(1+24\gamma)} \ < \ \frac{1}{2} \ < \  t^+ := \frac{1}{2(1-24\gamma)} \ < \  1. $$
For all $t \in [0,1]$, the system $F_t$ has precisely eight fixed points of which four are always elliptic-elliptic. The other four pass at $t=t^-$ from elliptic-elliptic via a Hamiltonian-Hopf bifurcation to focus-focus. At $t=t^+$, these four focus-focus points turn again back into elliptic-elliptic points via a Hamiltonian-Hopf bifurcation. For the special situation occuring at $t=\frac{1}{2}$, see \refintroPinchedTorus.
\end{theorem}

\begin{figure}[h!]
\centering
\begin{subfigure}{0.33\textwidth}
\centering
\caption*{$t = 0$}
\includegraphics[width=.9\linewidth]{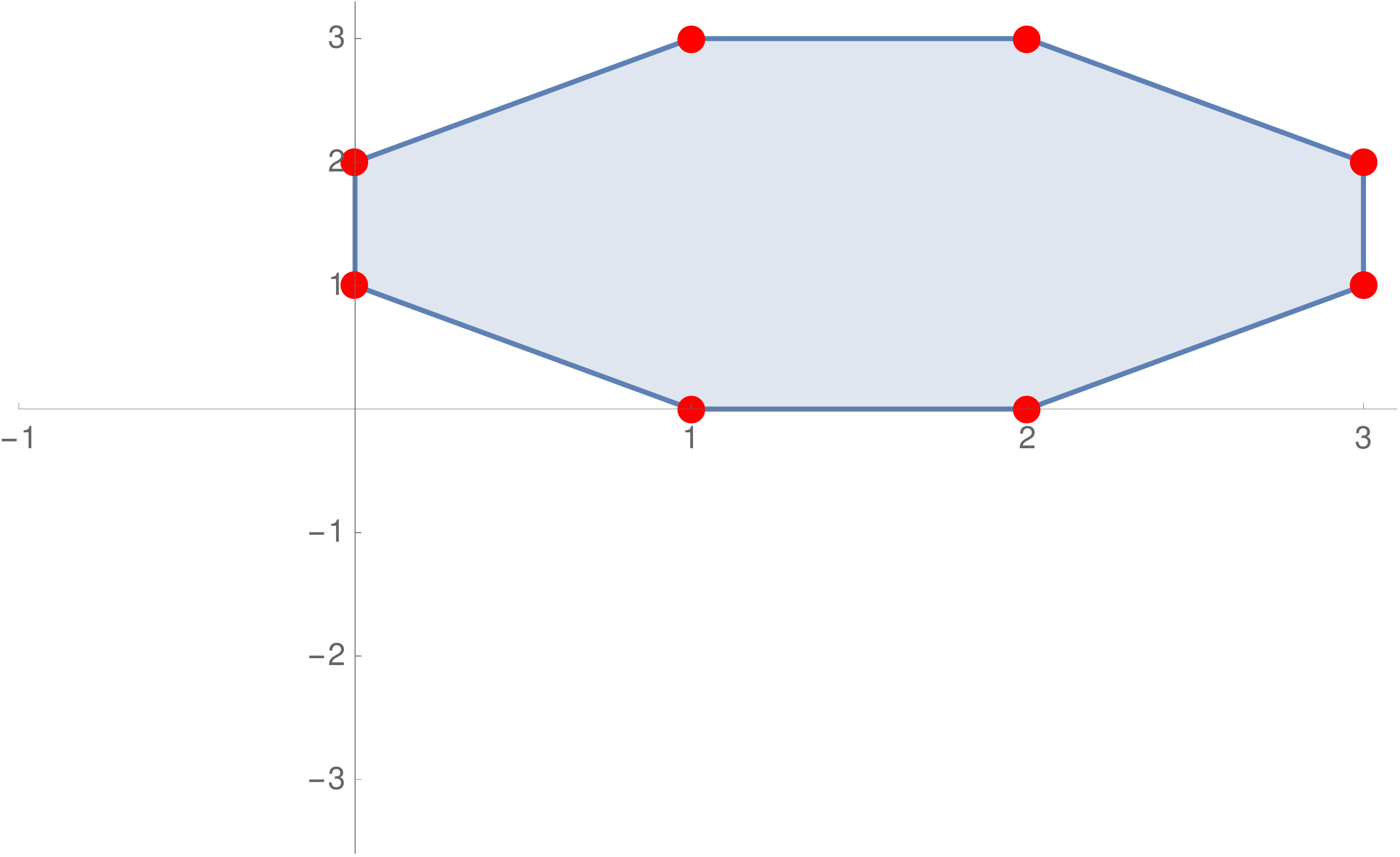}
\end{subfigure}%
\begin{subfigure}{0.33\textwidth}
\centering
\caption*{$t = 0.07$}
\includegraphics[width=.9\linewidth]{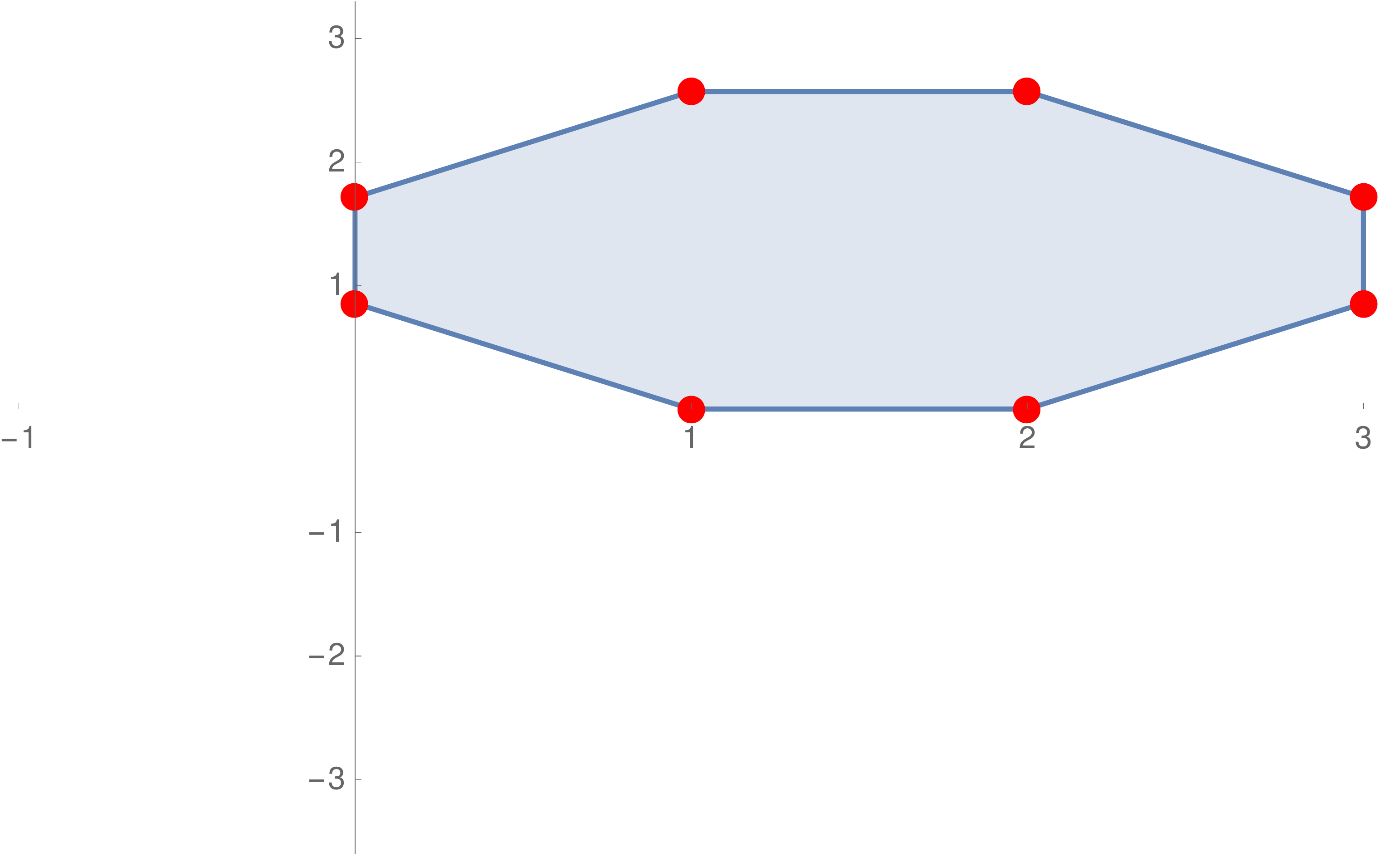}
\end{subfigure}%
\begin{subfigure}{0.33\textwidth}
\centering
\caption*{$t = 0.14$}
\includegraphics[width=.9\linewidth]{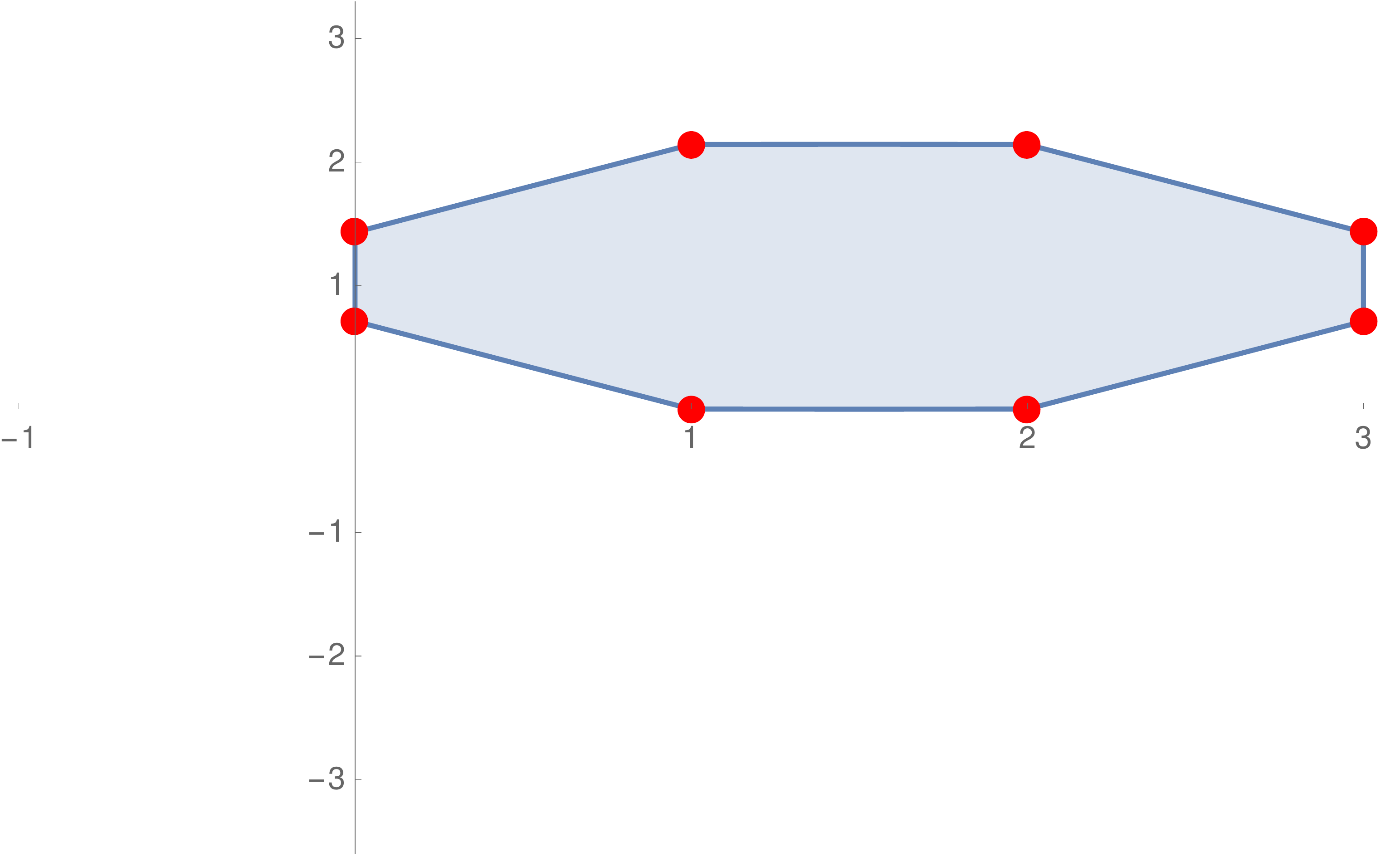}
\end{subfigure}

\begin{subfigure}{0.33\textwidth}
\centering
\caption*{$t = 0.21$}
\includegraphics[width=.9\linewidth]{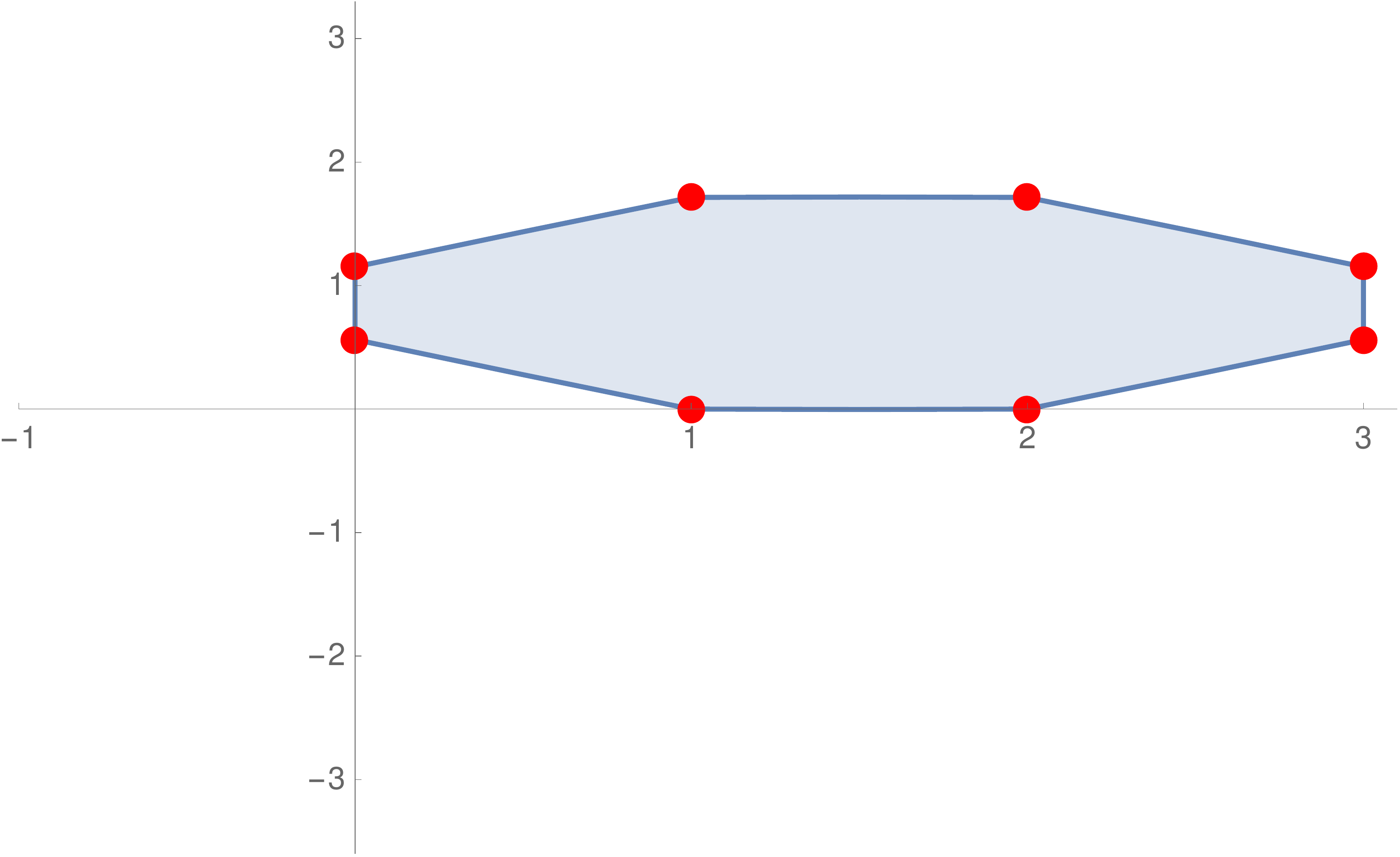}
\end{subfigure}%
\begin{subfigure}{0.33\textwidth}
\centering
\caption*{$t = 0.29$}
\includegraphics[width=.9\linewidth]{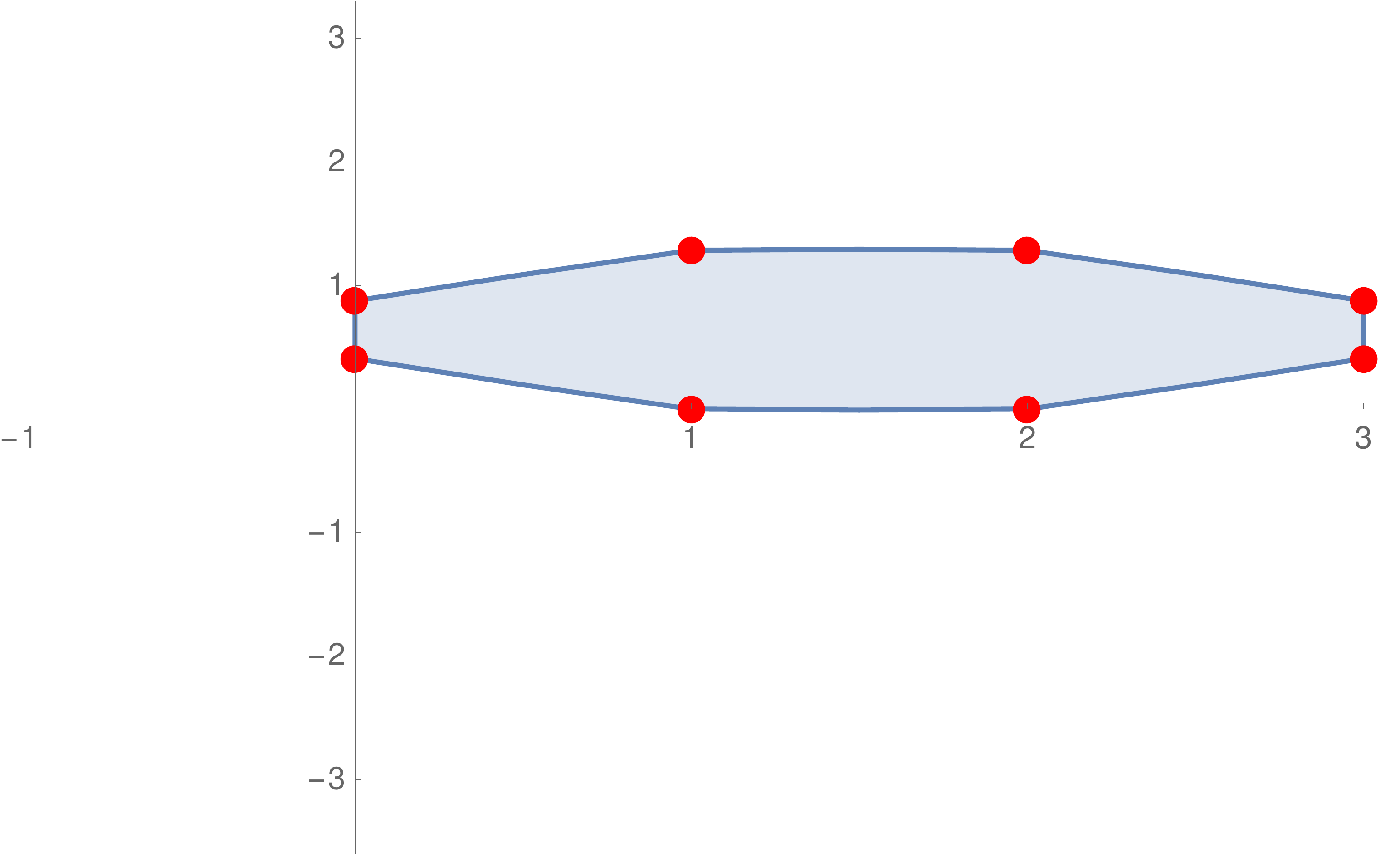}
\end{subfigure}%
\begin{subfigure}{0.33\textwidth}
\centering
\caption*{$t = 0.36$}
\includegraphics[width=.9\linewidth]{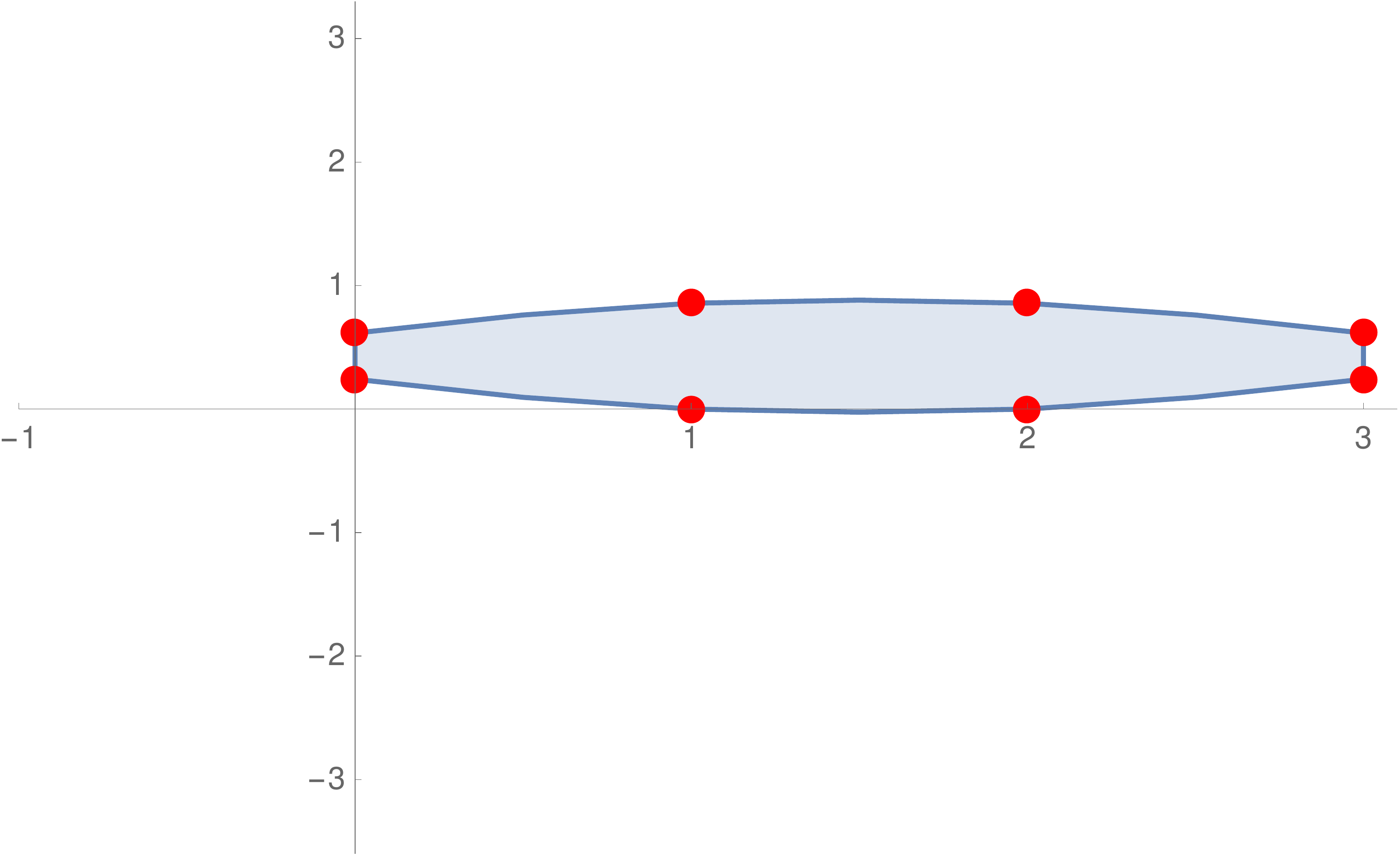}
\end{subfigure}

\begin{subfigure}{0.33\textwidth}
\centering
\caption*{$t = 0.43$}
\includegraphics[width=.9\linewidth]{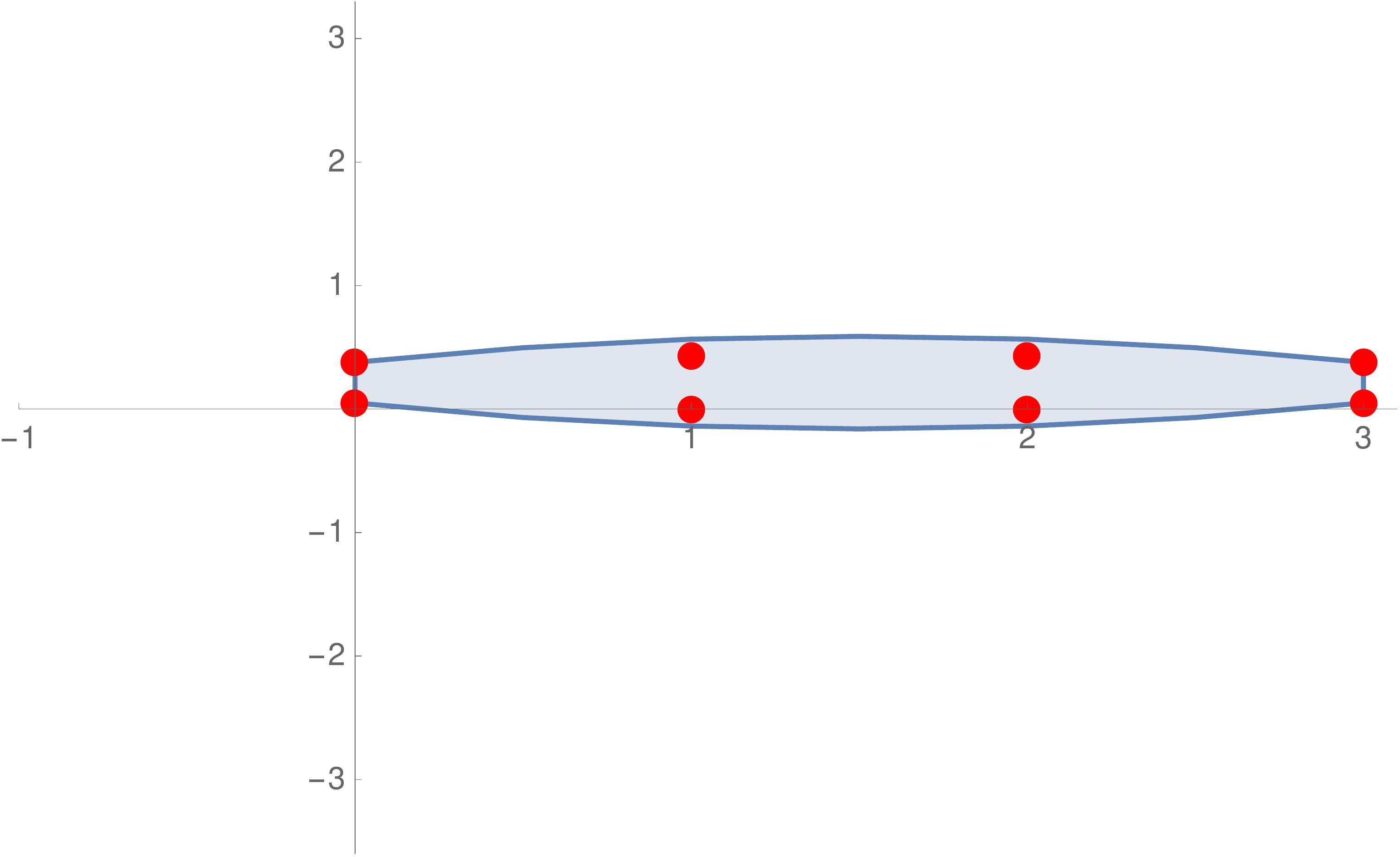}
\end{subfigure}%
\begin{subfigure}{0.33\textwidth}
\centering
\caption*{$t = 0.5$}
\includegraphics[width=.9\linewidth]{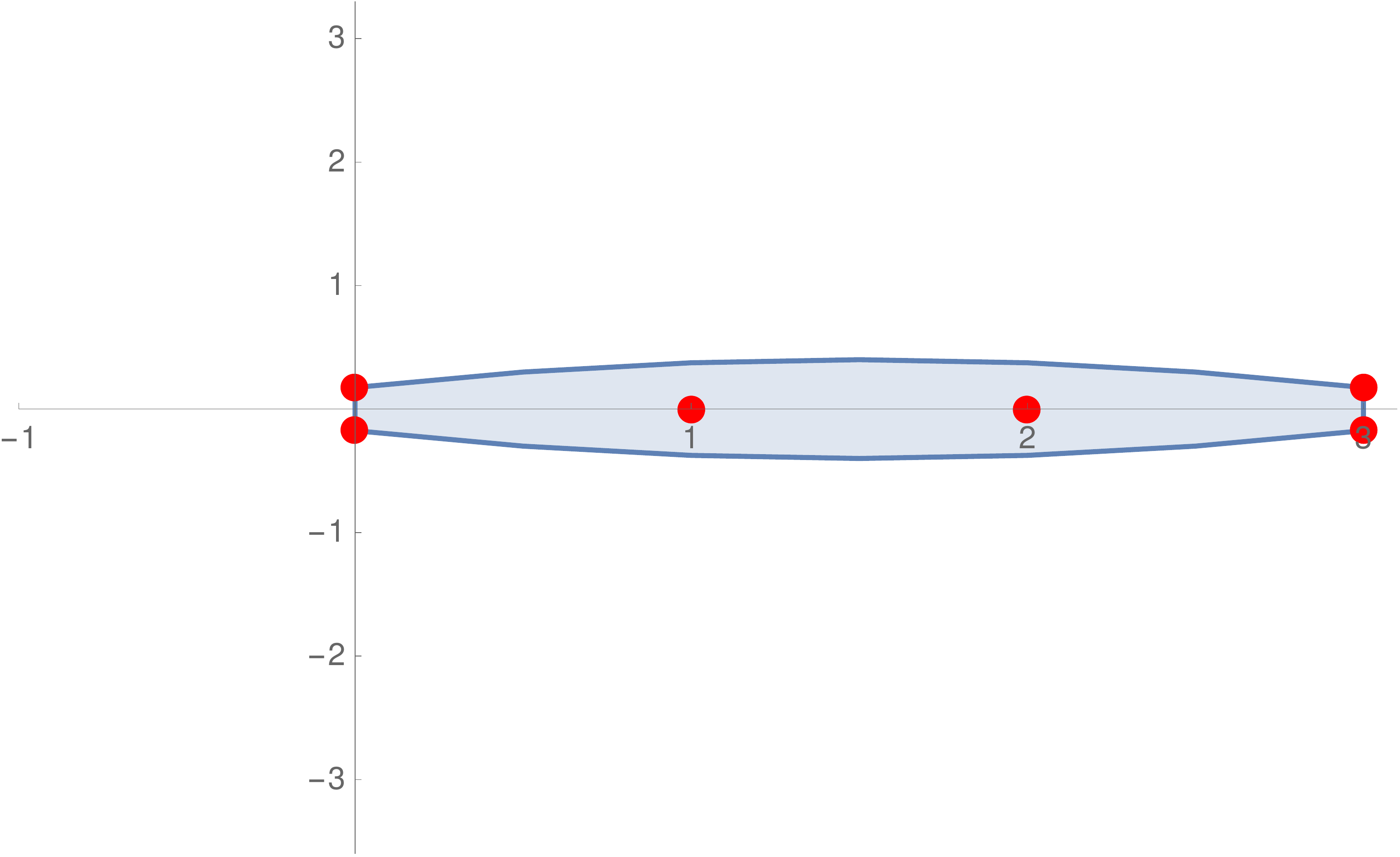}
\end{subfigure}%
\begin{subfigure}{0.33\textwidth}
\centering
\caption*{$t = 0.57$}
\includegraphics[width=.9\linewidth]{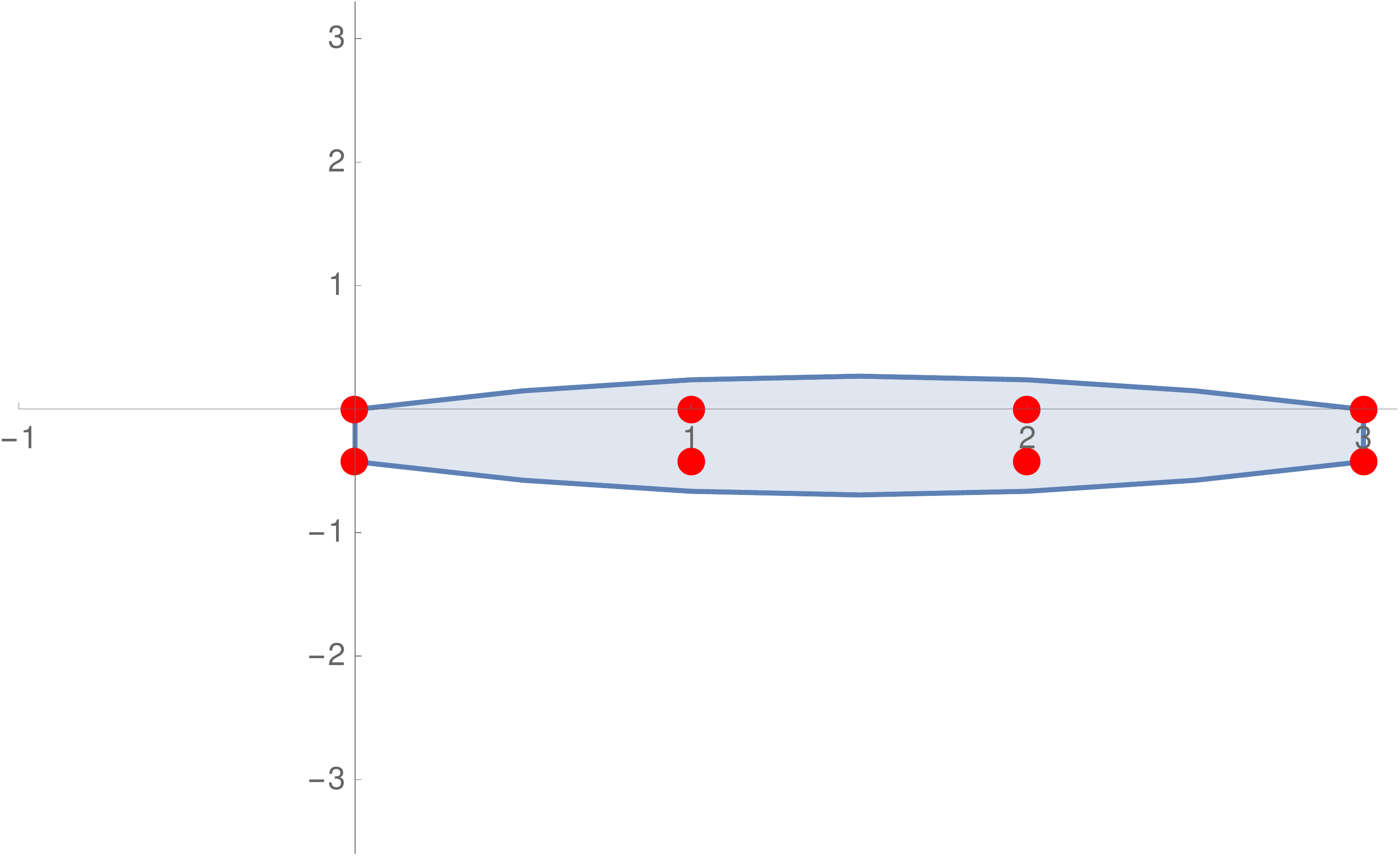}
\end{subfigure}

\begin{subfigure}{0.33\textwidth}
\centering
\caption*{$t = 0.64$}
\includegraphics[width=.9\linewidth]{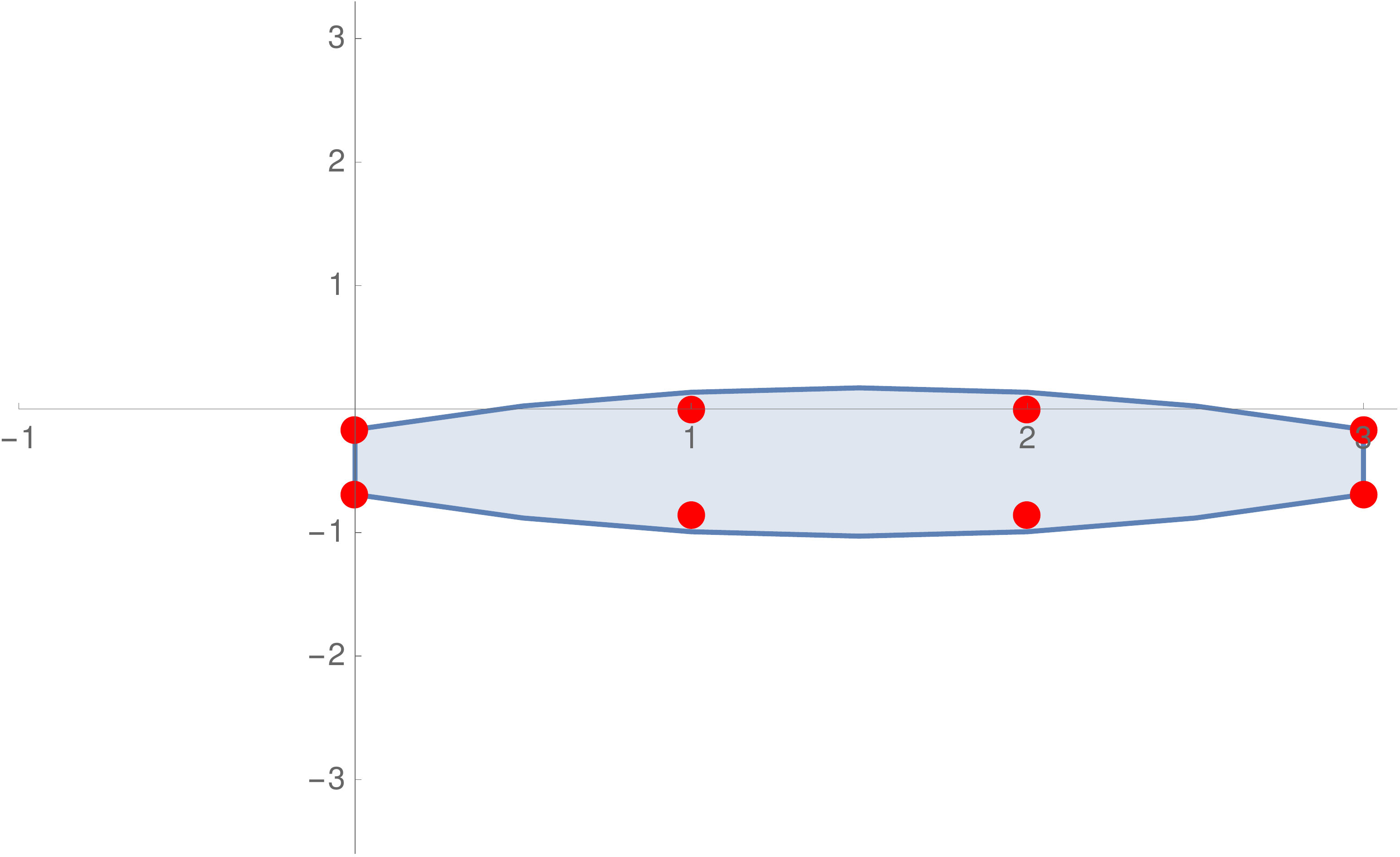}
\end{subfigure}%
\begin{subfigure}{0.33\textwidth}
\centering
\caption*{$t = 0.71$}
\includegraphics[width=.9\linewidth]{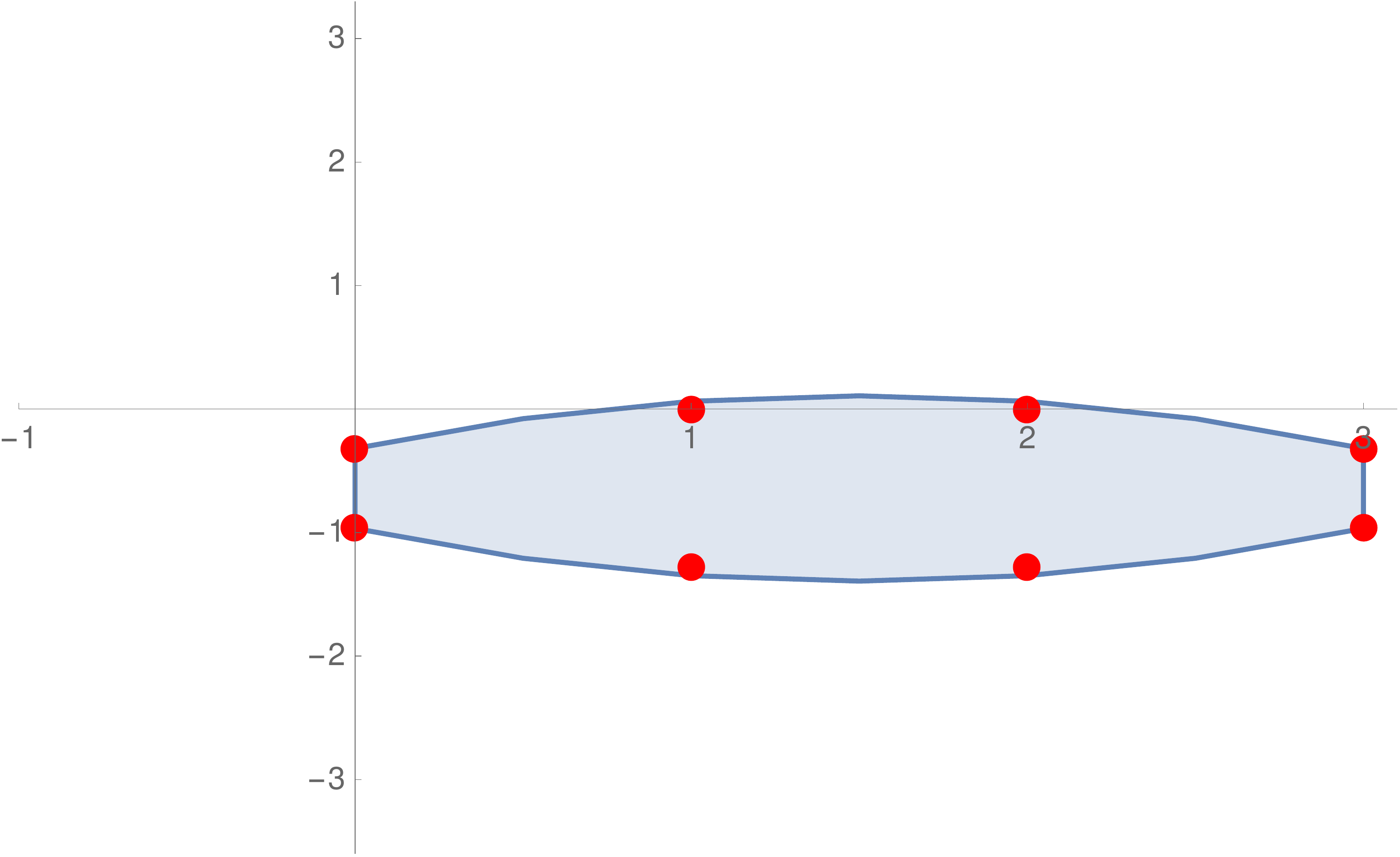}
\end{subfigure}%
\begin{subfigure}{0.33\textwidth}
\centering
\caption*{$t = 0.79$}
\includegraphics[width=.9\linewidth]{Nfig-MMI-11.pdf}
\end{subfigure}

\begin{subfigure}{0.33\textwidth}
\centering
\caption*{$t = 0.86$}
\includegraphics[width=.9\linewidth]{Nfig-MMI-11.pdf}
\end{subfigure}%
\begin{subfigure}{0.33\textwidth}
\centering
\caption*{$t = 0.93$}
\includegraphics[width=.9\linewidth]{Nfig-MMI-11.pdf}
\end{subfigure}%
\begin{subfigure}{0.33\textwidth}
\centering
\caption*{$t = 1$}
\includegraphics[width=.9\linewidth]{Nfig-MMI-11.pdf}
\end{subfigure}

\caption{The image of the momentum map $F_t(M)$ of the family $(M, \om, F_t)$ from \refintroSemitoric\ plotted with {\em Mathematica} for fifteen time steps between $t=0$ and $t=1$ with $\gamma = \frac{1}{60}$.
}
\label{Fig_maintheorem}
\end{figure}

The momentum map image $F_t(M)$ is plotted for various values of $t$ in Figure \ref{Fig_maintheorem}. 
\refintroSemitoric\ is restated as \refmainTheorem\ in Section \ref{section semitoric}. Moreover, \refrankOpoints\ computes the explicit coordinates of the eight fixed points of $F_t$ for time $0 \leq t \leq 1$, thus extending \refcoordEEPoints\ from the toric situation at $t=0$ to $t \in [0, 1]$.

Let us remark that, in the terminology of Kane $\&$ Palmer $\&$ Pelayo \cite{kpp}, the systems $(F_t)_{t^-< t < t^+ }$ are minimal of type (6) with $k=-2$ and $c=4$ and $d=4$ (see \cite[Theorem 2.4 and the table in Theorem 4.15]{kpp}).

The proof of \refintroSemitoric\ resp.\ \refmainTheorem\ is done in several steps spread over the following sections: 
\begin{itemize}
 \item 
 In Section \ref{section posRankZero}, we show the existence and determine the positions of the fixed points. 
 \item
 In Section \ref{section typeRankZero}, we show that the fixed points are nondegenerate and we determine their type.
 \item
 In Section \ref{section rankOneType}, we determine the rank one points and show that they are nondegenerate and of elliptic-regular type. 
 \item
 In Section \ref{section pinched}, we summarise all steps and prove \refmainTheorem\ and hereby also \refintroSemitoric.
\end{itemize}
At $t=\frac{1}{2}$, the fibres over $(1,0)$ and $(2, 0)$ contain precisely two focus-focus points each. These fibres can be thought of as double pinched tori as sketched in Figure \ref{Fig_double_pinched} and we exemplarily parametrise the one over $(1, 0)$ explicitly.

\begin{figure}[h]
\centering
\includegraphics[scale=.3]{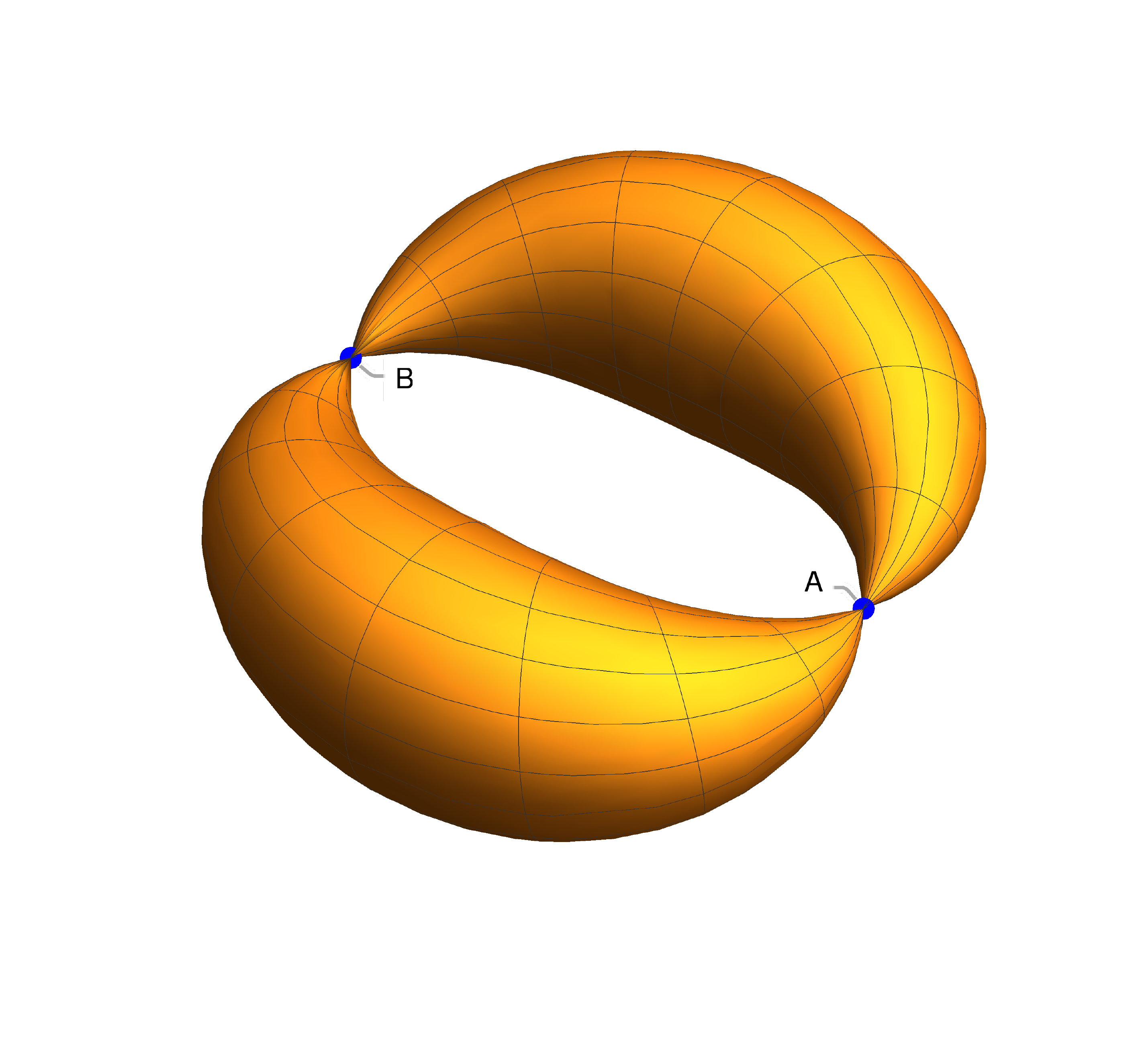}
\caption{The fibre $F_\frac{1}{2}^{-1}(1,0)$ seen as double pinched torus. The parametrisation with $r$ and $\theta$ in \refintroPinchedTorus\ attains $A$ for $r=0$ and $B$ for $r=\sqrt{6}$. The sign in front of $r$ tells the two `bulges' of the double pinched torus apart. $\theta$ describes the rotational coordinate around each bulge. The plot is done with {\em Mathematica}.
}
\label{Fig_double_pinched}
\end{figure}

\begin{proposition}
\label{introPinchedTorus}
At $t=\frac{1}{2}$, the system $F_{\frac{1}{2}}$ has precisely two focus-focus fibres, each of which contains precisely two focus-focus points so that each of these two fibres has the shape of a double pinched torus (see Figure \ref{Fig_double_pinched}). Exemplarily, $ F_{\frac{1}{2}}^{-1}\left(1,0\right)$ can be parametrised as
$$
\left\{ \left. \left[\sqrt{2}, \sqrt{6 - r^2}, \sqrt{6 - r^2}, \sqrt{8 - r^2}, 2, \sqrt{2 + r^2}, \pm r i e^{-i\theta}, r e^{i\theta}\right] \in M \ \right| \theta \in [0, 2\pi[, \ r \in \left[0, \sqrt{6} \right]\right\}.
$$
\end{proposition}

\noindent
This statement is reformulated as \refdoublePinchParam\ and proven in Section \ref{section pinched}.

\vspace{3mm}

The successful construction of the 1-parameter family $F_t$ suggests that Le Floch $\&$ Palmer's \cite{leFlochPalmer} method of interpolation with invariant functions may work in more generality.
In particular, it should make the construction of semitoric systems possible with any wished number of focus-focus points and with a {\em globally defined, explicitly given} momentum map. 
In practice, the verification of such examples by hand may be lengthy and awkward since the number of equations defining the symplectic manifold grows.

\vspace{3mm}

Control of the Taylor series invariant or twisting index is unfortunately not (yet) possible with these techniques. Nevertheless, constructing focus-focus fibres containing two focus-focus points seemed easy enough. But since these focus-focus points originated from elliptic-elliptic points underlying vertices of the momentum polytope, one needs to employ additional methods if one wants to obtain fibres containing more than two focus-focus points.

\subsection*{Acknowledgements}
We wish to thank Jaume Alonso, Yohann Le Floch, and Joseph Palmer for useful comments and helpful discussions.
The second author was partially supported by the FWO-EoS project G0H4518N and the UA-BOF project with Antigoon-ID 31722. 

\subsection*{Figures}
The figures are either done with {\em Mathematica} or with {\em Xfig}.


\section{{\bf Fundamental definitions, facts, and conventions}}


\subsection{Symplectic conventions}

Within this paper, we usually consider $\R^{2n}$ equipped with the coordinates $(x_1, y_1, x_2, y_2, \dots, x_n, y_n)$. Sometimes we identify $ (x_k, y_k) = x_k +i y_k = z_k$ for $1 \leq k \leq n$ and consider $\R^{2n} \simeq \C^n$ with coordinates $(z_1, \dots, z_n)$. The standard symplectic form $\om_{st}$ on $\R^{2n}$ is thus represented by a $(2n \times 2n)$-matrix having submatrices 
$
\left(
\begin{smallmatrix}
0 & -1 \\ 1 & 0 
\end{smallmatrix}
\right)
$
along the diagonal and zeros otherwise.

Given a symplectic manifold $(M, \om)$ and a smooth function $f: M \to \R$, we define the Hamiltonian vector field $\mcX^f$ of $f$ via $\om(\mcX^f, \cdot ) =  df$. When we consider $\mcX^f$ at a point $p \in M$, we usually write $\mcX^f(p)$ and, in case of the symplectic form, $\om_p$. 

The Poisson bracket of two functions $f, g: M \to \R$  induced by $\om$ is given by 
$$\{ f, g\} := - \om( \mcX^f, \mcX^g) = -df (\mcX^g) = dg(\mcX^f).$$


\subsection{Completely integrable systems}

A \textit{$2n$-dimensional completely integrable system} is a triple $(M, \omega, f)$ where $(M, \omega)$ is a $2n$-dimensional symplectic manifold and $f=(f_1, \dots, f_n) \in C^{\infty}(M, \mathbb{R}^n)$ is a map that satisfies
\begin{itemize}
\item $\lbrace f_j, f_k \rbrace = 0$ for all $ 1 \leq j, k \leq  n. $
\item $\mathcal{X}^{f_i}(p), \ldots, \mathcal{X}^{f_n}(p)$ are linearly independent for almost all $p \in M$.
\end{itemize}
The function $f$ is called the \textit{momentum map} of the integrable system $(M, \omega, f)$.

Let $(M, \om, f)$ be a $2n$-dimensional completely integrable system. A point $p \in M$ is called \textit{regular} for $f$ if $\mathcal{X}^{f_1}(p), \ldots, \mathcal{X}^{f_n}(p)$ are linearly independent. Otherwise $p$ is said to be \textit{singular}. The \textit{rank} of the singular point $p$ is the rank of $(\mathcal{X}^{f_1}(p), \ldots, \mathcal{X}^{f_n}(p))$ or, equivalently, the rank of the Jacobian $df(p)$. A singular point of rank zero is usually called a \textit{fixed point}. A value $r \in \R^n$ is called {\it regular} if the whole {\it fibre} or {\it level set} $f^{-1}(r)$ only contains regular points. Otherwise $r \in \R^n$ is called {\it singular}. A fibre is said to be {\em regular} if it contains only regular points and otherwise {\em singular}.


\subsection{Nondegeneracy and local normal form in $2n$ dimensions}

A singular point $p \in M$ of rank zero of a completely integrable system $(M, \om, f=(f_1, \dots, f_n))$ is \textit{nondegenerate} if the Hessians $d^2{f_1}(p), \ldots, d^2f_n(p)$ span a Cartan subalgebra of the Lie algebra of quadratic forms on $T_pM$. We refer to Bolsinov $\&$ Fomenko~\cite{bolsinov-fomenko} for the precise definition of nondegenerate points of higher rank. For the case $\dim M=4$, we recall it below.

Williamson~\cite{williamson} classified Cartan subalgebras of $\mfs \mfp (n,\R)$. This in turn yields a classification of the possible subalgebras generated by the Hessians in $T_pM$ seen as $\mfs \mfp(n,\R)$. Vey \cite{vey}, Eliasson~\cite{eliassonThesis, eliassonPaper}, Miranda $\&$ Zung~\cite{mirandaZung}, \vungoc\ $\&$ Wacheux \cite{vungocWacheux}, Chaperon \cite{chaperon}, and others extended Williamson's pointwise classification to the following local classification and even more general versions.

\begin{theorem}[Local normal form]
\label{thmEliasson}
 Let $p\in M$ be a nondegenerate singular point of a $2n$-dimensional completely
 integrable system $(M,\om,f=(f_1, \dots, f_n))$. Then there exist local symplectic coordinates
 $(x,y):=(x_1, \dots, x_n, y_1, \dots, y_n)$ around $p$ such that there exists
 $q=(q_1, \dots, q_n) :  M\to\R^n$ where each $q_j$ is given by one of 
 \begin{enumerate}[nosep]
  \item elliptic: $q_j(x,y) = \frac{1}{2}(x_j^2+y_j^2)$,
  \item hyperbolic: $q_j (x,y)= x_j y_j$,
  \item focus-focus: $\begin{cases}  q_j (x,y)  = x_j y_{j +1}-x_{j+1} y_j,\\ q_{j+1}(x,y)  = x_j y_j +x_{j+1}y_{j+1},\end{cases}$
  \item non-singular: $q_j (x,y)= y_j$,
 \end{enumerate}
 such that $\{f_j, q_k\}=0$ for all $1 \leq j,k \leq n$.
\end{theorem}

To detect the type of a nondegenerate singular point, it is sufficient to check the eigenvalues:

\begin{proposition}[\vungoc {\cite[Chapter 3]{vungocBook}}]
\label{propEigValues}
Let $C$ be a regular element in the Cartan subalgebra generated by the Hessians of the components of
the momentum map (i.e., $C$ has $2n$ distinct eigenvalues) at a fixed point. Then there appear three distinct types of groups of eigenvalues of $C$:
\begin{enumerate}[nosep]
\item \emph{elliptic block}: a pair of imaginary eigenvalues $\pm \mathrm{i}\beta$,
\item \emph{hyperbolic block}: a pair of real eigenvalues $\pm \alpha$,
\item \emph{focus-focus block}: a quadruple of complex eigenvalues $\pm \alpha \pm \mathrm{i}\beta$,
\end{enumerate}
where $\alpha,\beta\in\R^{\neq 0}$.
\end{proposition}


\subsection{Nondegeneracy and local normal form in $4$ dimensions}

In this paper, we mainly work on 4-dimensional symplectic manifolds, i.e., there are only singular points of rank zero or rank one possible. In this situation, nondegeneracy of rank zero points (= fixed points) can be verified as follows.

\begin{lemma}[Bolsinov $\&$ Fomenko~\cite{bolsinov-fomenko}]
\label{nondeg}
 Let $(M, \om, f=(f_1, f_2))$ be a $4$-dimensional completely integrable system having a fixed point $p \in M$.
 Let $\om_p$ be the matrix of the symplectic form with respect to a
 basis of $T_p M$ and let $d^2f_1(p)$ and $d^2f_2(p)$ be the matrices of the Hessians
 of $f_1$ and $f_2$ with respect to the same basis.
 Then the fixed point $p$ is nondegenerate if and only if $d^2f_1(p)$ and
 $d^2f_2(p)$ are linearly independent and there exists a linear
 combination of $\om_p^{-1}d^2 f_1(p)$ and $\om_p^{-1}d^2 f_2(p)$ which 
 has four distinct eigenvalues. 
\end{lemma}

Denote by $\lam_1, \lam_2, \lam_3, \lam_4$ the distinct eigenvalues of a nondegenerate fixed point $p$ as considered in \refnondeg. Then \refthmEliasson\ and \refpropEigValues\ imply that the type of $p$ is one of the following:
\begin{itemize}
\item \textit{elliptic-elliptic} if $\lambda_1, \lambda_2 = \pm i \alpha$ and $\lambda_3, \lambda_4 = \pm i\beta$,
\item \textit{elliptic-hyperbolic} if $\lambda_1, \lambda_2 = \pm i \alpha$ and $\lambda_3, \lambda_4 = \pm \beta$,
\item \textit{hyperbolic-hyperbolic} if $\lambda_1, \lambda_2 = \pm \alpha$ and $\lambda_3, \lambda_4 = \pm \beta$,
\item \textit{focus-focus} if $\lambda_1 = \alpha + i\beta$, $\lambda_2 = \alpha - i \beta$, $\lambda_3 = -\alpha + i\beta$ and $\lambda_4 = -\alpha - i\beta$,
\end{itemize}
where $\alpha, \beta \in \mathbb{R}^{\neq 0}$ and $\alpha \neq \beta$ for the elliptic-elliptic and hyperbolic-hyperbolic cases.

Nondegeneracy of rank one points can be characterised as follows on $4$-dimensional symplectic manifolds $(M, \om)$, see Bolsinov $\&$ Fomenko~\cite[Section 1.8.2]{bolsinov-fomenko}.
Let $p$ be a singular point of rank one of a $4$-dimensional completely integrable system $\bigl(M,\om,f=(f_1, f_2)\bigr)$.
Then there are $\mu,\lam\in\mathbb{R}$ such that $\mu\ d f_1(p) + \lam\ df_2(p) = 0$ 
and $L_p:=\Span\{\mcX^{f_1}(p), \mcX^{f_2}(p)\} \subset T_pM$ is the tangent line in $p$ of the orbit generated by the $\mathbb{R}^2$-action. Denote by $L^\perp_p$ the symplectic orthogonal complement to $L_p$ in $T_pM$.
Note that $L_p\subset L_p^\perp$. The Poisson commutativity $\{f_1, f_2\}=0$ implies that they are invariant under the $\R^2$-action. Therefore $\mu\ d^2f_1(p) + \lam\ d^2f_2(p)$
descends to the quotient $L^\perp_p/L_p$. This allows to define

\begin{definition}[Bolsinov $\&$ Fomenko~\cite{bolsinov-fomenko}]
\label{nondegRankOne}
 A rank one critical point $p$ of a $4$-dimensional completely integrable system $\bigl(M, \om, f=(f_1, f_2)\bigr)$ is \emph{nondegenerate} if $\mu\ d^2f_1(p) + \lam\ d^2f_2(p)$
 is invertible on $L^\perp_p/L_p$.
\end{definition}

The possible types of nondegenerate rank one points on $4$-dimensional manifolds are
\begin{itemize}
\item \textit{elliptic-regular} if the eigenvalues of $\om_p^{-1}(\mu d^2f_1 (p)+ \lam d^2f_2(p))$ on $L^\perp_p \slash L_p$ are of the form $\pm i\alpha, \: \alpha \in \mathbb{R}^{\neq 0}$,
\item \textit{hyperbolic-regular} if the eigenvalues of $\om_p^{-1}(\mu d^2f_1(p) + \lam d^2f_2(p))$ on $L^\perp_p \slash L_p$ are of the form $\pm \alpha, \: \alpha \in \mathbb{R}^{\neq 0}$.
\end{itemize}

If one of the integrals has a periodic flow, the search for rank one points can be done by means of reduced spaces as described in \refredRankOne\ later on.


\subsection{Symplectic reduction}

Let us recall the following important notations and results:

\begin{theorem} [Marsden $\&$ Weinstein \cite{marsdenWeinstein}]
\label{marsdenWeinstein}
Let $G$ be a compact Lie group with Lie algebra $\mfg$ inducing a Hamiltonian action on a symplectic manifold $(M, \omega)$. Let $f: M \to \mfg^*$ be the momentum map of this action. Let $r \in \mfg^* $ be a regular value of $f$ that is fixed by the coadjoint action and denote by $\ka: f^{-1}(r) \hookrightarrow M$ the inclusion. Assume that $G$ acts freely and properly on $f^{-1}(r)$. Then
\begin{itemize}
\item 
$M^{red, r} := f^{-1}(r)/G$ is a manifold of dimension $\dim M - 2 \dim G$.
\item
The quotient map $\tau: f^{-1}(r) \to f^{-1}(r)/G = M^{red, r} $ is a principal $G$-bundle.
\item 
Then there is a unique symplectic structure $\omega^{red, r}$ on $M^{red, r}$ such that $ \tau^{*}\omega^{red, r} = \ka^{*}\omega$.
\end{itemize}
The symplectic manifold $(M^{red, r}, \om^{red, r})$ is called the \emph{symplectic reduction} of $M$ at level $r$ by the action generated by $G$, also known as \emph{symplectic quotient} or \emph{Marsden-Weinstein quotient}.
\end{theorem}

We will use \refmarsdenWeinstein\ in particular in the situation of a $4$-dimensional integrable systems where one of the integrals induces an $\mbS^1$-action:

\begin{corollary}
\label{redHam}
 Let $(M, \om, F=(J, H))$ be a $4$-dimensional completely integrable system and let $J$ have a periodic flow. Let $j \in \R$ be a regular value of $J$ and assume that the $\mbS^1$-action induced by $J$ is free and proper on $J^{-1}(j)$. Then the symplectic reduction of $M$ at level $j$ by the $\mbS^1$-action is given by the Marsden-Weinstein quotient
 $$
 M^{red,j}:= J^{-1}(j) \slash \mbS^1
 $$
 and the Poisson commutativity of $J$ and $H$ assures that $H$ descends to a smooth function $H^{red, j}: M^{red, j} \to \R$, called {\em reduced Hamiltonian of $H$}.
\end{corollary}

The following statement can be found in Le Floch $\&$ Palmer \cite[Lemma 2.6]{leFlochPalmer} who in turn refer to Toth $\&$ Zelditch \cite[Definition 3]{tothZelditch} and Hohloch $\&$ Palmer \cite[Corollary 2.5]{hohlochPalmer}. It means that the reduced Hamiltonian function is particularly useful to study singular points of rank one of the original system.

\begin{lemma}
\label{redRankOne}
Let $(M, \om, F=(J, H))$ be a completely integrable system on a $4$-dimensional symplectic manifold and let $J$ have a periodic flow. Let $j \in \R$ be a regular value of $J$ and assume that the $\mbS^1$-action induced by $J$ is free and proper on $J^{-1}(j)$. Let $p \in J^{-1}(j)$ and set $[p]:=\tau(p) \in M^{red, j}$. Then
\begin{enumerate}
\item 
$p$ is a singular point of rank one of $F = (J,H)$. 

$\Leftrightarrow$ \ $\left[ p \right]$ is a singular point of $H^{red,j}$, i.e., $dH^{red,j}([p]) = 0$.
\item 
$p$ is a nondegenerate singular point of rank one of $F$.

$\Leftrightarrow$ \  $\left[ p \right]$ is a nondegenerate singular point of $H^{red,j}$.

$\Leftrightarrow$ \ $dH^{red,j}([p]) = 0$ and $d^2H^{red,j}$ is invertible at $\left[ p \right]$ in $M^{red, j}$.

\item
$p$ is an elliptic-regular (respectively hyperbolic-regular) point of $F$.

$\Leftrightarrow$ \   $\left[ p \right]$ is an elliptic (respectively hyperbolic) singular point of $H^{red,j}$.
\end{enumerate}
\end{lemma}

So far, we considered symplectic reduction of regular values. If the value is singular one encounters stratified spaces and needs to work with singular reduction, see Sjamaar $\&$ Lerman \cite{sjamaarLerman} or, for an overview, Alonso \cite[pp. 26-28]{alonso}, or do it by hand for simple examples as in \refmjred.


\subsection{Toric systems and Delzant's construction}


\label{subsection Delzant}

Let $(M, \om, f)$ be a $2n$-dimensional completely integrable system. If $f$ is in fact the momentum map of a Hamiltonian $n$-torus action, we call $(M, \om, f)$ a {\em toric manifold}. According to the convexity theorem by Atiyah \cite{atiyah} and Guillemin $\&$ Sternberg \cite{guilleminSternberg}, $f(M)$ is a convex polytope spanned by the images of the fixed points. It is often referred to as {\it momentum polytope}. 
A convex polytope in $\R^n$ is {\em Delzant} if (i) its edges have rational slope, (ii) at each vertex, precisely $n$ edges meet, (iii) at each vertex, the $n$ tangent directions considered as vectors in $\Z^n$ span $\Z^n$.

\begin{theorem}[Delzant \cite{delzant}]
A $2n$-dimensional toric manifold having an effective Hamiltonian torus action is classified up to equivariant isomorphism by its momentum polytope which is in fact Delzant. Conversely, for each Delzant polytope, there exists, up to equivariant symplectomorphism, a toric manifold having the Delzant polytope as momentum polytope. 
\end{theorem}

Crucial for the present paper is the fact that Delzant's proof is constructive, i.e., given a Delzant polytope $\De$, there is an explicit way how to construct a toric manifold $(M, \om, f)$ with $f(M)= \De$. We briefly sketch the construction as outlined in Cannas da Silva's book \cite{cannasDaSilva} since we will make ample use of it in Section \ref{section toric}: 
\begin{enumerate}
 \item 
 Given a Delzant polytope $\De \subset \R^n$ with $k$ edges, write it as intersection of $k$ half\-spaces (whose  boundaries lies on the edges of $\De$) by means of the primitive, integral, outer normal vectors.
 \item
 Define a map $\theta: \R^k \to \R^n$ by sending the standard basis of $\R^n$ to the $k$ outer normal vectors of the edges. This map is also welldefined as map $\theta: \Z^k \to \Z^n$ and descends to $ \theta: \R^k \slash \Z^k = \T^k \to\R^n \slash \Z^n = \T^n $.
 \item
 The kernel $N:= \ker(\theta)$ of $\theta : \T^k \to \T^n$ is isomorphic to an $(k-n)$-torus. Consider the inclusion $N \hookrightarrow \T^k$ and extend it to a map $\ell: \R^{k-n} \to \R^n$. Denote its dual map by $\ell^*: \R^n \to \R^{k-n}$.
 \item
 Let $c \in \Z^k$ be the vector whose entries are given by the minimal distance to the origin of the boundaries of the halfspaces defining $\De$. Consider the standard $k$-torus action on $\C^k$ with momentum map 
 $$
 L:=L_c: \C^k \to \R^k, \qquad  L_c(z_1, \dots, z_k):=- \frac{1}{2} \left(\abs{z_1}^2, \dots, \abs{z_k}^2 \right) + c
 $$
 and define the map
 $$
 \Lti:= \ell^* \circ L: \C^k \to \R^{k-n}.
 $$
 \item
 Note that $0 \in \R^{k-n}$ is a regular value of $\Lti$ and that $N$ acts freely and properly on $\Lti^{-1}(0)$. According to \refmarsdenWeinstein, $M^{red, 0}:= \Lti^{-1}(0) \slash N$ is a symplectic manifold with symplectic form $\om^{red, 0}$ which satisfies $\tau^* \om^{red, 0} = \ka^*\om_{st} $ where $\tau:\Lti^{-1}(0) \to\Lti^{-1}(0) \slash N$ is the quotient map and $\ka: \Lti^{-1}(0) \hookrightarrow (\C^k, \om_{st})$ the inclusion and $\om_{st}$ the standard symplectic form on $\C^k$.
 \item
 Find a right inverse $\si: \R^n \to \R^k$ of the map $\theta: \R^k \to \R^n$ and consider the concatenation
 $$
 \Fti:= \si^* \circ L \circ\ka : \Lti^{-1}(0) \to \R^n.
 $$
 Since $\Fti$ is invariant under the action of $N$ it descends to a map $F: M^{red, 0} \to \R^n$. Then $F: (M^{red, 0}, \om^{red, 0}) \to \R^n$ is the momentum map of a Hamiltonian $\T^n$-action satisfying $F(M^{red, 0}) = \De$.
\end{enumerate}


\subsection{Systems of toric type}

The following notion appeared first in \vungoc\ \cite[Definition 2.1]{vungocPolytope} for systems that are `toric up to a diffeomorphism of the momentum polytope'.

\begin{definition}
A completely integrable system $(M, \om, F)$ is {\em of toric type} if $F: M \to \R^2$ is proper and if there exists an effective, Hamiltonian $\T^2$-action on $M$ whose momentum map is of the form $f \circ F$ where $f: F(M) \to f(F(M))$ is a diffeomorphism.
\end{definition}

Properness of $F$ is automatically satisfied if the underlying manifold $M$ is compact.
Systems of toric type are a special case of so-called {\em weakly toric systems} introduced in Hohloch $\&$ Sabatini $\&$ Sepe $\&$ Symington \cite{hsss}.


\subsection{Semitoric systems and semitoric transition families}

Toric systems only admit elliptic-elliptic or elliptic-regular points. In order to admit more types of singular points while keeping the class of systems as accessible as possible \vungoc\ \cite{vungocTaylorSeries, vungocPolytope} began to focus on the following class of systems:

\begin{definition}
A $4$-dimensional completely integrable system $(M, \omega, F = (J,H))$ is \emph{semitoric} if
\begin{enumerate}[label=\arabic*)]
\item 
$J: M \to \R$ is proper;
\item 
$J$ is the momentum map of an effective Hamiltonian $\mathbb{S}^1$-action (i.e. the flow of $\mathcal{X}^J$ is periodic, more precisely, $2\pi$-periodic in our convention);
\item 
all singular points of $F$ are nondegenerate and have no hyperbolic components.
\end{enumerate}
\end{definition}

A semitoric system is called \textit{simple} if each fibre of $J$ contains at most one focus-focus point. Simple semitoric systems where studied and classified by Pelayo and V\~u Ng\d{o}c in \cite{pelayoVungocInv, pelayoVungocActa} by means of five invariants. 

The semitoric systems constructed later in Section \ref{section semitoric} have always two focus-focus points in a fibre of $J$ whenever the fibre contains focus-focus points, so they are not simple. However, had we started the construction with a Delzant octagon where the vertices of the horizontal edges do not lie on the same vertical line, then we would expect to get a simple semitoric system.

We will study the transition of elliptic-elliptic points to focus-focus points and back. Note that, at the moment of transition, the singular points have to be degenerate, i.e., the system is at this very moment {\em not} semitoric. The following definition is adapted from Le Floch $\&$ Palmer \cite{leFlochPalmer} who study transitions in parameter depending semitoric families over Hirzebruch surfaces.

\begin{definition} 
Let $(M, \om)$ be a $4$-dimensional symplectic manifold. Let $J: (M, \om) \to \R$ be smooth with periodic Hamiltonian flow and let $H: [0, 1] \times M \rightarrow \mathbb{R}$ be smooth with $H_t:=H(t, \cdot)$ for $t \in [0,1]$ such that $F_t:=(J, H_t): (M, \om) \to \R^2$ is a completely integrable system.
We call $(M, \om, F_t=(J, H_t))$ a {\em semitoric family with fixed $\mbS^1$-action and $k \in \N_0$ degenerate times $t_1, \dots, t_k$} if $(M, \om, F_t=(J, H_t))$ is semitoric for all $t \in [0,1] \setminus \{t_1, \dots, t_n\}$. 
\end{definition}

Coupled angular momenta are an example for such a family. Hohloch $\&$ Palmer \cite{hohlochPalmer} generalized them to a 2-parameter family admitting 2 focus-focus points while leaving the $S^1$-action unchanged. Le Floch $\&$ Palmer \cite{leFlochPalmer} studied semitoric families with fixed $\mbS^1$-action on Hirzebruch surfaces.

Here we are in particular interested in fixed points changing from elliptic-elliptic to focus-focus and back (so-called {\em Hamiltonian-Hopf bifurcations}). The following definition is inspired by Le Floch $\&$ Palmer \cite{leFlochPalmer}.

\begin{definition}
A \emph{semitoric family with transition points} $p_1, \ldots, p_n \in M$ and \emph{transition times} $0< t^- < t^+<0 $ is a semitoric family with fixed $\mbS^1$-action $\left(M, \omega, F_t \right)_{0 \leq t \leq 1}$ and degenerate times $t^-$ and $t^+$, such that
\begin{enumerate}[label=\arabic*)]
\item 
for $0< t < t^-$ and $t^+<t<1$, the points $p_1, \ldots, p_n$ are elliptic-elliptic;
\item 
for $t^- < t < t^+$, the points $p_1, \ldots, p_n$ are focus-focus;
\item 
for $t = t^-$ and $t = t^+$, the points $p_1, \ldots, p_n$ are degenerate and there are no degenerate singular points in $M \setminus \{ p_1, \ldots, p_n \}$;
\item 
if $p_i$ is a maximum (resp. minimum) of $H_0\vert_{J^{-1}(J(p_i))}$ then $p_i$ is a minimum (resp. maximum) of $H_1\vert_{J^{-1}(J(p_i))}$.
\end{enumerate}
Briefly, we just speak of a \emph{semitoric transition family} in such a situation.
\end{definition}

The following statements describe what happens at the transition times.

\begin{proposition}[Hohloch $\&$ Palmer \cite{hohlochPalmer}]
\label{degPoint}
Let $\left(M, \omega, F_t \right)_{0 \leq t \leq 1}$ be a semitoric transition family with fixed $\mbS^1$-action. Let $p \in M$ and $t_0 \in \: ]0,1[$ be such that $p$ is a fixed point for all $t$ in a neighbourhood of $t_0$ where $p$ is of focus-focus type for $t > t_0$ and of elliptic-elliptic type for $t < t_0$. Then $p$ is a degenerate fixed point for $t = t_0$.
\end{proposition}

Note that singular points might change from rank zero to rank one points:

\begin{proposition}[Le Floch $\&$ Palmer \cite{leFlochPalmer}] 
\label{zeroToOne}
Let $(M, \omega, (J,H_t))$ be a semitoric family with fixed $\mathbb{S}^1$-action and suppose that $p \in M$ is a rank zero singular point for some $t_0 \in [0,1]$. Then $p$ is a singular point for all $t \in [0,1]$ but not necessarily of rank zero. If $p$ does not belong to a fixed surface of $J$, then $p$ is a fixed point for all $t \in [0,1]$.
\end{proposition}


\section{{\bf The toric system constructed from the octagon in Figure \ref{fig_octagon}}}

\label{section toric}

Using Delzant's construction sketched in Section \ref{subsection Delzant}, we will construct in this section the toric system associated to the polytope $\De$ in Figure \ref{fig_octagon}.


\subsection{The symplectic manifold associated to the octagon $\De$} 

\label{subsection_Delzant_construction}

Consider the octagon $\De$ in Figure \ref{fig_octagon}. A brief look at the coordinates of the vertices 
$$
(0,2) \quad (0,1) \quad (1, 0) \quad (2, 0) \quad (3, 1) \quad (3, 2) \quad (2, 3) \quad (1, 3)
$$
of $\De$ yields for the edges, listed counterclockwise, as spanning vectors
$$
(-1,0) \quad (-1,-1) \quad  (0,-1) \quad (1,-1) \quad (1,0) \quad (1,1) \quad (0,1) \quad (-1,1) 
$$
which in turn yields that $\De$ is a Delzant polygon. Thus we may use Delzant's construction (see Section \ref{subsection Delzant}) to obtain a symplectic manifold $(M, \omega)$ and momentum map $F:(M, \omega) \to \R^2 $ with $F(M)= \De$. 
To this aim, write $\De$ as intersection of eight halfplanes:
\begin{equation}
 \label{eightHalfPlanes}
 \De := \quad 
\left\lbrace (x,y) \in \mathbb{R}^2 \  \middle| \
\begin{aligned}
 \langle (x,y), (-1,0) \rangle \: & \leq  \; 0, \qquad & \langle (x,y), (1,0) \rangle \: & \leq   \; 3 \\
 \langle (x,y), (-1,-1) \rangle \: & \leq  \; -1, \qquad & \langle (x,y), (1,1) \rangle \: & \leq  \; 5 \\
 \langle (x,y), (0,-1) \rangle \: & \leq   \; 0, \qquad & \langle (x,y), (0,1) \rangle \: & \leq   \; 3 \\
 \langle (x,y), (1,-1) \rangle \: & \leq   \; 2, \qquad & \langle (x,y), (-1,1) \rangle \: & \leq   \; 2
\end{aligned}
 \ \right\rbrace
\end{equation} 
Note that all appearing outer-pointing normal vectors of the half planes are primitive.
Denote by $e_1, \ldots, e_8$ the standard basis of $\mathbb{R}^8$. Consider the linear map $\theta: \mathbb{R}^8 \to \mathbb{R}^2$ that assigns to each basis vector an outer normal vector as follows (starting at the left edge and going counter-clockwise):
\begin{equation}
\label{theta}
\begin{aligned}
 & e_1 \mapsto (-1,0) &\qquad&   e_3 \mapsto (0,-1) & \qquad  & e_5 \mapsto (1,0) & \qquad &  e_7 \mapsto (0,1) \\
 & e_2 \mapsto (-1,-1) & \qquad &   e_4 \mapsto (1,-1) & \qquad &   e_6 \mapsto (1,1) & \qquad & e_8 \mapsto (-1,1) \\
\end{aligned}
\end{equation}
The map $\theta : \mathbb{R}^8 \to \mathbb{R}^2$ is surjective and induces a surjective linear map $\theta : \mathbb{Z}^8 \to \mathbb{Z}^2$. Therefore passing to the quotients $\R^8 \slash \Z^8 =:\mathbb{T}^8$ and $\R^2 \slash \Z^2=: \T^2$ yields a welldefined quotient map
$$ \theta: \mathbb{T}^8 \to \mathbb{T}^2 $$
that is linear.
Its kernel is given by
\begin{align*}
N := \ker(\theta) &= \lbrace (x_1, \ldots, x_8) \in \mathbb{T}^8 \mid x_1\cdot(-1,0) + \ldots + x_8 \cdot (-1,1) = (0,0) \ \rbrace \\
&= \left\lbrace (x_1, \ldots, x_8) \in \mathbb{T}^8 \quad \middle| \ 
\begin{aligned}
x_5 = x_1 + x_2 -x_4 -x_6 + x_8 \ \\
x_7  = x_2 + x_3 +x_4 -x_6 -x_8 \
\end{aligned}
\right\rbrace \ \simeq\ \mathbb{T}^6
\end{align*}
which yields an inclusion $ N \hookrightarrow \mathbb{T}^8$. By renaming two variables, we obtain the (injective) linear map 
\begin{gather*}
\ell: \mathbb{R}^6 \to \mathbb{R}^8, \\
 (x_1, \ldots, x_6) \mapsto (x_1, \; x_2, \; x_3, \; x_4, \; x_1 + x_2 - x_4 - x_5 + x_6, \; x_5, \; x_2+x_3+x_4-x_5-x_6, \; x_6).
\end{gather*}
To simplify notation, identify in the following the dual vector space $(\R^m)^*$ with $\R^m$ for all $m \in \N_0$. Recall that the dual map $f^*: \R^m \to \R^n$ of a linear map $f: \R^n \to \R^m$ is defined via the standard dual pairing
$$ < y, f(x) > \ \ = \ \ < f^*(y), x > \quad \forall x \in \mathbb{R}^n,\ \forall\ y \in \mathbb{R}^m. $$
Thus, if $f(x)= C x$ for an $(m \x n)$-matrix $C$ then $f^*(y) = C^Ty$.
Therefore, we find
$$
\ell^*: \mathbb{R}^8 \to \mathbb{R}^6, \qquad (y_1, \ldots, y_8) \mapsto C^T  
\begin{pmatrix}
y_1 \\ \vdots \\ y_8
\end{pmatrix} 
\qquad \mbox{with} \quad 
C = \begin{pmatrix}
1 & 0 & 0 & 0 & 0 & 0 \\
0 & 1 & 0 & 0 & 0 & 0 \\
0 & 0 & 1 & 0 & 0 & 0 \\
0 & 0 & 0 & 1 & 0 & 0 \\
1 & 1 & 0 & -1 & -1 & 1 \\
0 & 0 & 0 & 0 & 1 & 0 \\
0 & 1 & 1 & 1 & -1 & -1 \\
0 & 0 & 0 & 0 & 0 & 1
\end{pmatrix} 
$$
i.e., explicitly
$$
\ell^*(y) = (y_1+y_5, \; y_2+y_5+y_7, \; y_3+y_7, \; y_4-y_5+y_7, \; -y_5+y_6-y_7,\; y_5-y_7+y_8).
$$
Note that, since $\ell$ is injective, $\ell^*$ is surjective.

The goal is to consider $N$ as an action on $\mathbb{C}^8$ and then to pass to the quotient. As a first step, we note

\begin{remark}
Consider the symplectic manifold $(\mathbb{C}^8,\om_{st}:= -\frac{i}{2} \sum_{k = 1}^8 dz_k \wedge d \zbar_k)$ with the standard $\T^8$-action
$$ \T^8 \x \C^8 \to \C^8, \qquad  (t, z):=(t_1, \dots, t_8, z_1, \dots, z_8) \mapsto (e^{it_1}z_1, \ldots, e^{it_8}z_8) .$$
This action is Hamiltonian and admits the family of momentum maps 
\begin{align*}
L_c: \mathbb{C}^8 \to \mathbb{R}^8, \qquad L_c(z)= -\frac{1}{2} (\vert z_1 \vert ^2, \ldots, \vert z_8 \vert^2) + c
\end{align*}
where $c \in \R^8$.
\end{remark}

Now choose $c= (0,-1,0,2,3,5,3,2)$ which are the distances of the eight half planes in \eqref{eightHalfPlanes} to the origin and set $L:=L_{(0,-1,0,2,3,5,3,2)} : \mathbb{C}^8 \to \mathbb{R}^8$.

Recall $\T^6 \simeq N \subset \T^8$ and use the inclusion $ N \hookrightarrow \mathbb{T}^8$ as motivation to define the action 
\begin{equation} \label{Action of N}
\begin{gathered}
 (\T^6 \simeq N) \x  \mathbb{C}^8 \to \C^8, \\
 (t, z) \mapsto(e^{it_1}z_1, e^{it_2}z_2, e^{it_3}z_3, e^{it_4}z_4, e^{i(t_1+t_2-t_4-t_5+t_6)}z_5, e^{it_5}z_6, e^{i(t_2+t_3+t_4-t_5-t_6)}z_7, e^{it_6}z_8).
\end{gathered}
\end{equation}
The corresponding momentum map is given by
$$ \widetilde{L}:= \ell^* \circ L: \mathbb{C}^8 \overset{L}{\to} \mathbb{R}^8 \overset{\ell^*}{\to} \mathbb{R}^6 $$
which is in coordinates given by
\begin{eqnarray*}
(z_1, \ldots, z_8) &\overset{L}{\mapsto}& -\frac{1}{2} (\vert z_1 \vert ^2, \ldots, \vert z_8 \vert^2) + (0,-1,0,2,3,5,3,2) \\
&\overset{\ell^*}{\mapsto}& -\frac{1}{2} (\vert z_1 \vert ^2 + \vert z_5 \vert^2, \vert z_2 \vert^2 + \vert z_5 \vert^2 + \vert z_7 \vert^2, \vert z_3 \vert^2 + \vert z_7 \vert^2, \vert z_4 \vert^2 - \vert z_5 \vert^2 + \vert z_7 \vert^2, \\
&& -\vert z_5 \vert^2 + \vert z_6 \vert^2 - \vert z_7 \vert^2, \vert z_5 \vert^2 - \vert z_7 \vert^2 + \vert z_8 \vert^2) + (3,5,3,2,-1,2)
\end{eqnarray*}
Note that, for the action of $N$ on $(\mathbb{C}^8, \omega_0)$ with momentum map $\widetilde{L}: \mathbb{C}^8 \to \mathbb{R}^6$, the value $0 \in \mathbb{R}^6$ is regular since $\ell^*$ is surjective. The regular level set $\widetilde{L}^{-1}(0) $ is described by six equations:
\begin{equation} 
\label{manifold eqn}
\widetilde{L}^{-1}(0) 
= \left\lbrace (z_1, \ldots, z_8) \in \mathbb{C}^8 \quad  \middle| \quad 
\begin{aligned}
(A) & \quad && \vert z_1 \vert ^2 + \vert z_5 \vert^2 = 6 \\
(B) &&& \vert z_2 \vert^2 + \vert z_5 \vert^2 + \vert z_7 \vert^2 = 10 \\
(C) &&& \vert z_3 \vert^2 + \vert z_7 \vert^2 = 6 \\
(D) &&& \vert z_4 \vert^2 - \vert z_5 \vert^2 + \vert z_7 \vert^2 = 4 \\
(E) &&& \vert z_5 \vert^2 - \vert z_6 \vert^2 + \vert z_7 \vert^2 = 2 \\
(F) &&& \vert z_5 \vert^2 - \vert z_7 \vert^2 + \vert z_8 \vert^2 = 4
\end{aligned}
\right\rbrace 
\end{equation}
The six equations (A) -- (F) in \eqref{manifold eqn} are referred to as the \textbf{manifold equations} of $M$.
Consider the inclusion $\ka: \widetilde{L}^{-1}(0) \hookrightarrow \mathbb{C}^8$ and note that $N$ acts freely and properly on $\widetilde{L}^{-1}(0)$ (for a proof, see Cannas da Silva \cite{cannasDaSilva}).
Therefore we may apply \refmarsdenWeinstein\ to this situation and conclude that 
$$ M :=M^{red, 0} = \widetilde{L}^{-1}(0) / N $$  
is a $4$-dimensional symplectic manifold, denoting the quotient map by $\tau: \widetilde{L}^{-1}(0) \to M$, On the quotient, we have the symplectic form $\om:= \om^{red, 0} $ which satisfies 
\begin{equation}
\label{omega}
 \tau^* \omega = \ka^* \omega_{st}ṡ
\end{equation}


\subsection{Explicit charts of the manifold $\mathbf M$} 

\label{Parametrisation}

For the constructions later on, we need explicit charts for $M$ and some additional notation.

\begin{notation}
 The notation $\Mod$ $8$ means that integers are counted modulo $8$ and that they are considered as element of the set $\{1, 2, 3, 4 ,5, 6, 7, 8\}$. Note that this set starts with $1$ instead of $0$.
\end{notation}

To define suitable charts for the manifold $M$, we will make use of the following technical observation.

\begin{lemma} 
\label{paramU}
Let $[z]=[z_1, \ldots, z_8] \in M = \widetilde{L}^{-1}(0) / N$ with $z_k = 0$ for a given $1 \leq k \leq 8$. Then $z_j \neq 0$ for $j \notin \{k-1, k+1\} \Mod 8$. This means that maximally two coordinates of $z$ may vanish and, if so, these must be consecutive ones.
\end{lemma}

\begin{proof}
We show the assertion for $k = 1$. The other cases follow similarly. Let $z_1 =0$ and consider the manifold equations (A) -- (F) from \eqref{manifold eqn}. Equation (A) leads to $\vert z_5 \vert^2 = 6 \neq 0$ implying $z_5 \neq 0$. That $z_3$, $z_4$, $z_6$, $z_7$ must be nonzero is proven by contradiction:
\begin{align*}
\mbox{If } z_3 = 0 \quad &\overset{(C)}{\Rightarrow} \quad \vert z_7 \vert^2 = 6 \quad\ \ \overset{(B)}{\Rightarrow} \quad 0 \leq \vert z_2 \vert^2 = -2 \quad \lightning \\
\mbox{If } z_4 = 0 \quad &\overset{(D)}{\Rightarrow} \quad \vert z_7 \vert^2 = 10 \quad \overset{(C)}{\Rightarrow} \quad  0 \leq \vert z_3 \vert^2 = -4 \quad \lightning \\
\mbox{If } z_6 = 0 \quad &\overset{(E)}{\Rightarrow} \quad  0 \leq \vert z_7 \vert^2 = -4 \quad  \lightning \\
\mbox{If } z_7 = 0 \quad &\overset{(F)}{\Rightarrow} \quad  0 \leq \vert z_8 \vert^2 = -2 \quad \lightning 
\end{align*}
Thus, if $z_1 = 0$, only $z_2$ or $z_8$ may vanish, too.
Moreover, since the action of $N$ on $\widetilde{L}^{-1}(0)$ preserves the norm $\vert z_j \vert$ for each entry $1 \leq j   \leq 8$, the action does not affect the property of being zero.
\end{proof}

This motivates us to define, for $1 \leq \nu \leq 8$, the open sets
$$ 
U_\nu := \left\lbrace \ [z_1, \ldots, z_8] \in M \mid z_k \neq 0 \mbox{ for all } k \in \{\nu+2, \ldots, \nu+7\} \Mod 8 \ \right\rbrace 
$$
In other words, $U_\nu$ is the only subset of $M$ where $z_\nu$ and $z_{\nu+1}$ may vanish. This is independent of the chosen representative for the point $[z_1, \ldots, z_8] \in M$. In particular, we have an open covering
$
M = \bigcup_{\nu=1}^8 U_\nu. 
$

Now we will construct explicit charts $ \psi_\nu: U_\nu \to \mathbb{R}^4$ for all $1 \leq \nu \leq 8$. Given $[z] \in U_\nu$, there are at least six variables nonzero among $z_1, \dots, z_8$. Now use the $N \simeq \T^6$-action to pick strictly positive real numbers as representatives for them.
Writing $z_k=x_k+i y_k$ for $1 \leq k \leq 8$, this means $y_k=0 $ and $x_k >0$ for these six representatives. Therefore, for $\nu=1$, we represent $U_1$ by points of the form
$$(z_1, \dots, z_8) = (x_1, y_1, x_2, y_2, x_3, 0, x_4, 0, x_5, 0, x_6, 0, x_7, 0, x_8, 0).$$
Using the manifold equations from \eqref{manifold eqn}, the variables $x_3, \dots, x_8$ can be written in $U_1$ as functions of $x_1$, $y_1$, $x_2$, $y_2$ as follows: abbreviating $\vert z_1 \vert^2 = x_1^2 + y_1^2$ and $\vert z_2 \vert^2 = x_2^2 + y_2^2$, we obtain
\begin{equation}
 \label{sixVarEq}
\begin{aligned}
x_3 &= \sqrt{2 - \vert z_1 \vert^2 + \vert z_2 \vert^2},  &&&   x_5 &= \sqrt{6 - \vert z_1 \vert^2}, &&&  x_7 &= \sqrt{4 + \vert z_1 \vert^2 - \vert z_2 \vert^2}, \\ 
x_4 &= \sqrt{6 - 2\vert z_1 \vert^2 + \vert z_2 \vert^2}, &&& x_6 & = \sqrt{8 - \vert z_2 \vert^2},  &&& x_8 &= \sqrt{2 + 2\vert z_1 \vert^2 - \vert z_2 \vert^2}.
\end{aligned}
\end{equation}
Therefore
$$
\psi_1: U_1 \to \R^4, \quad [x_1, y_1, x_2, y_2, x_3, 0, x_4, 0, x_5, 0, x_6, 0, x_7, 0, x_8, 0] \mapsto (x_1, y_1, x_2, y_2)
$$
is a chart for $U_1$ with inverse
$$
\psi_1^{-1}: \R^4 \to U_1, \quad (x_1, y_1, x_2, y_2) \mapsto [x_1, y_1, x_2, y_2, x_3, 0, x_4, 0, x_5, 0, x_6, 0, x_7, 0, x_8, 0]
$$
where $x_3, \dots, x_8$ in $\psi_1^{-1}$ are determined by \eqref{sixVarEq}. Applying analogous reasoning to the remaining indices, we obtain, for all $1 \leq \nu \leq 8$, charts of the form
\begin{gather*}
 \psi_\nu:  U_\nu \to V_\nu:=\psi_\nu ( U_\nu) \subseteq \mathbb{R}^4, \\
[x_1, 0, x_2, 0, \ldots, x_\nu, y_\nu, x_{\nu+1}, y_{\nu+1}, \ldots, x_8, 0]  \mapsto  (x_\nu, y_\nu, x_{\nu+1}, y_{\nu+1}).
\end{gather*}


\subsection{The symplectic form $\mathbf \om$ on $\mathbf M$ in local coordinates} 


In this subsection, we will compute the symplectic form $\om$ on $M$ constructed in Section \ref{subsection_Delzant_construction} in the local charts $(U_\nu, \psi_\nu)$ defined in Section \ref{Parametrisation}.

Let $z \in \widetilde{L}^{-1}(0)$. Then there is (at least one) index $1 \leq \nu \leq 8$ such that $\tau(z)=[z] \in U_\nu$. Over $U_\nu$, the quotient map $\tau$ has the right inverse
\begin{align*}
 \tau^{-1}_\nu: U_\nu \to \widetilde{L}^{-1}(0) , \quad [z] \mapsto (x_1, 0, x_2, 0, \ldots, x_\nu, y_\nu, x_{\nu+1}, y_{\nu+1}, \ldots, x_8, 0)
\end{align*}
defined analogously to the chart $\psi_\nu$, $\psi_\nu^{-1}$ in Section \ref{Parametrisation}, i.e., we have $\tau \circ \tau^{-1}_\nu=\Id_{U_\nu}$ and thus in turn $\psi_\nu \circ \tau \circ \tau^{-1}_\nu \circ \psi_\nu^{-1}=\Id_{V_\nu}$ where $\psi_\nu(U_\nu)=:V_\nu$. 
Denote by $\omega_\nu:= (\psi_\nu^{-1})^* \om$ the symplectic form in local coordinates on $V_\nu \subseteq \R^4$. Then the relation $\tau^* \omega = \ka^* \omega_{st}$ in \eqref{omega} becomes 
$$
 \ka^* \omega_{st} = \tau^* \omega = \tau^* \psi_v^* \om_\nu = (\psi_\nu \circ \tau)^*\om_\nu
$$
and, pulling $  \ka^* \omega_{st}$ back with $\tau_\nu^{-1} \circ \psi_\nu^{-1}$ to $V_\nu$, we obtain
\begin{align*}
(\ka \circ \tau_\nu^{-1} \circ \psi_\nu^{-1})^*\omega_{st} = (\tau_\nu^{-1} \circ \psi_\nu^{-1})^* \ka^* \omega_{st}  = (\tau_\nu^{-1} \circ \psi_\nu^{-1})^*(\psi_\nu \circ \tau)^*\om_\nu = \om_\nu
\end{align*}
Consider the symplectic form $\omega_{st}$ on $\mathbb{C}^8 \simeq \R^{16}$ as represented by the $(16 \times 16)$-matrix with submatrices
$
\left( 
\begin{smallmatrix}
 0 & -1 \\1 & 0
\end{smallmatrix}
\right) 
$
on the diagonal and everywhere else zero.
For $\nu=1$, we now compute the pullback $\om_1$ of the matrix of $\om_{st}$ explicitly. The cases $\nu=2, \dots, 8$ are similar. We find
\begin{gather*}
\ka \circ \tau_1^{-1} \circ \psi_1^{-1}: V_1 \subseteq \R^4 \to \C^8 \simeq \R^{16},  \\
(x_1, y_1, x_2, y_2) \ \mapsto\ (x_1, y_1, x_2, y_2, x_3, 0, x_4, 0, x_5, 0, x_6, 0, x_7, 0,  x_8, 0)
\end{gather*}
and calculate the Jacobian
$$
d(\ka \circ \tau_1^{-1} \circ \psi_1^{-1}) \ = \ 
\begin{pmatrix}
1 & 0 & 0 & 0 \\
0 & 1 & 0 & 0 \\
0 & 0 & 1 & 0 \\
0 & 0 & 0 & 1 \\
\frac{-x_1}{\sqrt{2 - x_1^2 + x_2^2 - y_1^2 + y_2^2}} & \frac{-y_1}{\sqrt{2 - x_1^2 + x_2^2 - y_1^2 + y_2^2}} & \frac{x_2}{\sqrt{2 - x_1^2 + x_2^2 - y_1^2 + y_2^2}} & \frac{y_2}{\sqrt{2 - x_1^2 + x_2^2 - y_1^2 + y_2^2}} \\
0 & 0 & 0 & 0 \\
\frac{- 2 x_1}{\sqrt{6 + x_2^2 - 2 (x_1^2 + y_1^2) + y_2^2}} & \frac{-2 y_1}{\sqrt{6 + x_2^2 - 2 (x_1^2 + y_1^2) + y_2^2}} & \frac{x_2}{\sqrt{6 + x_2^2 - 2 (x_1^2 + y_1^2) + y_2^2}} & \frac{y_2}{\sqrt{6 + x_2^2 - 2 (x_1^2 + y_1^2) + y_2^2}} \\
0 & 0 & 0 & 0 \\
\frac{-x_1}{\sqrt{6 - x_1^2 - y_1^2}} & \frac{-y_1}{\sqrt{6 - x_1^2 - y_1^2}} & 0 & 0 \\
0 & 0 & 0 & 0 \\
0 & 0 & \frac{-x_2}{\sqrt{8 - x_2^2 - y_2^2}} & \frac{-y_2}{\sqrt{8 - x_2^2 - y_2^2}} \\
0 & 0 & 0 & 0 \\
\frac{x_1}{\sqrt{4 + x_1^2 - x_2^2 + y_1^2 - y_2^2}} & \frac{y_1}{\sqrt{4 + x_1^2 - x_2^2 + y_1^2 - y_2^2}} & \frac{-x_2}{\sqrt{4 + x_1^2 - x_2^2 + y_1^2 - y_2^2}} & \frac{-y_2}{\sqrt{4 + x_1^2 - x_2^2 + y_1^2 - y_2^2}} \\
0 & 0 & 0 & 0 \\
\frac{2 x_1}{\sqrt{2 - x_2^2 + 2 (x_1^2 + y_1^2) - y_2^2}} & \frac{2 y_1}{\sqrt{2 - x_2^2 + 2 (x_1^2 + y_1^2) - y_2^2}} & \frac{-x_2}{\sqrt{2 - x_2^2 + 2 (x_1^2 + y_1^2) - y_2^2}} & \frac{-y_2}{\sqrt{2 - x_2^2 + 2 (x_1^2 + y_1^2) - y_2^2}} \\
0 & 0 & 0 & 0
\end{pmatrix}$$
which, considering $\om_{st}$ and $\om_1$ as matrices, yields
\begin{equation}
 \label{omIsStandard}
\omega_1
= 
\bigl(d(\ka \circ \tau_1^{-1} \circ \psi_1^{-1})\bigr)^T \ \om_{st} \ d(\ka \circ \tau_1^{-1} \circ \psi_1^{-1})
=
\begin{pmatrix}
0 & -1 & 0 & 0 \\
1 & 0 & 0 & 0 \\
0 & 0 & 0 & -1 \\
0 & 0 & 1 & 0
\end{pmatrix} 
=
\om_{st}
\end{equation}
Analogous calculations show $\om_\nu= \om_{st}$ on $V_\nu \subseteq \R^4$ for all $1 \leq \nu \leq 8$.


\subsection{Momentum map of the toric system on $(M, \om)$}


Now let us resume Delzant's construction (see Section \ref{subsection Delzant} for an overview of the steps) to obtain a momentum map $F$ on $M$ satisfying $F(M) = \Delta$. First, we have to find a map $\si:\R^2 \to \R^8$ that is a right inverse of the map $\theta: \R^8 \to \R^2$ defined in \eqref{theta}. 

\begin{lemma}
The map $\sigma := \frac{1}{6}\theta^*: \R^2 \to \R^8$ is a right inverse of $\theta$.
\end{lemma}

\begin{proof}
Written in matrix form, $\theta: \R^8 \to \R^2$ is given by 
$$
\theta (x) = 
\begin{pmatrix}
 -1 & -1 & 0 & 1 & 1 & 1 & 0 & -1 \\
 0 & -1 & -1 & -1 & 0 & 1 & 1 & 1 
\end{pmatrix}
x
=:C x.
$$
Therefore the dual map is given by $\theta^*: \R^2 \to \R^8$, $\theta^*(y)=C^T y$. Since $C C^T = 6 \Id_{\R^2}$, we find $\Id_{\R^2} = \frac{1}{6}\theta\circ \theta^* = \theta \circ \left(\frac{1}{6} \theta^*\right)$
such that $\si:= \frac{1}{6} \theta^*$ is a right inverse of $\theta$.
\end{proof}

The next step in Delzant's construction is to compute the concatenation
$$ \widetilde{L}^{-1}(0) \quad \overset{\ka}{\hookrightarrow} \quad \mathbb{C}^8 \quad \overset{L}{\longrightarrow} \quad \mathbb{R}^8 \quad \overset{\si^*}{\longrightarrow} \quad \mathbb{R}^2 $$
where we remark  that $\si^* = \frac{1}{6} \theta$. 

\begin{lemma}
 The map $\Fti:=\frac{1}{6} \theta \circ L \circ \ka: \widetilde{L}^{-1}(0) \to \R^2$ is explicitly given by
 $\Fti (z_1, \ldots, z_8) = \frac{1}{2} (\vert z_1 \vert^2, \vert z_3 \vert^2)$ and is invariant under the action of $N$. Thus $\Fti$ decends to a map $F: M \to \R^2$.
\end{lemma}

\begin{proof}
 Let $z = (z_1, \ldots, z_8) \in \widetilde{L}^{-1}(0)$ and compute
 $$
  L ( \ka(z_1, \ldots, z_8)) = L(z_1, \ldots, z_8) = -\frac{1}{2}(\vert z_1 \vert^2, \ldots, \vert z_8 \vert^2) + (0,-1,0,2,3,5,3,2) 
 $$
 and 
 \begin{gather*}
 \frac{1}{6} \theta( L (z)))  = \frac{1}{6} \theta\left( -\frac{1}{2}\left(\vert z_1 \vert^2, \ldots, \vert z_8 \vert^2 \right) + (0,-1,0,2,3,5,3,2) \right) \\
  =  \frac{1}{12} \left(\abs{ z_1}^2 + \abs{ z_2}^2 - \abs{ z_4 }^2 - \abs{ z_5}^2 - \abs{ z_6}^2 + \abs{ z_8}^2 \;, \; \abs{ z_2}^2 + \abs{ z_3 }^2 + \abs{ z_4}^2 -\abs{z_6}^2 - \abs{ z_7}^2 - \abs{ z_8}^2 \right) 
+ \frac{1}{6}(9,9) 
\end{gather*}
and now we use twice the manifold equations from \eqref{manifold eqn} to obtain
 \begin{align*}
& =  \frac{1}{12} \left(6 - \abs{ z_5} ^2 + 10 - \abs{ z_5}^2 - \abs{ z_7}^2 - 4 - \abs{ z_5}^2 + \abs{ z_7 }^2 - \abs{ z_5 }^2 + 2 - \abs{ z_5 }^2 - \abs{ z_7}^2 + 4 - \abs{ z_5}^2 + \abs{ z_7}^2 + 18, \right. \\
& \left. \ \ \qquad 10 - \abs{ z_5}^2 - \abs{ z_7}^2 + 6 - \abs{ z_7}^2 + 4 + \abs{ z_5}^2 -\abs{ z_7 }^2 + 2 - \abs {z_5 }^2 - \abs{ z_7}^2 - \abs{ z_7}^2 - 4 + \abs{ z_5}^2 - \abs{ z_7}^2 + 18\right) \\
&= \frac{1}{12} \; \left(-6 \: \abs{ z_5}^2 + 36, \; -6 \: \abs{ z_7}^2 + 36 \right) \\
&= \frac{1}{12} \; \left(-6 \: \left(6-\abs{ z_1}^2 \right) + 36, \: -6 \: \left(6- \abs{ z_3}^2 \right) + 36 \right) \\
&= \frac{1}{2} \left(\abs{ z_1}^2, \; \abs{ z_3}^2 \right).
\end{align*}
Since the action of $N$ given in \eqref{Action of N} does not change the absolute value of $z_1$ and $z_3$, the map $F(z_1, \ldots, z_8) = \frac{1}{2} \left(\abs{ z_1}^2, \; \abs{ z_3}^2 \right)$ is invariant under the action of $N$ and passes thus to the quotient $M=\widetilde{L}^{-1}(0)/N$. 
\end{proof}

According to Delzant's construction, the map $F: (M, \om) \to \R^2$ is the desired toric momentum map having the original polytope as image of the momentum map:

\begin{theorem} 
\label{thOctagon}
Let $\De$ be the octagon displayed in Figure \ref{fig_octagon}.
Then $F = (J,H): (M, \om) \to  \mathbb{R}^2$ with
\begin{equation*} 
J([z_1, \ldots, z_8]) = \frac{1}{2}\vert z_1 \vert^2, \qquad  H([z_1, \ldots, z_8]) = \frac{1}{2}\vert z_3 \vert^2
\end{equation*}
is a momentum map of an effective Hamiltonian 2-torus action and satisfies $F(M)= \De$. The $\mbS^1$-action of $J$ is given by 
\begin{align*}
 \mbS^1 \times M \to M, \qquad (t,[z_1, \ldots, z_8]) \mapsto  [e^{-it}z_1, z_2, \ldots, z_8]
\end{align*}
and the $\mbS^1$-action of $H$ by
\begin{align*}
 \mbS^1 \times M \to M, \qquad (t,[z_1, \ldots, z_8]) \mapsto  [z_1, z_2, e^{-it}z_3, z_4, \ldots, z_8].
\end{align*}
\end{theorem}

\begin{proof}
Constructed via Delzant's construction, $F$ is the momentum map of an effective Hamiltonian 2-torus action. Note that $\{J,H\}=0$ is actually additionally proven, explicitly in local coordinates, in the proof of \refHtFt\ when showing Poisson commutativity of the integrals of the (future) semitoric system.

Let us now verify $F(M)= \De$. Recall the manifold equations (A) -- (F) from \eqref{manifold eqn} and note that equations (A) and (C) imply
$$ 0 \leq J \leq 3 \quad \mbox{and} \quad 0 \leq H \leq 3\quad \mbox{and} \quad \abs{ z_5}^2 = 6 - 2J  $$
which we will use in the following. We conclude
\begin{align*}
& (B) \ \Rightarrow\ \vert z_2 \vert^2 + \vert z_7 \vert^2 = 4+2J  
\ \Rightarrow \  \vert z_7 \vert^2 \leq 4+2J 
\ \overset{(C)}{\Rightarrow} \  \vert z_3 \vert^2 \geq 2-2J  
\ \Rightarrow\ H \geq 1-J ,
\\
& (F) \ \Rightarrow\ \vert z_7 \vert^2 = \vert z_8 \vert^2 + 2 - 2J \geq 2-2J 
\ \overset{(C)}{\Rightarrow} \ \vert z_3 \vert^2 \leq 4+2J 
\ \Rightarrow\ H \leq 2+J ,
\\
& (D) \ \Rightarrow\ \vert z_7 \vert^2 = -\vert z_4 \vert^2 + 10 - 2J \leq 10-2J 
 \ \overset{(C)}{\Rightarrow} \ \vert z_3 \vert^2 \geq -4+2J 
 \ \Rightarrow\ H \geq -2+J ,
\\
& (E) \ \Rightarrow\ \vert z_7 \vert^2 = \vert z_6 \vert^2 -4 + 2J \geq -4+2J 
\ \overset{(C)}{\Rightarrow} \ \vert z_3 \vert^2 \leq 10-2J 
\ \Rightarrow\ H \leq 5-J .
\end{align*}
The equations $0 \leq J \leq 3 $ and $ 0 \leq H \leq 3$ describe the square $[0,3] \times [0,3]$ and the equations 
$$H \geq 1-J, \quad H \leq 2+J, \quad  H \geq -2+J,\quad  H \leq 5-J$$ 
describe the edges after chopping off corners of the square $[0,3] \times [0,3]$ in Figure \ref{fig_octagon}. Thus, indeed, $(J, H)(M) = F(M) =  \De$.
\end{proof}

The Atiyah-Guillemin-Sternberg convexity theorem tell us that the {\em images} of the elliptic-elliptic fixed points of $F=(J, H): (M, \om) \to \R^2$ are precisely the vertices of the octagon $\De$, but it does not say {\em where} in the four dimensional manifold $(M, \om)$ the {\em preimages} of these vertices lie.

\begin{figure}[h]
\centering
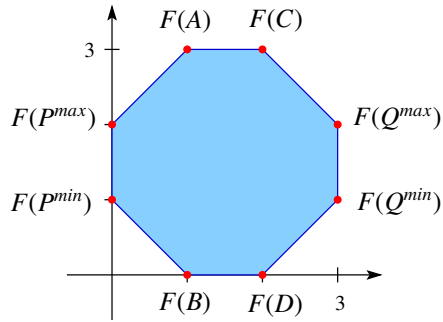
\caption{The momentum polytope $F(M)=\De$ with the images of the fixed points.
}
\label{fig_octagon_corners}
\end{figure}

\begin{proposition}
\label{coordEEPoints}
$F = (J,H)$ has precisely the following eight fixed points:
\begin{align*}
P^{min}  &:=  [0, \ 0,\ \sqrt{2},\ \sqrt{6}, \ \sqrt{6}, \ 2\sqrt{2},\ 2, \ \sqrt{2}] \qquad \mbox{is the fixed point with } z_1 = z_2 = 0, \\
 B \quad & :=  [\sqrt{2},\ 0,\ 0, \ \sqrt{2}, \ 2, \ 2\sqrt{2}, \ \sqrt{6}, \  \sqrt{6}]  \qquad \mbox{is the fixed point with } z_2 = z_3 = 0 , \\
 D \quad  & := [2, \ \sqrt{2}, \ 0,\  0, \ \sqrt{2}, \ \sqrt{6}, \ \sqrt{6}, \ 2\sqrt{2}] \qquad \mbox{is the fixed point with }  z_3 = z_4 = 0 , \\
Q^{min} & := [\sqrt{6}, \ \sqrt{6},\ \sqrt{2},\ 0, \ 0,\ \sqrt{2}, \ 2,  \ 2\sqrt{2}]\qquad \mbox{is the fixed point with } z_4 = z_5 = 0, \\
 Q^{max} & := [\sqrt{6}, \ 2\sqrt{2}, \ 2, \ \sqrt{2}, \ 0, \ 0, \ \sqrt{2}, \ \sqrt{6}] \qquad \mbox{is the fixed point with } z_5 = z_6 = 0 , \\
 C \quad \ & := [2, \ 2\sqrt{2}, \ \sqrt{6}, \ \sqrt{6}, \ \sqrt{2}, \ 0, \ 0, \ \sqrt{2}] \qquad \mbox{is the fixed point with } z_6 = z_7 = 0 , \\
 A \quad \ & := [\sqrt{2}, \ \sqrt{6}, \ \sqrt{6}, \ 2\sqrt{2}, \ 2, \ \sqrt{2}, \ 0, \ 0]\qquad \mbox{is the fixed point with } z_7 = z_8 = 0 , \\
 P^{max} & := [0, \ \sqrt{2}, \ 2, \ 2\sqrt{2}, \ \sqrt{6}, \ \sqrt{6}, \ \sqrt{2}, \ 0]\qquad \mbox{is the fixed point with } z_8 = z_1 = 0 .
\end{align*}
They are mapped under $F$ to the vertices of the octagon as displayed in Figure \ref{fig_octagon_corners}.
\end{proposition}

\begin{proof}
We first look for the fixed points of $J$ and $H$ separately since fixed points of $F=(J,H)$ are precisely the points that are fixed points for both $J$ and $H$.

{\it Fixed points of $J$:}
The point $[z]=[z_1, \dots, z_8]\in M$ will be fixed by the action of $J$ if the following holds true:

{\it Case 1:} 
$z_1 = 0$ in the first coordinate yields a fixed point of the $J$-action.

{\it Case 2:} 
$z_5 = 0$ also gives rise to a fixed point, since we can take $(t_1, \ldots, t_6) = (-t, 0, \ldots, 0)$ in the $N$-action, which only gives another representation for the same point in $M$. This means we can actually see the $J$-action as a special case of the $N$-action.

{\it Case 3:} 
Seeing the $J$-action as an $N$-action with $(t_1, \ldots, t_6) = (-t, 0, \ldots, 0)$ means that $t$ rotates both $z_1$ and $z_5$. There are several possibilities to compensate the rotation of $z_5$. Let us for example set the parameter $t_2 = t$ (so that $z_5$ is kept invariant). Then $t_2$ also affects $z_2$ and $z_7$, so a fixed point should satisfy $z_2 = 0$ and either $z_7 = 0$ or one of the $z_3, z_4, z_6, z_8$ is zero, where we choose a third parameter (for example $t_3 = -t$) in the $N$-action. We know from \refparamU\ that only two subsequent entries can be zero, so the only remaining possibility is $ z_2 = 0 $ and $ z_3 = 0$.

{\it Case 4:} The same idea works when we compensate the rotation of $z_5$ by choosing another parameter in the $N$-action: picking $t_4 = -t $ leads to $ z_4 = 0$ and $ z_3 = 0$. Moreover, setting $ t_5 = -t$ gives rise to $ z_6 = 0$ and $ z_7 = 0$. Last, considering $ t_6 = t$ implies $ z_8 = 0$ and $z_7 = 0$.

{\it Fixed points of $H$:}
Similar to the previous case, a point $[z]=[z_1, \dots, z_8]$ will be fixed by the $H$-action in the following cases:

{\it Case 1:} $z_3 = 0$.

{\it Case 2:} $z_7 = 0$.

{\it Case 3:} $z_2 = 0 \mbox{ and } z_1 = 0; \ \ z_4 = 0 \mbox{ and }  z_5 = 0; \ \ z_6 = 0 \mbox{ and } z_5 = 0; \ \ z_8 = 0 \mbox{ and } z_1 = 0.$

{\it Conclusion: } When considering these conditions together, we see that the fixed points are precisely those points with two subsequent entries equal to zero. If we consider such a fixed point with $z_\nu = 0$ and $z_{\nu+1} = 0$, then this point lies in $U_\nu \subseteq M$ and can be parametrised by $(x_\nu, y_\nu, x_{\nu+1}, y_{\nu + 1}) = (0,0,0,0)$. All other entries can be written as positive real numbers and are found by means of the manifold equations. 
\end{proof}


\section{{\bf A semitoric family with four focus-focus points and two double pinched tori}}
\label{section semitoric}

Within this section, let $(M, \om, F=(J, H))$ be the toric manifold constructed in Section \ref{section toric} that satisfies $F(M)= \De$ where $\De $ is the octagon from Figure \ref{fig_octagon}.

The aim of this section is to replace the second integral $H$ by a parameter depending family $H_t$ such that $(M, \omega, F_t = (J, H_t))$ is a semitoric family with fixed $\mbS^1$-action having four transition points $A,B, C, D \in M$ which map under $F_0=(J,H_0)$ to the points $(1,3), \; (1,0), \; (2,3), \; (2,0)$ as sketched in Figure \ref{fig_octagon_corners}.


\subsection{Geometric interpretation of $M^{red, j}$}

Following Karshon \cite{karshon}, we call a value $j$ of $J$ \emph{extremal} if $j$ is a global minimum or maximum of $J$, and \emph{interior} otherwise. A fixed point of $J$ in the fibre $J^{-1}(j)$ is {\em extremal} if $j$ is extremal. Analogously we define {\em interior} fixed points. The following statement is a compilation of Lemma 2.12 in Le Floch $\&$ Palmer \cite{leFlochPalmer} and Proposition 3.4 in Hohloch $\&$ Sabatini $\&$ Sepe \cite{hss}.

\begin{lemma}
\label{mjred}
Let $(M, \omega, F=(J,H))$ be a semitoric system. 
\begin{enumerate}[label=\arabic*)]
 \item 
 Let $j\in J(M)$ be an interior value. If $j$ is regular for $J$, then $M^{red,j}$ is diffeomorphic to a 2-sphere. If $j$ is singular for $J$ then $M^{red,j}$ is homeomorphic to a 2-sphere (but not diffeomorphic). 
 \item
 Let $j\in J(M)$ be an extremal value of $J$.
 If $j$ corresponds to a vertical edge of the momentum polytope of $F=(J,H)$ then $M^{red,j}$ is diffeomorphic to a 2-sphere. Otherwise $M^{red,j}$ is a point.
\end{enumerate}
\end{lemma}

A look at Figure \ref{fig_octagon} shows that $j \in \{0,3\}$ are the extremal values of $J$ and $j \in \{ 1, 2\}$ singular interior values. All $j \in [0,3] \setminus \{0, 1, 2, 3\}$ are regular interior values. We conclude

\begin{corollary}
Let $(M, \om, F=(J, H))$ be the toric system from \refthOctagon. Then
$$
M^{red, j} \stackrel{diffeo}{\simeq} \mbS^2  \mbox{ for } j \in [0,3] \setminus \{1,2\} \quad \mbox{and} \quad  M^{red, j} \stackrel{homeo}{\simeq} \mbS^2 \mbox{ for } j \in \{1, 2\}.
$$
\end{corollary}

The situation is displayed in Figure \ref{Fig_Mj_red_reeks}. Let us explain the geometric intuition of the shape of the reduced spaces and positions of the singular points of $H^{red, j}$. Later in Section \ref{subsection param}, we will give a parametrisation of $M^{red, j}$.

Recall that, according to \refredRankOne, nondegenerate elliptic-regular points of $F=(J, H)$ on the $4$-dimensional manifold $(M, \om)$ that are regular for $J$ correspond to nondegenerate elliptic points of $H^{red,j}$ on the $2$-dimensional space $M^{red, j}$. Hence there are, for $j \in\ ]0,3[ \ \setminus \{1, 2\}$, precisely two elliptic points of $H^{red, j}$ on $M^{red, j}$. They are located at the `north and south pole' of $M^{red, j} \stackrel{diffeo}{\simeq} \mbS^2$ since, as a maximum and minimum of $H^{red, j}$, the north and south pole are elliptic points of $H^{red, j}$. But there are not more than two elliptic points possible for $H^{red, j}$ since there are only two elliptic-regular points in $F^{-1}(j)$.

Now consider the interior singular values $j \in\{1, 2\}$. The proof of Lemma 2.12 in Le Floch $\&$ Palmer \cite{leFlochPalmer} (see also the overview in Alonso \cite[pp. 26-28]{alonso} on singular reduction) shows that an elliptic-elliptic or focus-focus point on $M$ causes $M^{red, j}$ to have a singularity (a `non-smooth peak') such that $M^{red, j}$ is only homeomorphic to a 2-sphere but not diffeomorphic. Since there are two elliptic-elliptic points in the level set of $j \in \{1, 2\}$ of $J$, the reduced space $M^{red, j}$ has two `non-smooth peaks' as displayed in Figure \ref{Fig_Mj_red_reeks}.

\begin{figure}[h]
\raggedright 
\begin{subfigure}{0.33\textwidth}
\centering
\includegraphics[width=.9\linewidth]{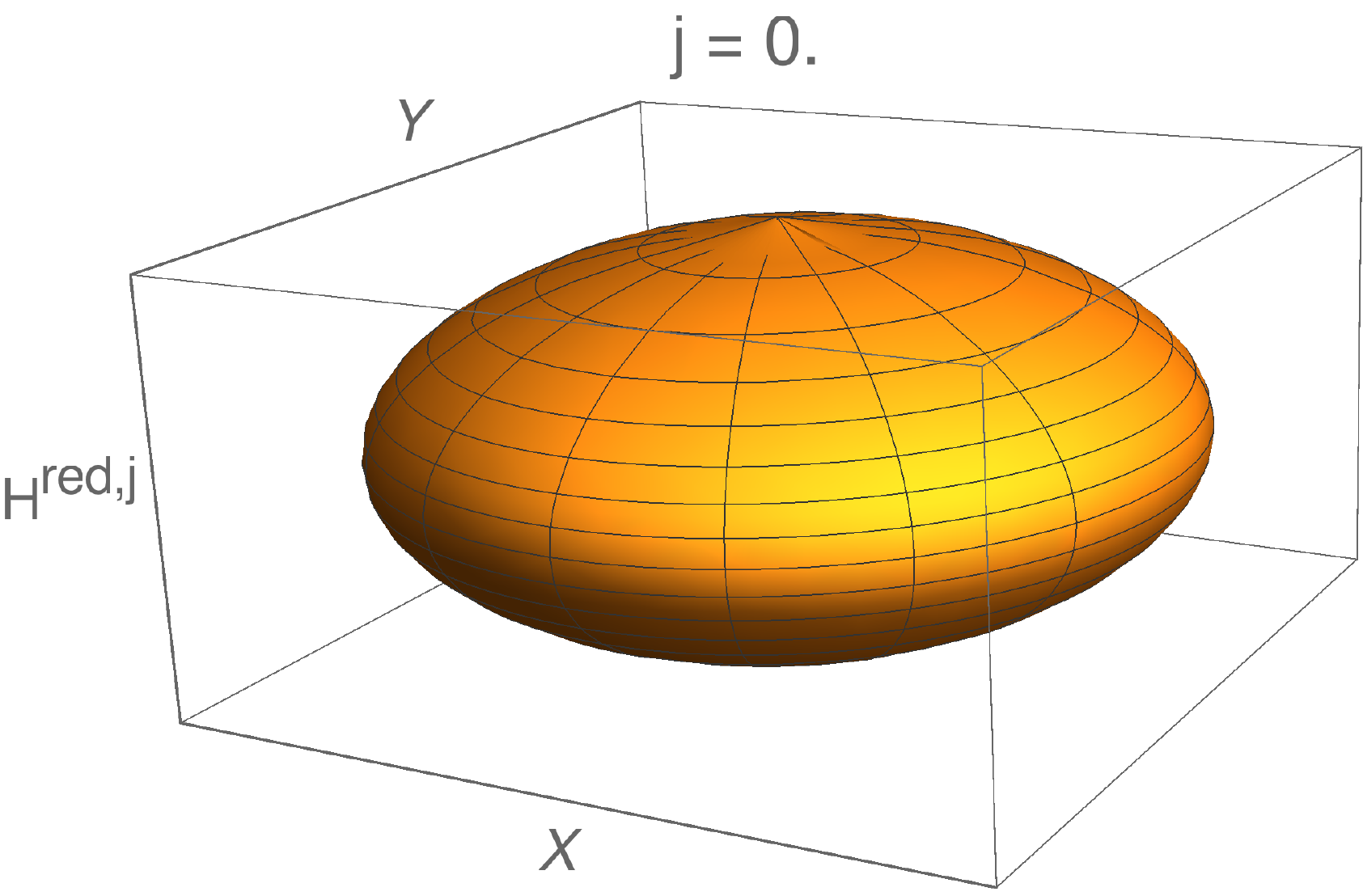}
\end{subfigure}%
\begin{subfigure}{0.33\textwidth}
\centering
\includegraphics[width=.9\linewidth]{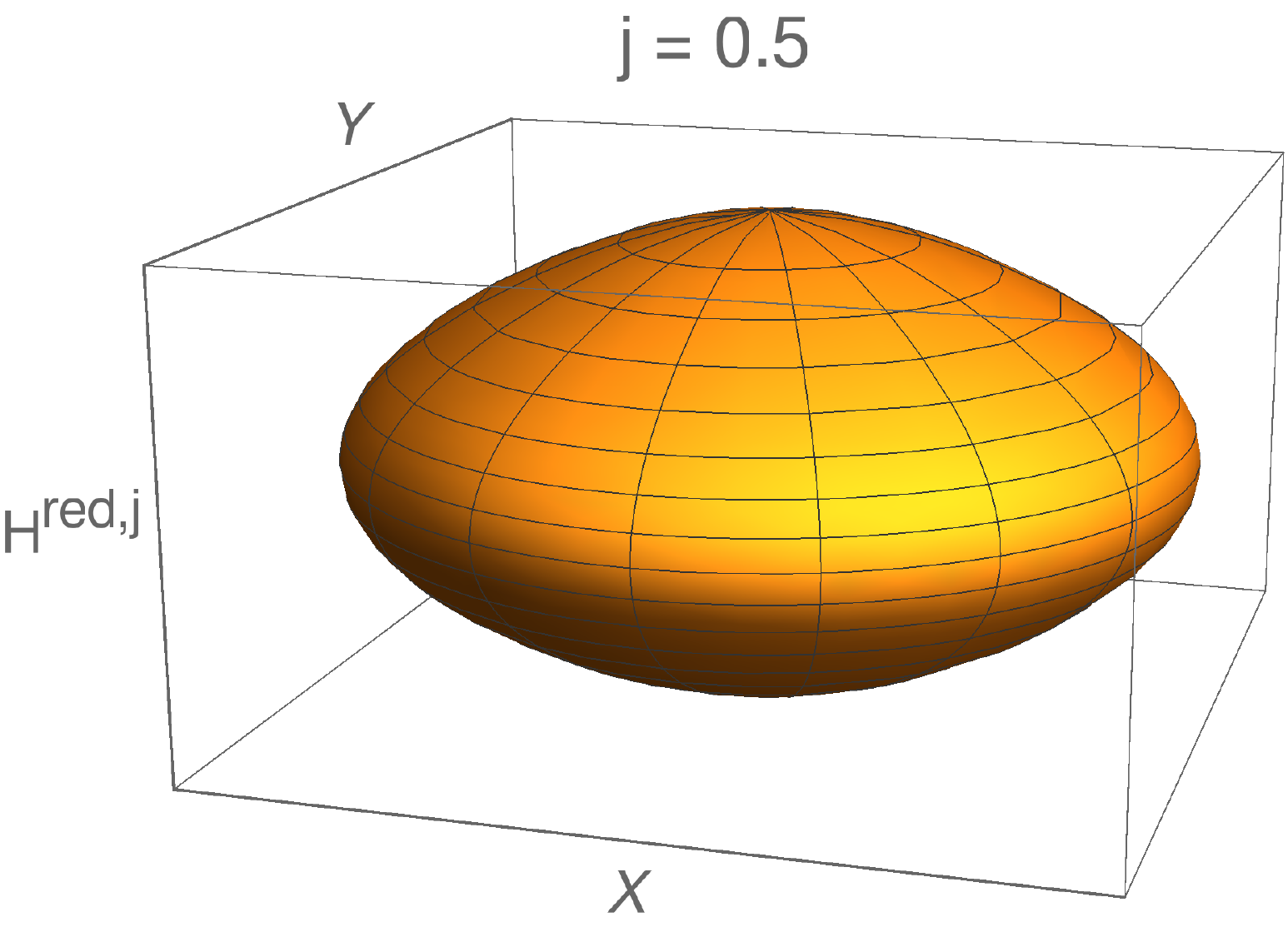}
\end{subfigure}

\vspace{5mm}
\begin{subfigure}{0.33\textwidth}
\centering
\includegraphics[width=.9\linewidth]{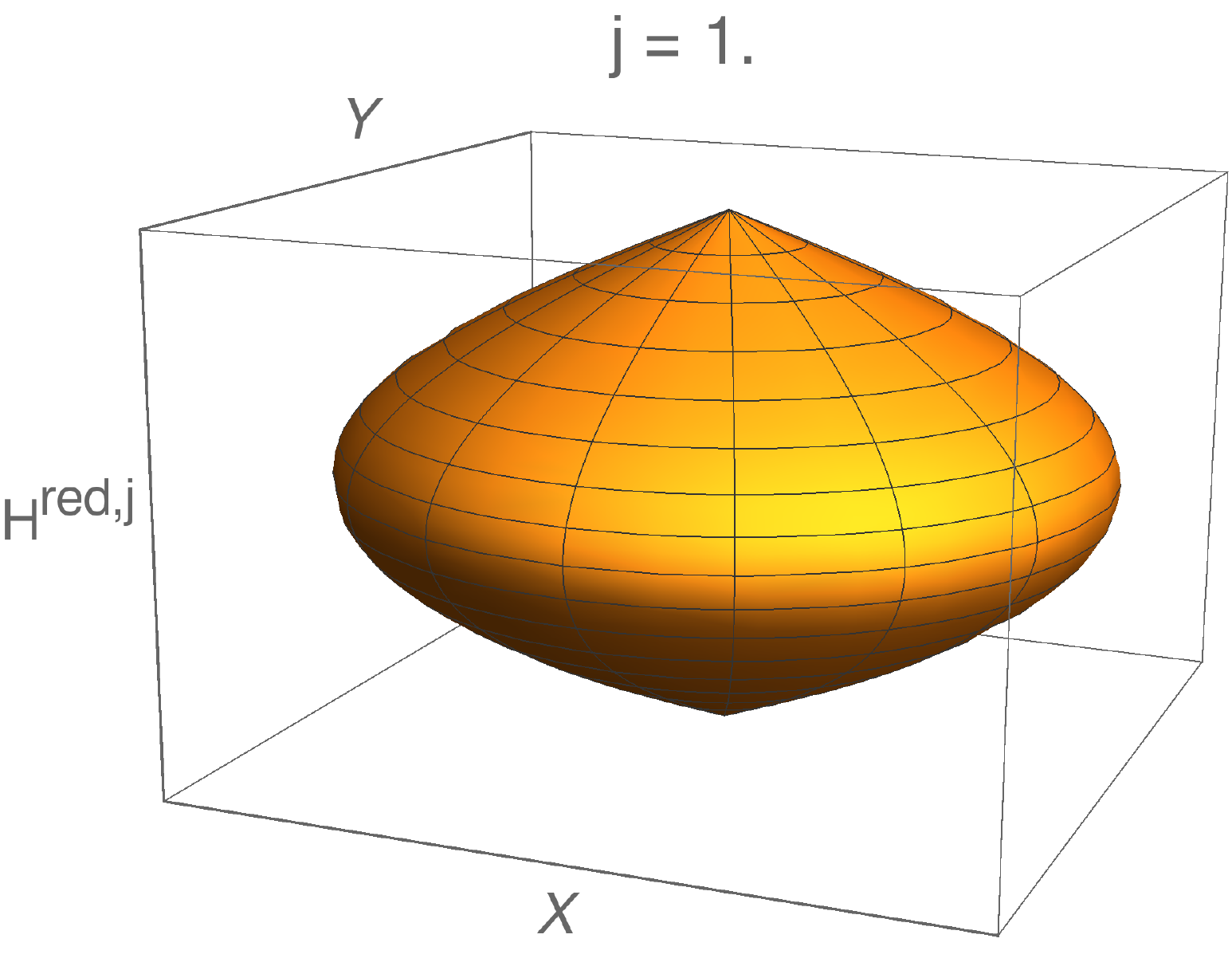}
\end{subfigure}%
\begin{subfigure}{0.33\textwidth}
\centering
\includegraphics[width=.9\linewidth]{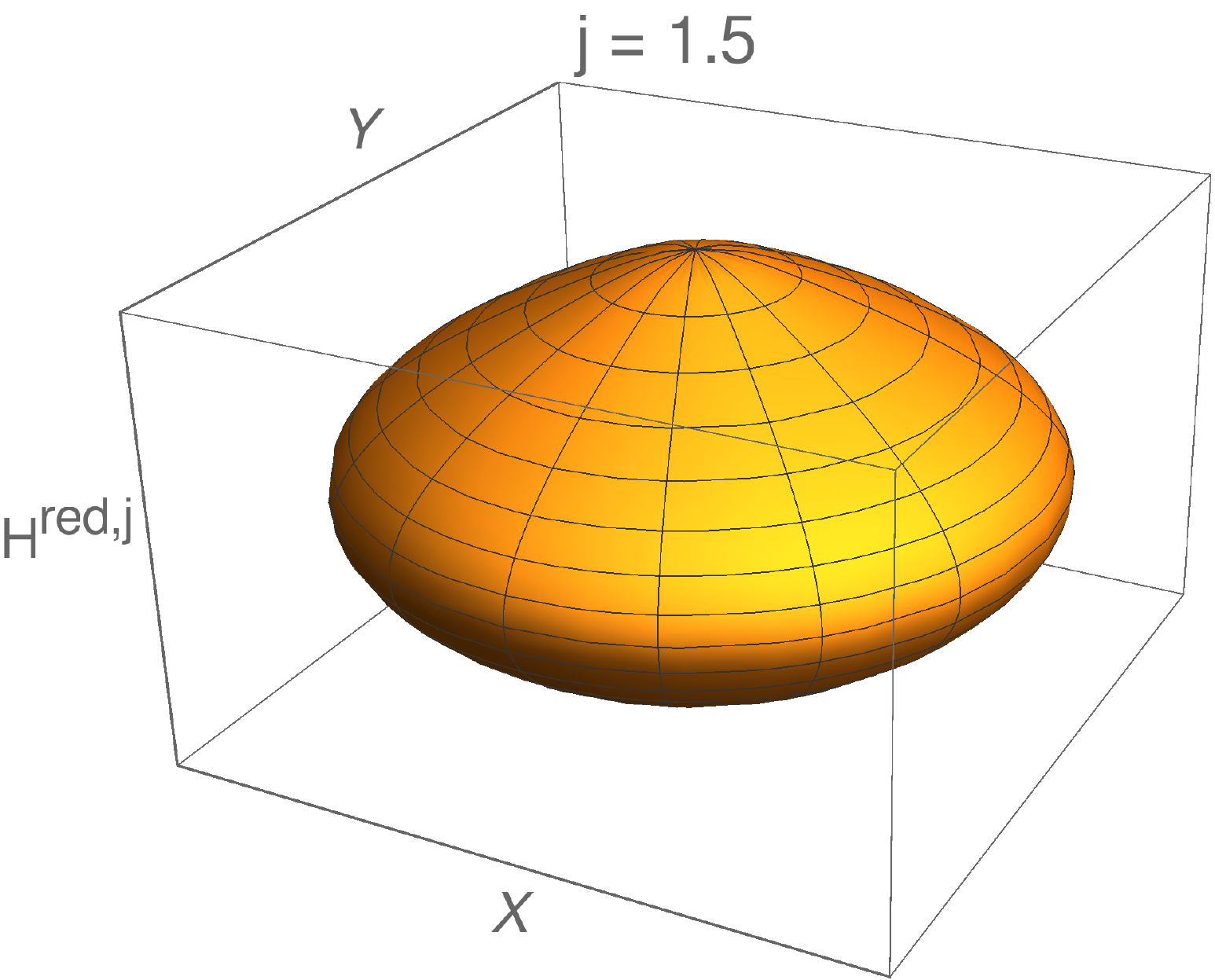}
\end{subfigure}%
\begin{subfigure}{0.33\textwidth}
\centering
\includegraphics[width=.9\linewidth]{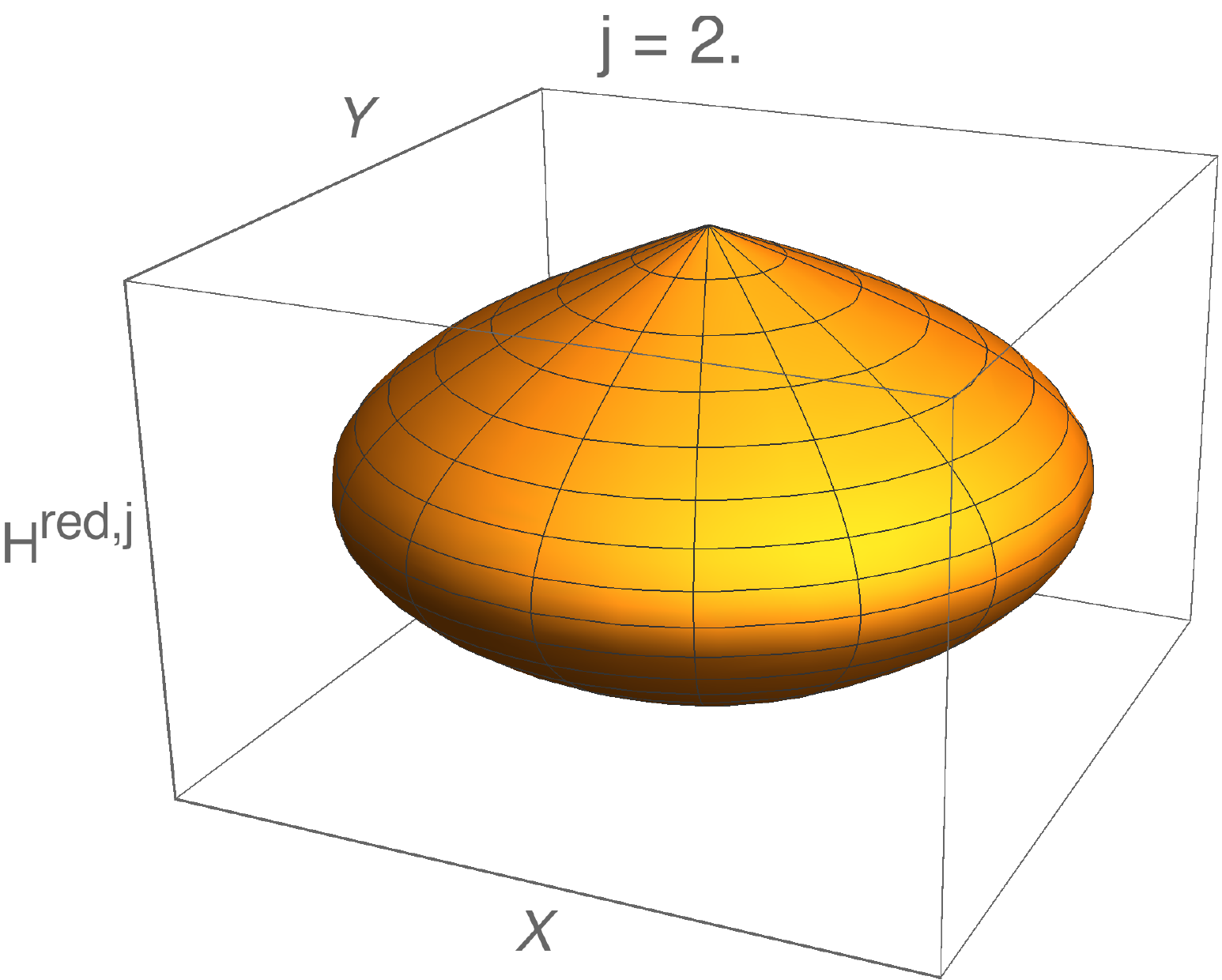}
\end{subfigure}

\vspace{5mm}
\begin{subfigure}{0.33\textwidth}
\centering
\includegraphics[width=.9\linewidth]{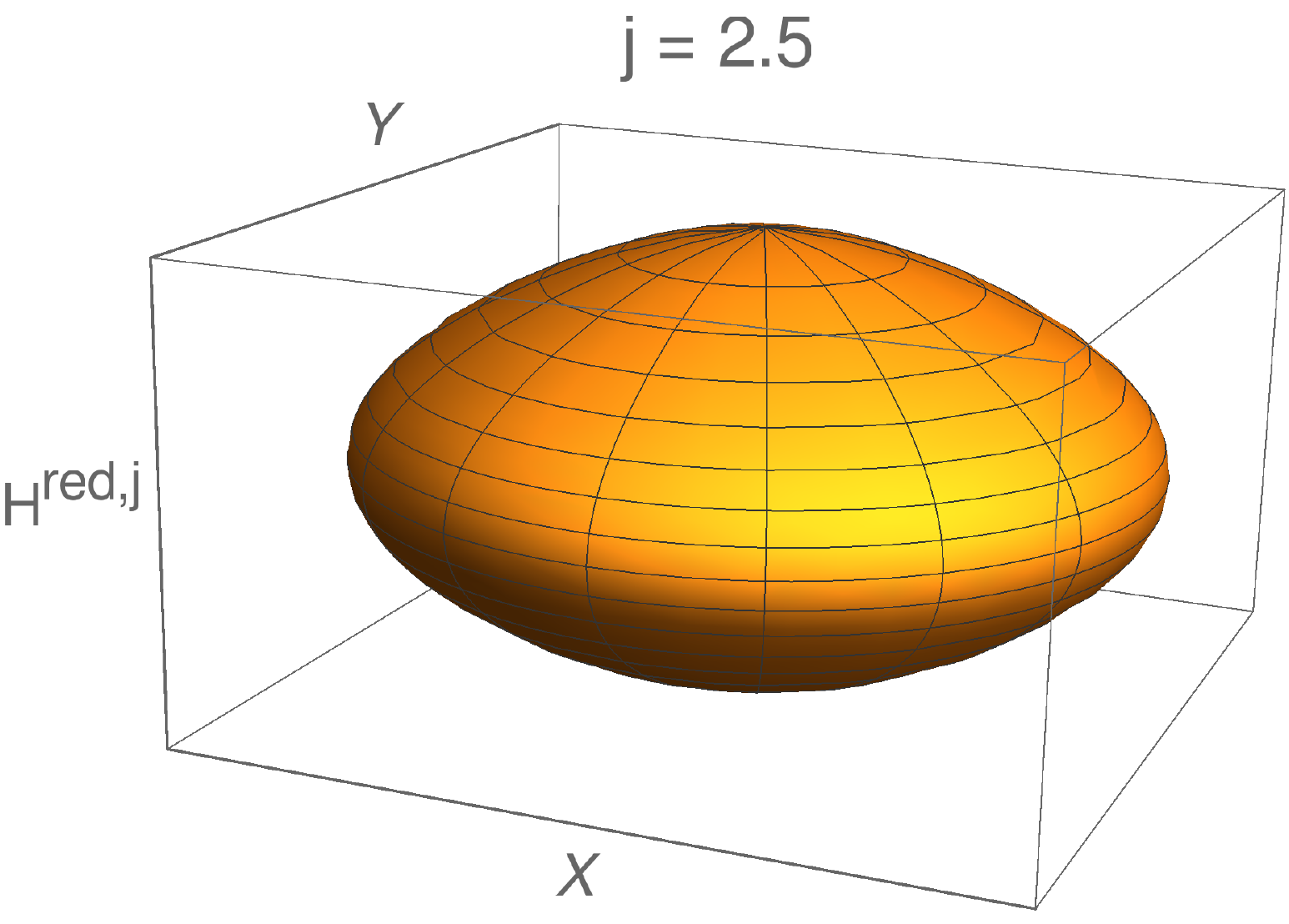}
\end{subfigure}%
\begin{subfigure}{0.33\textwidth}
\centering
\includegraphics[width=.9\linewidth]{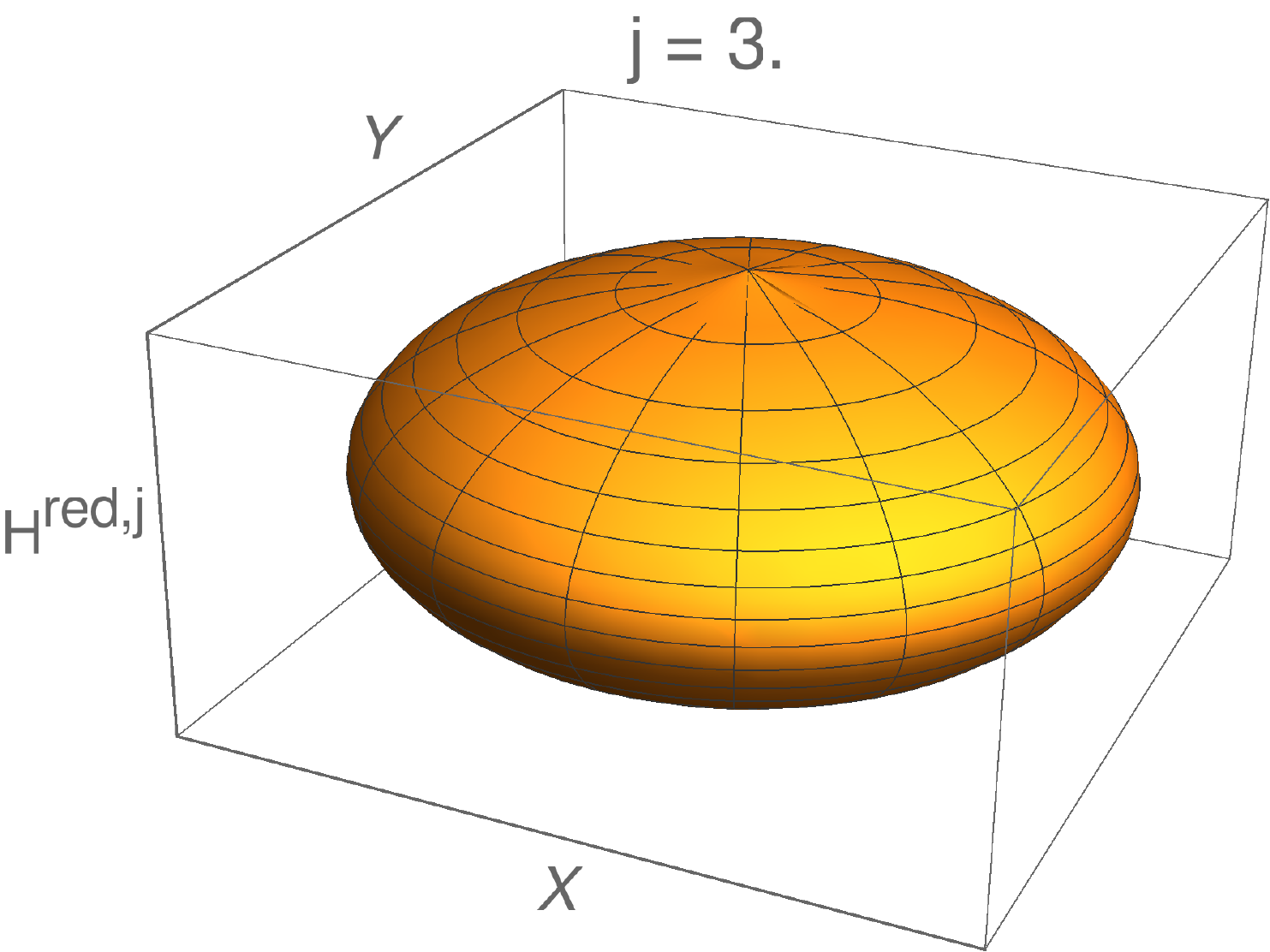}
\end{subfigure}

\caption{$M^{red, j}$ plotted for $j \in \{0,\ 0.5,\ 1, \ 1.5,\ 2,\ 2.5,\ 3 \}$ with {\it Mathematica}. For $j = 1$ and $j = 2$, the reduced space $M^{red, j}$ contains two singular points, i.e., it is homeomorphic to a 2-sphere with two `peaks', otherwise $M^{red, j}$ is diffeomorphic to a 2-sphere.
}
\label{Fig_Mj_red_reeks}
\end{figure}

Now consider the extremal values $j \in\{0, 3\}$ where, according to \refmjred, the reduced space is diffeomorphic to a 2-sphere. Thus there are two (elliptic) points $[p_{max}], [p_{min}] \in M^{red, j}$ where $H^{red, j}$ attains its maximum and minimum, i.e., $dH^{red, j}$ vanishes in $ [p_{max}]$ and $ [p_{min}]$. Since $j$ is extremal, we have $dJ(p) = 0$ for all $p \in J^{-1}(j)$. Altogether we conclude that $[p_{max}]$ and $[p_{min}] $ are mapped to elliptic-elliptic points of $F=(J,H)$ which are precisely the upper and lower vertex of the vertical edges over $j=0$ and $j=3$ in the octagon $\De$. The other points in $M^{red, j}$ correspond to elliptic-regular rank one points of $M$ and are mapped to the vertical segments of the momentum polytope between the vertices.


\subsection{Parametrisation of $M^{red, j}$}

\label{subsection param}

After these geometric considerations, we explain now how to obtain the coordinates used for displaying the reduced spaces $M^{red, j}$ in Figure \ref{Fig_Mj_red_reeks}. The idea is to `foliate' the $2$-dimensional space $M^{red, j}$ by the level sets of $H^{red, j}$ and consider $H^{red, j}$ as a height function of $M^{red, j}$. Then the possible values $h_j \in H^{red, j}(M^{red, j})$ lead to one coordinate. The other coordinate comes from parametrising the (usually $1$-dimensional) level sets of  $H^{red, j}$. Herefore we need a map that is welldefined on the reduced spaces $M^{red, j}$ and whose level sets are transverse to the ones of $H^{red, j}$.

\begin{lemma}
\label{XY}
Denote by $\Re$ and $\Im$ the real and imaginary parts of a complex valued function and consider $Z: \C^8 \to \C$ given by $Z(z_1, \ldots, z_8) :=\overline{z_2} \: \overline{z_3} \:\overline{z_4}\: z_6 \: z_7 \: z_8$. Then the functions
\begin{align*}
X: \C^8 \to \R,\quad X: = \Re (Z) \qquad \mbox{and} \qquad Y: \C^8 \to \R,\quad Y := \Im(Z)
\end{align*}
are $N$-invariant and thus descend as welldefined functions to $M$. Moreover, $X$ and $Y$ are also $J$-invariant and therefore descend to the reduced spaces $M^{red, j}$ for $j \in J(M)$.
\end{lemma}

\begin{proof} 
Recall the action of $N$ from \eqref{Action of N}. Let $(t_1, \ldots, t_6) \in \T^6 \simeq N$ and $(z_1, \ldots, z_8) \in \C^8$ and calculate
\begin{align*}
& Z(e^{it_1}z_1, e^{it_2}z_2, e^{it_3}z_3, e^{it_4}z_4, e^{i(t_1+t_2-t_4-t_5+t_6)}z_5, e^{it_5}z_6, e^{i(t_2+t_3+t_4-t_5-t_6)}z_7, e^{it_6}z_8) \\
&\quad = e^{i(-t_2-t_3-t_4+t_5+t_2+t_3+t_4-t_5-t_6+t_6)} (\overline{z_2} \: \overline{z_3} \:\overline{z_4}\: z_6 \: z_7 \: z_8) = \overline{z_2} \: \overline{z_3} \:\overline{z_4}\: z_6 \: z_7 \: z_8 = Z(z_1, \ldots, z_8)
\end{align*}
i.e., $Z$ is $N$-invariant. Now recall the $\mbS^1$-action induced by $J$ from \refthOctagon, let $t \in \mbS^1$ and $(z_1, \ldots, z_8) \in \C^8$, and compute
\begin{align*}
Z(e^{-it}z_1, z_2, \ldots, z_8) = \overline{z_2} \: \overline{z_3} \:\overline{z_4}\: z_6 \: z_7 \: z_8 = Z(z_1, \ldots, z_8)
\end{align*}
which shows the $J$-invariance of $Z$. Hence also $X=\Re(X)$ and $Y = \Im(Z)$ are $N$- and $J$-invariant.
\end{proof}

By means of the implicit function theorem, we would like to use $Z=(X, Y)$ to obtain coordinates on $M^{red, j}$ of the form $\left(u, H^{red, j}(u)\right)$ or $\left(u^{-1}(H^{red, j}), H^{red, j}\right)$. Thus we need to relate $Z=(X, Y)$ to $H$ and $J$.

\begin{lemma}
\label{XJH}
 Let $F=(J,H): (M, \om) \to \R^2$ be the toric system from \refthOctagon\ and $X$ and $Y$ the functions defined in \refXY. Then
 \begin{align*}
  X^2 + Y^2 = 2^6 H (H+J-1)(H-J+2)(-H-J+5)(3-H)(2-H+J).
 \end{align*}
\end{lemma}

\begin{proof}
Use the definition of $J$ and $H$ to find $\abs{ z_1}^2 = 2J$ and $\abs{ z_3}^2 = 2H$ and use the manifold equations to get furthermore
\begin{align*}
\vert z_2 \vert^2 & =-2 + 2H + 2J , \quad & \vert z_5 \vert^2 & = 6-2J ,    \quad  & \vert z_7 \vert^2 & = 6-2H ,  \\
\vert z_4 \vert^2 & = 4 + 2H - 2J ,   \quad & \vert z_6 \vert^2 & = 10 - 2H - 2J,    \quad & \vert z_8 \vert^2 & = 4 - 2H + 2J. 
\end{align*}
We now calculate 
 \begin{align*}
  X^2 + Y^2 & = \abs{Z}^2 =  \vert \overline{z_2} \: \overline{z_3} \:\overline{z_4}\: z_6 \: z_7 \: z_8 \vert^2 
= \vert z_2 \vert^2 \vert z_3 \vert^2 \vert z_4 \vert^2 \vert z_6 \vert^2 \vert z_7 \vert^2 \vert z_8 \vert^2 \\
& =  2^6 H (H+J-1)(H-J+2)(-H-J+5)(3-H)(2-H+J).
 \end{align*}
\end{proof}

Thus $Z=(X, Y)$ is related to $J$ and $H$ via a rotation invariant formula. Therefore we may set $Y=0$, solve for $X$ and obtain $M^{red, j}$ parametrised as a surface of revolution in $\R^3$ with coordinates $(X, Y, H^{red, j})$.

\begin{corollary}
\label{paramMred}
On the section $Y=0$, consider $X$ from \refXY\ by means of \refXJH\ as function 
$$
X(J, H) = 8 \sqrt{ H (H+J-1)(H-J+2)(-H-J+5)(3-H)(2-H+J)}.
$$
Then, for fixed $j \in J(M)$, the reduced space $M^{red, j}$ can be parametrised by
 \begin{align*}
 & \R\slash 2 \pi \Z \ \times\  H^{red, j}(M^{red, j}) \ \subseteq\ \R^2 \to \R^3, \\
 & (t, h_j) \mapsto \left( X(j, h_j) \cos(t), \  X(j, h_j) \sin(t) , \ h_j \right)
 \end{align*}
 where 
 $X(j, h_j) = 8 \sqrt{ h_j (h_j + j -1) (h_j -j + 2) (-h_j -j +5)(3-h_j)(2 - h_j +j)}$.
\end{corollary}

The function $h_j \mapsto \pm X(j, h_j)$ is plotted for various values of $j$ in Figure \ref{Figure_Mjred_2D}. Rotating the image of $h_j \mapsto X(j, h_j)$ around the horizontal axis gives the space $M^{red, j}$ that is plotted in Figure \ref{Fig_Mj_red_reeks} in the (rotated) coordinate system $(X, Y, H^{red, j})$.

\begin{figure}[h]
\raggedright 
\begin{subfigure}{0.33\textwidth}
\centering
\includegraphics[width=.9\linewidth]{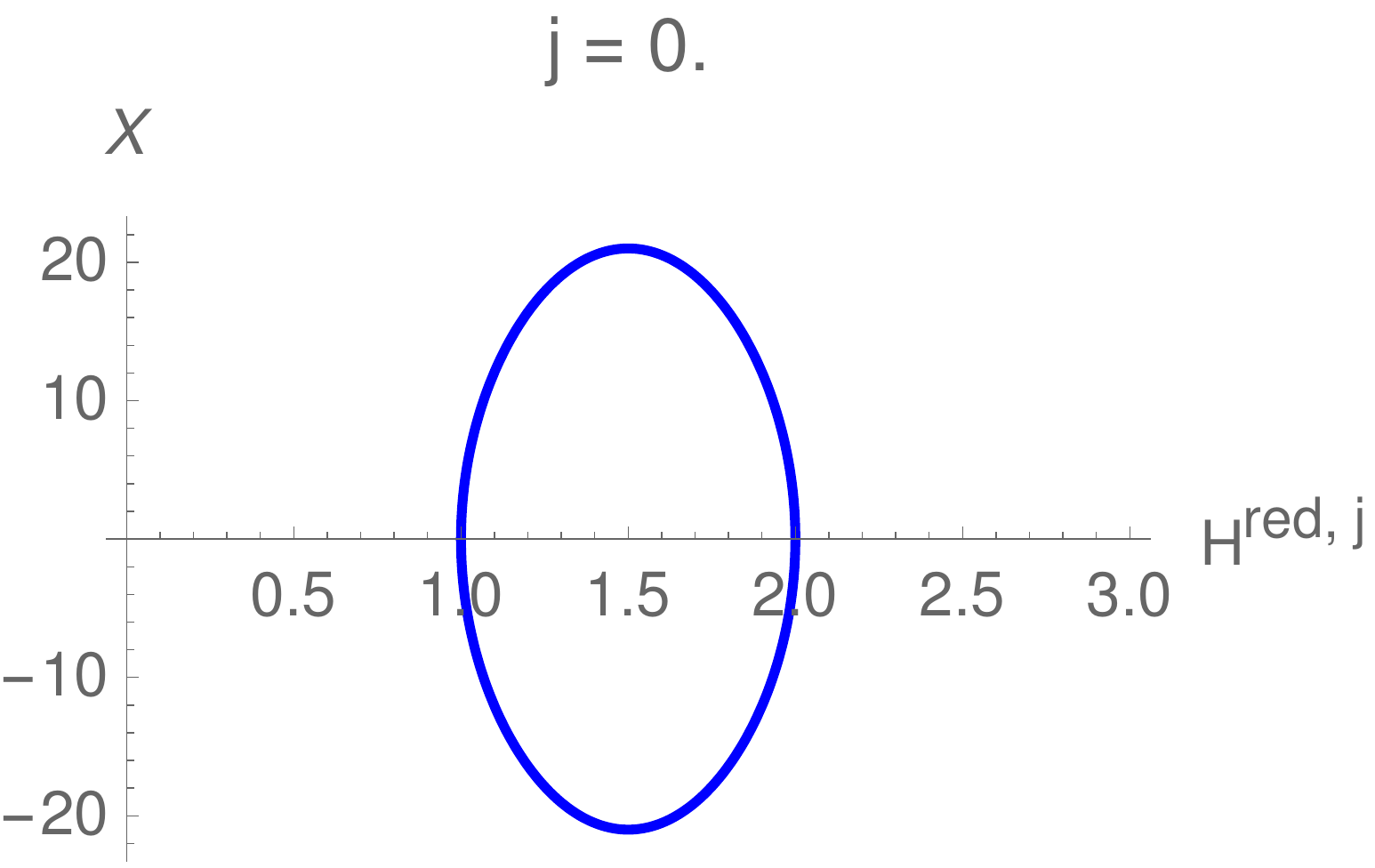}
\end{subfigure}%
\begin{subfigure}{0.33\textwidth}
\centering
\includegraphics[width=.9\linewidth]{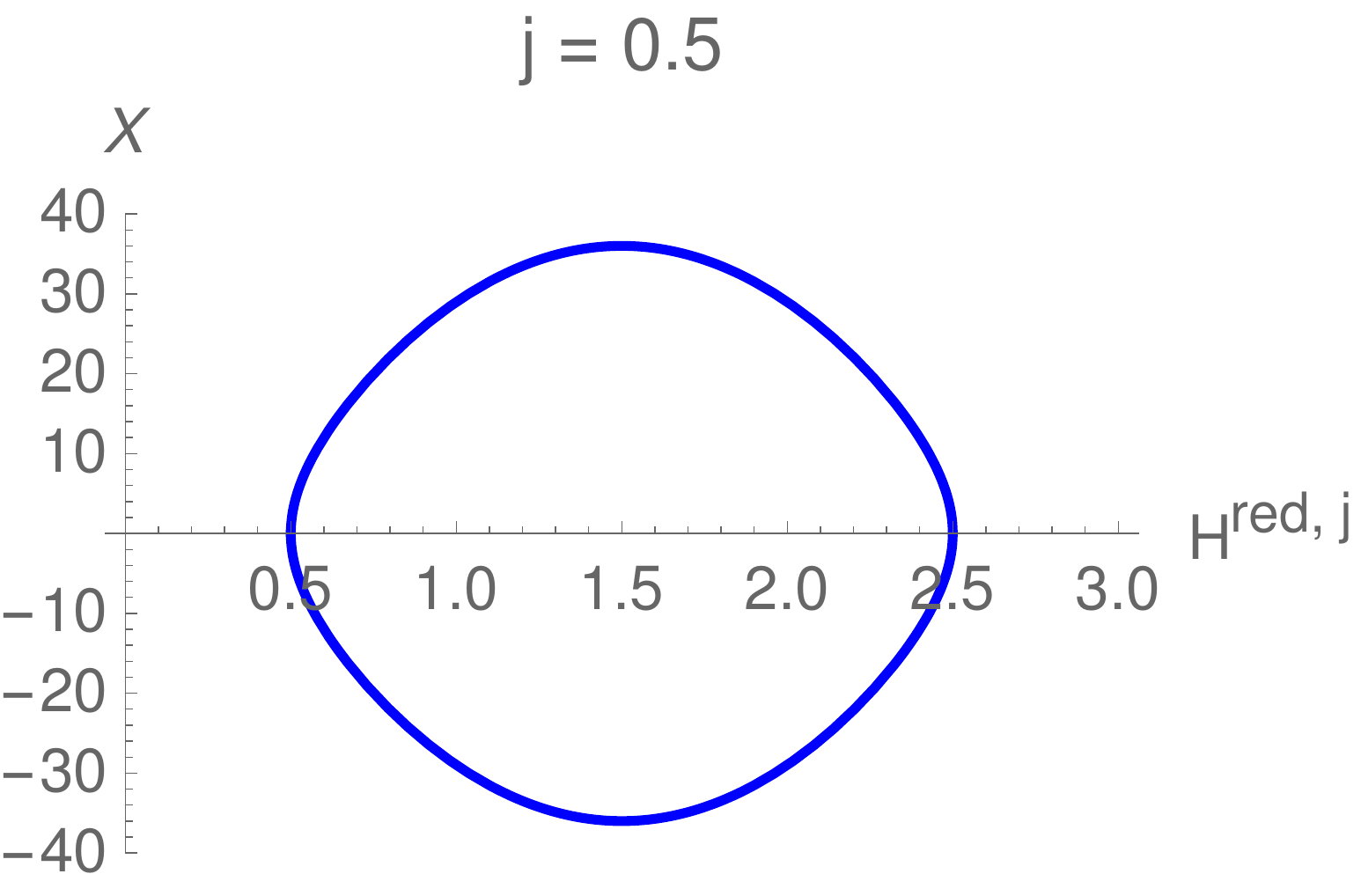}
\end{subfigure}

\vspace{5mm}
\begin{subfigure}{0.33\textwidth}
\centering
\includegraphics[width=.9\linewidth]{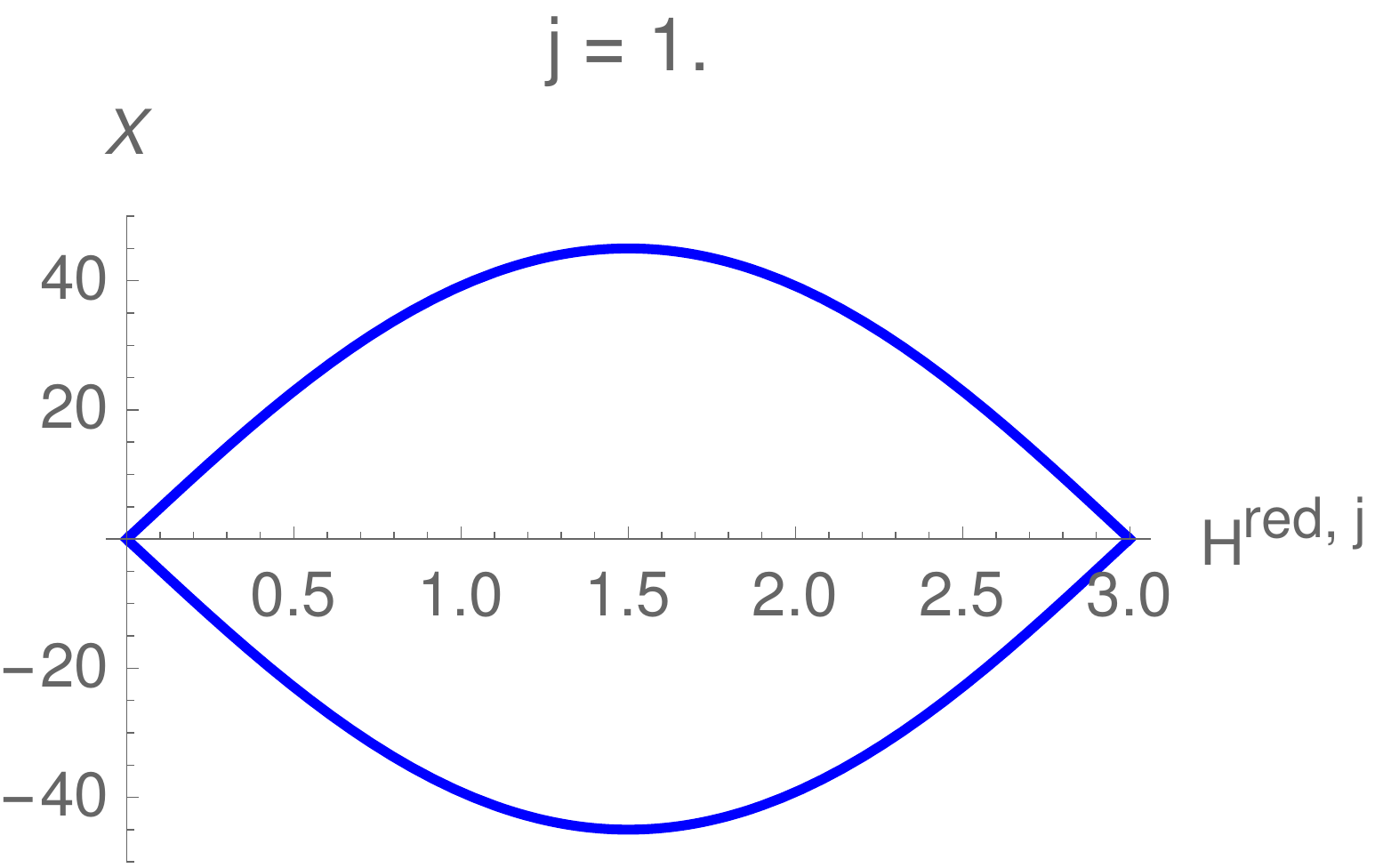}
\end{subfigure}%
\begin{subfigure}{0.33\textwidth}
\centering
\includegraphics[width=.9\linewidth]{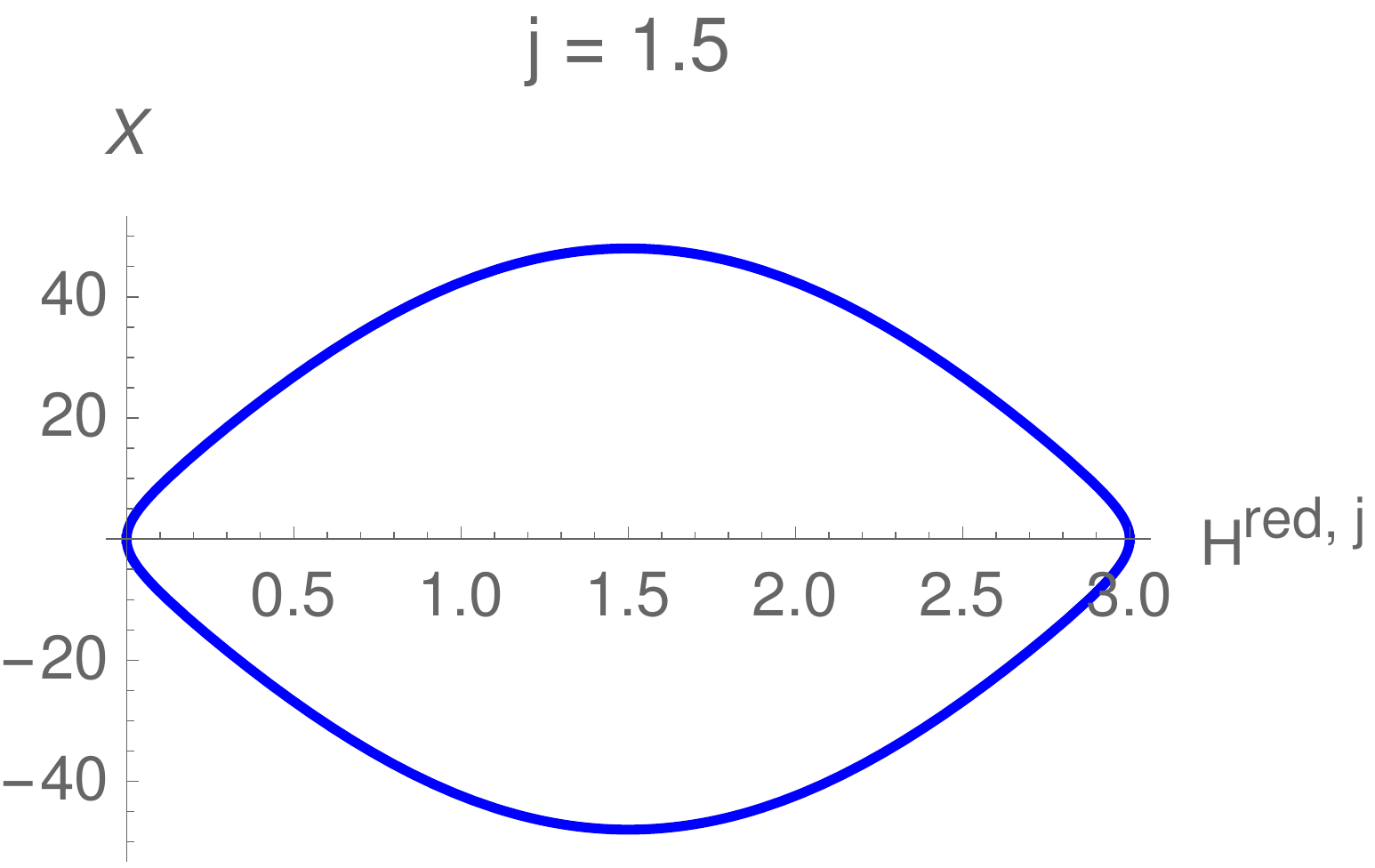}
\end{subfigure}%
\begin{subfigure}{0.33\textwidth}
\centering
\includegraphics[width=.9\linewidth]{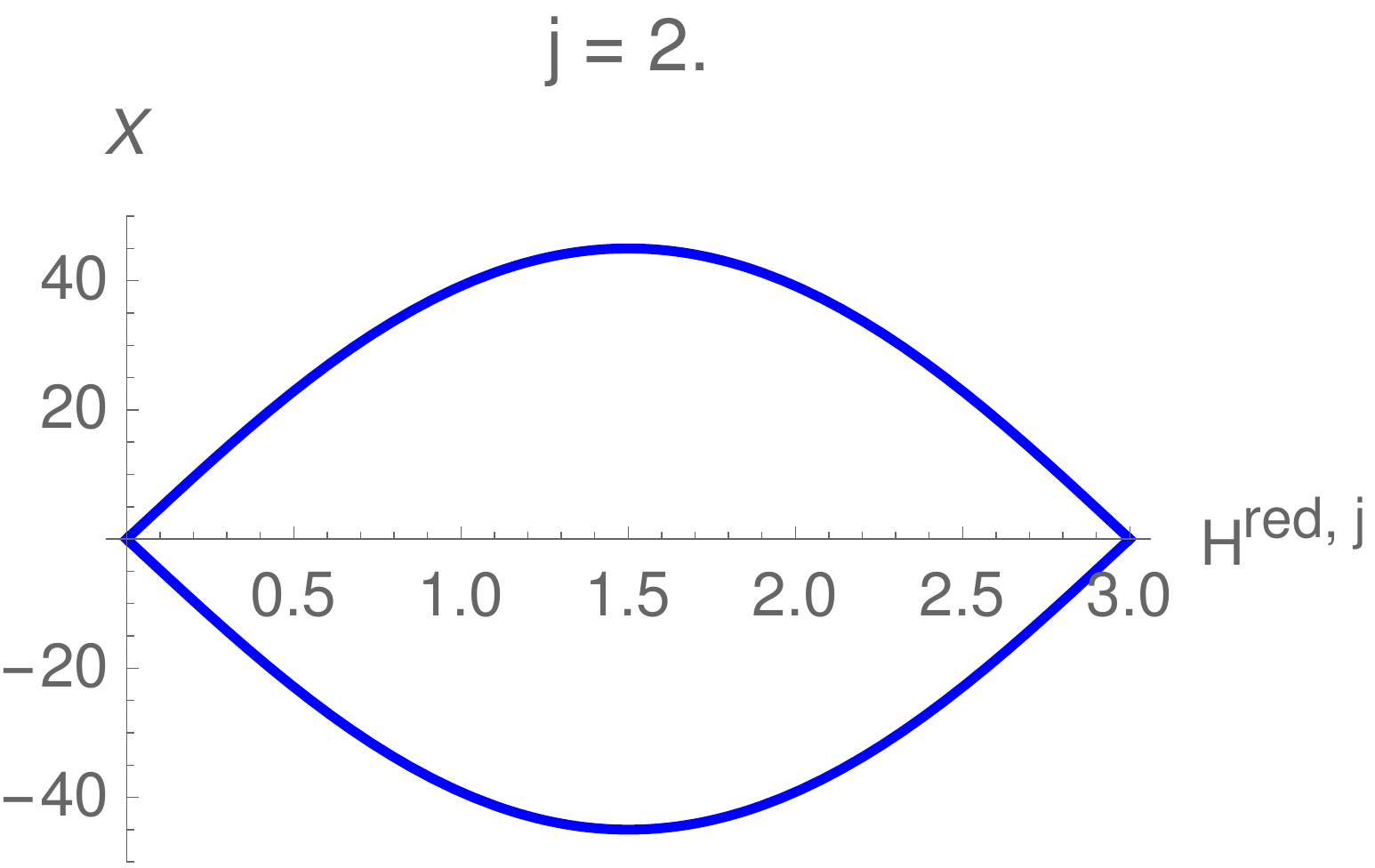}
\end{subfigure}

\vspace{5mm}
\begin{subfigure}{0.33\textwidth}
\centering
\includegraphics[width=.9\linewidth]{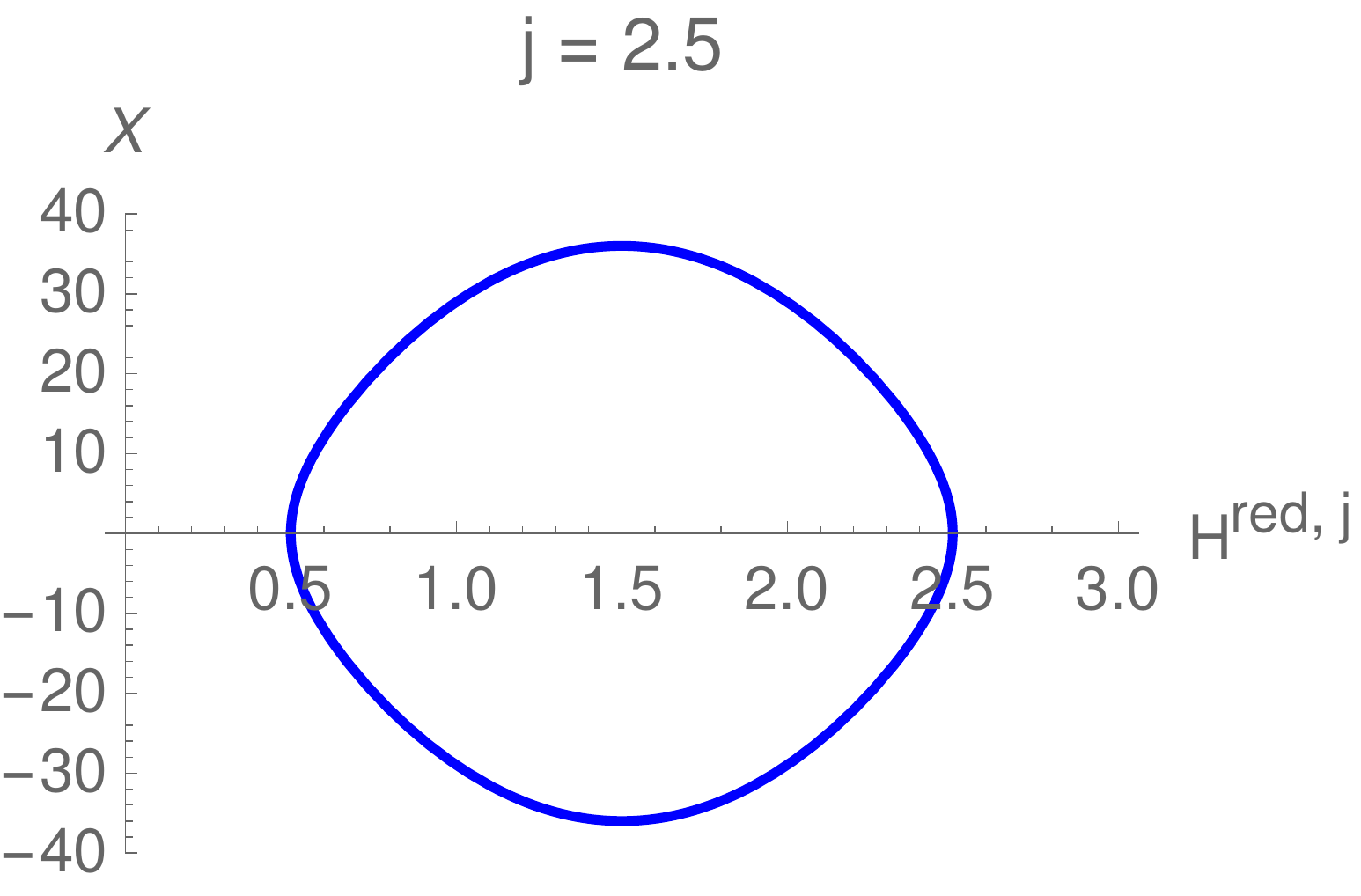}
\end{subfigure}%
\begin{subfigure}{0.33\textwidth}
\centering
\includegraphics[width=.9\linewidth]{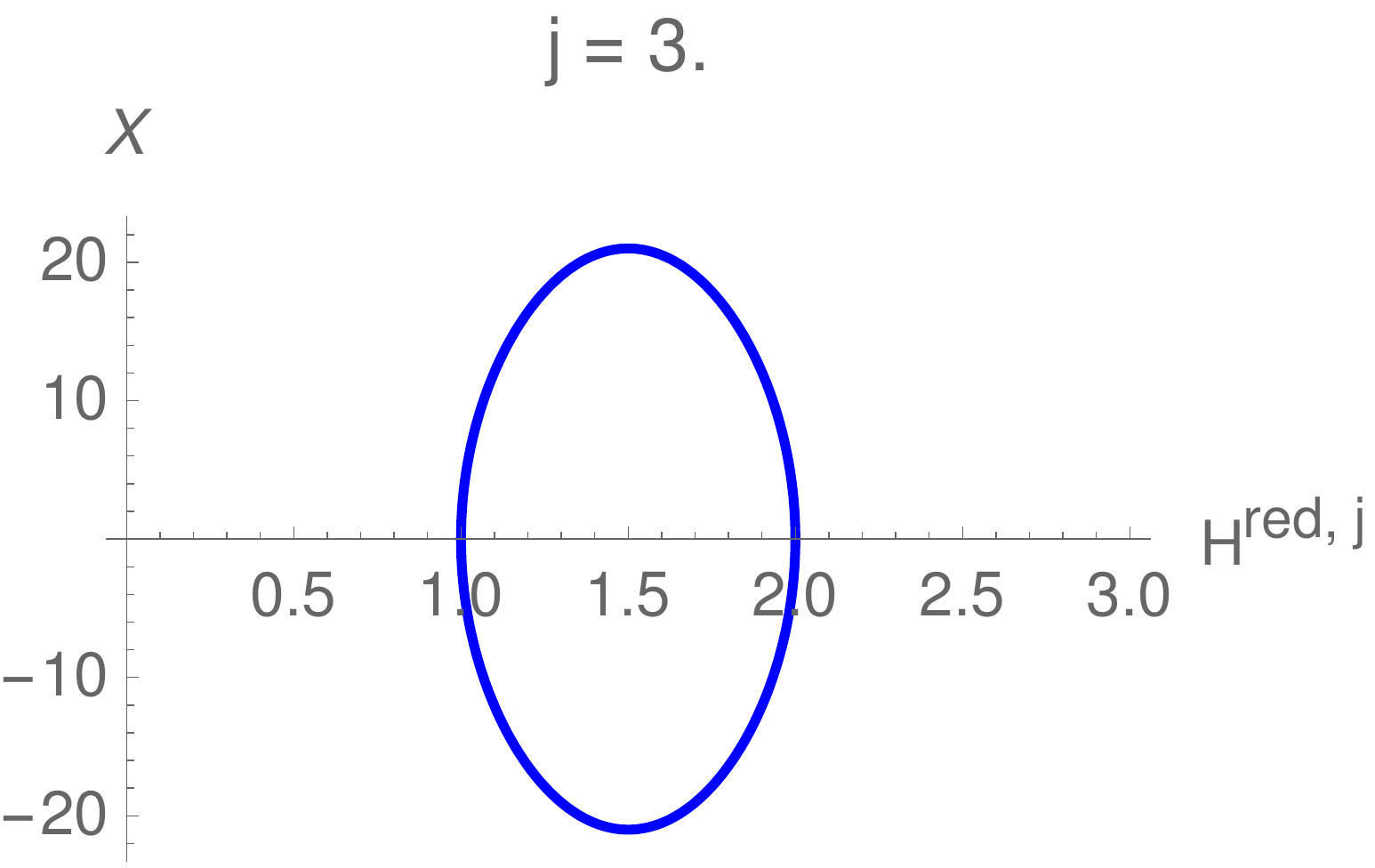}
\end{subfigure}

\caption{The section $Y = 0$ of the reduced space $M^{red, j}$ plotted with {\em Mathematica} for seven steps between $j=0$ and $j=3$, parametrised via $h_j \mapsto \pm X(j, h_j)$.
}
\label{Figure_Mjred_2D}
\end{figure}


\subsection{The family $\mathbf{ F_t = (J, H_t)}$ of integrable systems}

We know from \refthOctagon\ that $(M, \om, F=(J,H))$ is toric. The idea is now to obtain a family $H_t$ by interpolation between $H$ and $X$, i.e., we consider a new `height function' that changes from $H$ to $X$ when varying the parameter.

\begin{proposition}
\label{HtFt}
 Let $t \in \R$ and $\ga \in \R^{\neq 0}$ and define on the symplectic manifold $(M, \om)$ constructed in Section \ref{subsection_Delzant_construction} the map
 $$
 H_t: (M, \om) \to \R, \quad H_t: = (1-2t)H + t\gamma X.
 $$
 Then $ F_t:=(J, H_t) : (M, \om) \to \R^2$
 is a completely integrable system.
\end{proposition}

\begin{proof}
There are two ways to show that $J$ and $H_t$ Poisson commute w.r.t.\ the Poisson bracket $\{ \cdot, \cdot \}$ induced by $\om$. Let us first argue abstractly:

Since $F=(J,H) : (M, \om) \to\R^2$ is a toric momentum map according to \refthOctagon, we have in particular $\{J, H\}=0$. Moreover, $X$ can be seen as function $X(J,H)$ depending solely on the Poisson commuting functions $J$ and $H$ and some constants. Thus $\{X(J,H), J\} = 0 = \{ X(J,H), H\}$ such that linearity of the Poisson bracket implies 
$$
0 = (1-2t) \{J, H\} + t \ga \{J, X\} = \{J,  (1-2t)H + t\gamma X\} = \{J, H_t\}.
$$
Later on in this paper, we will need (parts of) the explicit calculations in local coordinates of $\{J, H_t\}=0$, so we add here here the proof in local coordinates. We only consider the chart $(U_1, \psi_1)$ since the cases $2 \leq \nu \leq 8$ go analogously. We find
$$
(J \circ \psi_1^{-1})(x_1, y_1, x_2, y_2) = \frac{1}{2} \left(x_1^2 + y_1^2 \right) 
\quad \mbox{and} \quad 
d(J \circ \psi_1^{-1}) = (x_1,y_2, 0, 0)
$$
and, using $x_3 = \sqrt{2 - \abs{z_1}^2 + \abs{z_2}^2} $ from \eqref{sixVarEq} and $y_3 =0$, we get
\begin{align}
\label{HinCoord}
(H \circ \psi_1^{-1})(x_1, y_1, x_2, y_2) & = \frac{1}{2} \left(x_3^2 + y_3^2 \right) = \frac{1}{2} \left( 2- x_1^2 - y_1^2 + x_2^2 + y_2^2 \right),  \\ \notag
d(H \circ \psi_1^{-1}) & = (-x_1, -y_1, x_2, y_2).
\end{align}
Recall that the symplectic form $\om_1$ on $\psi_1(U_1)=V_1$ is the standard symplectic form $\om_{st}$ on $\R^4$. We calculate
\begin{align*}
& \mcX^{J \circ \psi_1^{-1}}(x_1, y_1, x_2, y_2) = (y_1, -x_1, 0, 0)^T , \\
& \mcX^{H \circ \psi_1^{-1}} (x_1, y_1, x_2, y_2)= (-y_1, x_1, y_2, -x_2)^T, \\
& \{J \circ \psi_1^{-1} , H\circ \psi_1^{-1} \} = - \om_1 \left(\mcX^{J \circ \psi_1^{-1}},\mcX^{H \circ \psi_1^{-1}} \right) = - d(J \circ \psi_1^{-1})\mcX^{H \circ \psi_1^{-1}} = 
\bigl(x_1, \ y_1, \ 0, \ 0 \bigr) 
\begin{pmatrix}
-y_1 \\ x_1 \\ y_2 \\ -x_1
\end{pmatrix} = 0.
\end{align*}
Thus $J$ and $H$ Poisson commute. Moreover, using the definition of $\psi^{-1}_1$, we find
\begin{align*}
Z\circ \psi^{-1}_1 (x_1, y_1, x_2, y_2) &  = Z(x_1, y_1, x_2, y_2, x_3, 0, x_4, 0, x_5, 0, x_6, 0, x_7, 0, x_8, 0) \\
& = ( x_2 - i y_2) x_3 x_4 x_6 x_7 x_8 = x_2 x_3 x_4 x_6 x_7 x_8 + i x_3 x_4 x_6 x_7 x_8 y_2 
\end{align*}
and by using the relations defining $x_3, \dots, x_8$ we get
\begin{align}
\label{XinCoord}
& (X \circ \psi^{-1}_1)(x_1, y_1, x_2, y_2) = x_2 x_3 x_4 x_6 x_7 x_8  \\ \notag
& \quad = x_2 \sqrt{ \left(2-(x_1^2 + y_1^2) + (x_2^2 + y_2 ^2)\right) \left( 6- 2( x_1^2 + y_1^2 ) + ( x_2^2 + y_2 ^2 )\right) \left( 8 - ( x_2^2 + y_2 ^2) \right) } \\ \notag
& \quad \quad  \quad \sqrt{\left( 4 + (x_1^2 + y_1^2  ) - ( x_2^2 + y_2 ^2 ) \right) \left(2 + 2 ( x_1^2 + y_1^2 ) -( x_2^2 + y_2 ^2 )\right)} 
\end{align}
which can be seen, in the first two variables, as a function depending on $\xi_1=x_1^2$ and $\eta_1=y_1^2$ and that is symmetric in $\xi_1$ and $\eta_1$. Thus the partial derivatives of $X \circ \psi^{-1}_1$ in $(x_1, y_1, x_2, y_2)$ w.r.t.\ $x_1$ and $y_1$ coincide except for the factor $x_1$ resp.\ $y_1$, i.e., we obtain
\begin{align*}
 d(X \circ \psi^{-1}_1) & = \bigl(x_1 f,\  y_1 f,\  g,\  h \bigr)
\end{align*}
for suitable functions $f$, $g$, $h: V_1 \to \R$ with coordinates $(x_1, y_1, x_2, y_2)$. This yields at the point $(x_1, y_1, x_2, y_2)$
\begin{equation*}
 \mcX^{X \circ \psi^{-1}_1} = \bigl(y_1 f,\ - x_1 f,\  h,\  -g \bigr)^T
\end{equation*}
and thus
\begin{equation*}
 \{ J \circ \psi^{-1}_1, X \circ \psi^{-1}_1\} = - d(J \circ \psi^{-1}_1) \left(\mcX^{X \circ \psi^{-1}_1} \right) 
 = \bigl( x_1, \ y_1, \ 0, \ 0 \bigr)  \begin{pmatrix} y_1 f \\  - x_1 f \\  h \\  -g \end{pmatrix} =0.
\end{equation*}
Since $H_t=(1-2t) H + t \ga X$, linearity of the Poisson bracket yields $\{J \circ \psi^{-1}_1,H_t \circ \psi^{-1}_1\} = 0$ for all $t$ and all $\ga$.

Now we show that $\mcX^{J}$ and $\mcX^{H_t}$ are almost everywhere linearly independent. For $t=0$, we have $H_0=H$ and thus the claim follows from $F=(J, H)$ being an integrable system. For $t \neq 0$, consider $\mcX^{H_t \circ \psi^{-1}_1} = (1-2t) \mcX^{H \circ \psi^{-1}_1} + t \ga \mcX^{X \circ \psi^{-1}_1}$ and note that $\mcX^{J \circ \psi^{-1}_1}$ and $\mcX^{X \circ \psi^{-1}_1}$ are linearly dependent if and only if $\del_{x_2}(X \circ \psi^{-1}_1) = 0 = \del_{y_2}(X \circ \psi^{-1}_1)$.
This is only possible if $y_2 = 0$ and $x_2$ satisfies an equation, so these $(x_1, y_1, x_2, y_2)$ live in a two-dimensional subset and hence, the linear independence holds almost everywhere. 
\end{proof}


\subsection{Intuition on focus-focus points of the family $\mathbf{F_t=(J, H_t)}$}

On the reduced space, the family $H_t$ has the following geometric meaning: For $t=0$, we have $H^{red, j}_0 =H^{red, j}$ which is our `vertical' height coordinate. For $t \in\ ]0, \frac{1}{2}[$, the `height' $H^{red, j}_t$ is measured `tilted more and more towards the horizontal' $X$-coordinate which is reached (up to scaling by $\frac{\ga}{2}$) at $t=\frac{1}{2}$ with $H^{red, j}_{\frac{1}{2}} = (\frac{\ga}{2} X)^{red, j}$.
Then, the singular points of the manifold are interior points of this function and will be singular points of focus-focus type. This idea is visualised in Figure \ref{Figure_Mjred_3D} for the level $J = 1=j$.

\begin{figure}[h]
\centering
\includegraphics[scale=0.4]{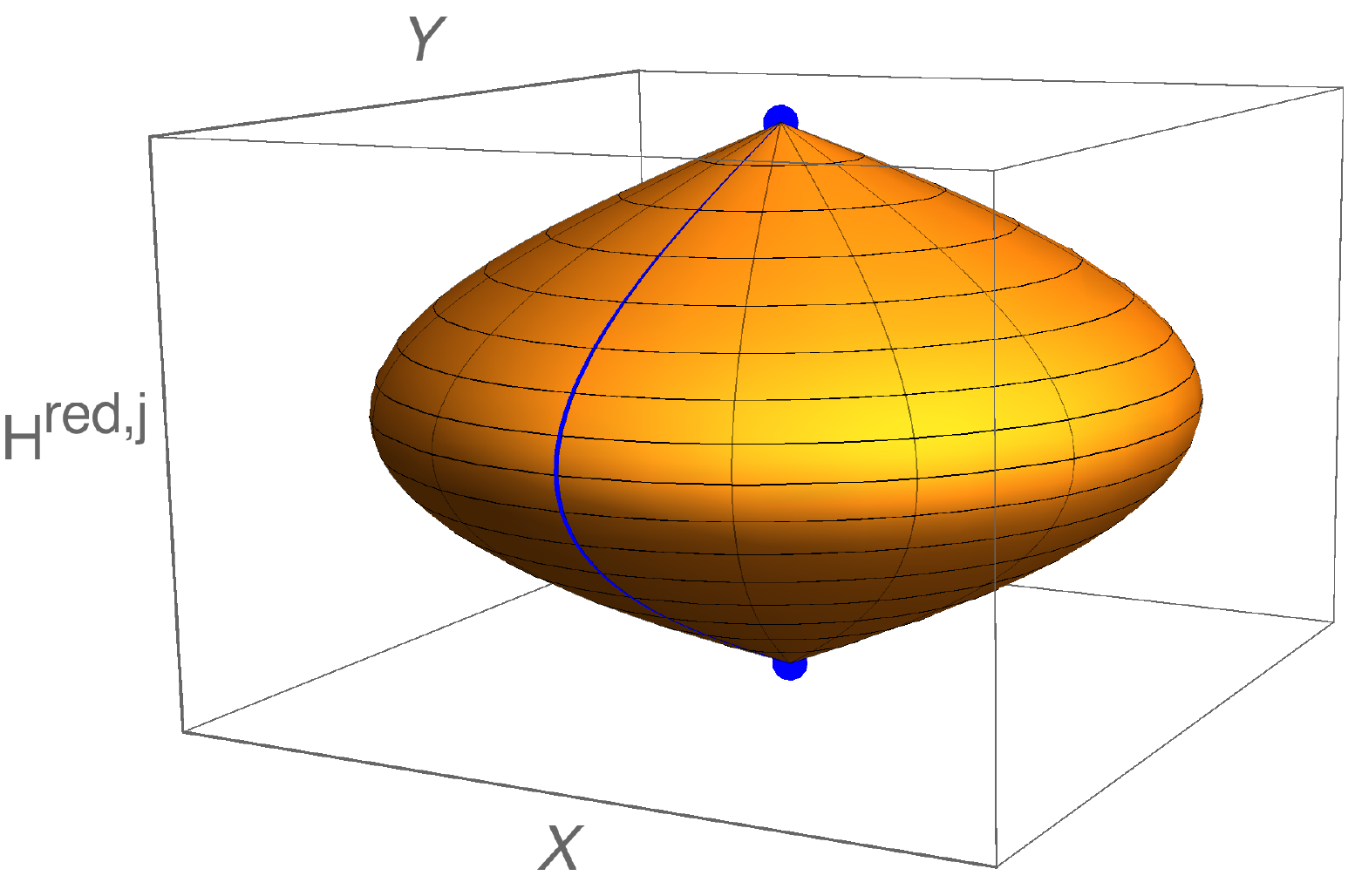}
\caption{The reduced space $M^{red, j}$ for the singular value $j = 1$. At $t=0$, we have $H_0^{red,1} =H^{red, 1}$ and the singular levels are just the two singular elliptic points. At $t= \frac{1}{2}$, we have $H_{\frac{1}{2}}^{red,1} = (\frac{\ga}{2}X)^{red, 1}$ and the singular level is the blue curve connecting the two singular focus-focus points. This plot is done with {\em Mathematica}.}
\label{Figure_Mjred_3D}
\end{figure}


\subsection{Intuition on the rank 1 points of the family $\mathbf{ F_t=(J,H_t)}$}

At the interior levels $j \in \: ]0,3[ \ \setminus \{1,2\}$, the reduced space $M^{red, j}$ is diffeomorphic to a $2$-sphere. Thus, no matter how far $H_t^{j,red}$ is `tilted' from $H$ towards $X$, the maximum and minimum of $H_t^{red, j}$ are each reached at a unique point on $M^{red, j}$. 

By \refredRankOne, these points correspond to the rank one points of $(M, \om, F_t)$ which will turn out to be nondegenerate of elliptic-regular type and map to the top and bottom edge points in the momentum polytope $F_t(M)$. As $H_t^{red, j}$ `tilts' from $H^{red, j}=H $ to $X$ as $t$ increases, also the rank one points on $M$ may change.

For the levels $j \in \{0,3\}$, the minimum resp.\ maximum of $H_t^{red, j}$ is also attained at a unique point on $M^{red, j}$. Both correspond to rank zero points of elliptic-elliptic type on $M$.

When $t$ changes from $0$ to $1$, other rank one points will become the rank zero points, but, according to \refzeroToOne, none of these points ever becomes regular.

For the levels $j  \in \{1, 2\} $, the reduced manifold has two singular points, which are the maximum and minimum of $H_t^{red, j}$ as long as $t < t^-$ and $t^+ < t$ for certain `degenerate times' $0< t^- < \frac{1}{2} < t^+ <1$. When $t$ is in between the degenerate times, two other points on the reduced space will become the extremal values of $H_t^{red, j}$. These points correspond to nondegenerate elliptic-regular rank one points on $M$ which are mapped to the upper and lower boundary of the octagon. The fixed points are then focus-focus points and mapped to the interior.


\subsection{Main results}

We showed in \refthOctagon\ that $(M, \om, F=(J, H))$ is a toric system with $F(M)=\De$. The family $F_t=(J, H_t)$ was shown to be an integrable system in \refHtFt. But $(M, \om, F_t)$ has even nicer properties:

\begin{theorem} 
\label{mainTheorem}  
Let $(M, \om, F=(J, H))$ be the toric system from \refthOctagon\ and let $\bigl(M, \om, F_t =(J, \ (1-2t) H + t\ga X) \bigr)$ be the family of integrable systems from \refHtFt\ and let $ 0 < \gamma < \frac{1}{48}$.
Then $(M, \omega, F_t)_{0 \leq t \leq 1}$ is a \emph{semitoric transition family with fixed $\mbS^1$-action} having four transition points and two transition times
$$ 0 \; < \; t^- := \frac{1}{2(1+24\gamma)} \; < \ \frac{1}{2} \ < \; t^+ := \frac{1}{2(1-24\gamma)} \; < \; 1. $$
\end{theorem}

The proof is spread over the following sections and summarised in Section \ref{section pinched}.
The image $F_t(M)$ of the momentum map is plotted for $\ga=\frac{1}{60}$ in Figure \ref{Fig_maintheorem}: The images of four elliptic-elliptic fixed points `pass into the interior' of the momentum polytope at $t=t^-$, becoming focus-focus points. At $t=\frac{1}{2}$, the focus-focus points form two pairs where each pair is mapped to the same value in the momentum polytope, i.e., the points of a pair lie in the same fibre. At $t=t^+$, the images of the focus-focus points become again boundary points of the momentum polytope, i.e., they switch back to being elliptic-elliptic. We compute now the coordinates in $M$ of the eight fixed points of $F_t=(J, H_t)$.

\begin{proposition} 
\label{rankOpoints}
For all $t \in [0,1]$, the system $F_t=(J, H_t)$ has precisely eight fixed points. They are given by four points not depending on $t$, namely
\begin{align*}
A &= [\sqrt{2},\ \sqrt{6},\ \sqrt{6}, \ 2\sqrt{2},\ 2,\ \sqrt{2},\ 0,\ 0] , & C &= [2,\ 2\sqrt{2},\ \sqrt{6},\  \sqrt{6},\ \sqrt{2},\ 0,\ 0,\ \sqrt{2}], \\
B &= [\sqrt{2},\ 0,\ 0,\ \sqrt{2},\ 2,\ 2\sqrt{2},\ \sqrt{6},\ \sqrt{6}] , & D &= [2,\ \sqrt{2},\ 0,\ 0,\ \sqrt{2},\ \sqrt{6}, \ \sqrt{6},\ 2\sqrt{2}]
\end{align*}
as in \refcoordEEPoints, and four points that change with $t$ as follows:
\begin{eqnarray*}
P_t^{min} &=& \left[0, \ u_-(t),\sqrt{2+u_-^2(t)}, \sqrt{6+u_-^2(t)}, \sqrt{6}, \sqrt{8 - u_-^2(t)}, \sqrt{4 - u_-^2(t)}, \sqrt{2 - u_-^2(t)}\right], \\
P_t^{max} &=& \left[0, \ u_+(t),\sqrt{2+u_+^2(t)}, \sqrt{6+u_+^2(t)}, \sqrt{6}, \sqrt{8 - u_+^2(t)}, \sqrt{4 - u_+^2(t)}, \sqrt{2 - u_+^2(t)}\right] ,\\
Q_t^{min} &=& \left[\sqrt{6},\sqrt{8-v_+^2(t)}, \sqrt{4-v_+^2(t)}, \sqrt{2-v_+^2(t)}, \ 0, \  v_+(t), \sqrt{2+v_+^2(t)}, \sqrt{6+v_+^2(t)}\right], \\
Q_t^{max} &=& \left[\sqrt{6},\sqrt{8-v_-^2(t)}, \sqrt{4-v_-^2(t)}, \sqrt{2-v_-^2(t)},\ 0, \   v_-(t), \sqrt{2+v_-^2(t)}, \sqrt{6+v_-^2(t)}\right],
\end{eqnarray*}
where $ u_\pm: [0,1] \to[-\sqrt{2}, \sqrt{2}]$ are the unique smooth solutions of the equation
\begin{equation} 
\label{Rank0eqnLem}
(1-2t)u \sqrt{\mff\bigl(u^2\bigr)} + \gamma t \mff\bigl(u^2\bigr) + \gamma t u^2 \mff'\bigl(u^2\bigr) = 0
\end{equation}
where
$$ \mff(\ze):= (2+ \ze) (6+ \ze) (8- \ze) ( 4- \ze) (2-\ze). $$
$u_\pm$ has initial values $u_-(0)= 0$ and $u_+(0) = \sqrt{2}$. 
Moreover, $ v_\pm: [0,1] \to [-\sqrt{2}, \sqrt{2}]$ are the unique smooth solutions of
\begin{equation}
\label{Rank0eqn2}
-(1-2t)v \sqrt{\mff\bigl(v^2\bigr)} + \gamma t \mff\bigl(v^2\bigr) + \gamma t v^2 \mff'\bigl(v^2\bigr) =0.
\end{equation}
$v_\pm $ has initial values $v_-(0)= - \sqrt{2}$ and $v_+(0) = 0$.
For $t=0$, we recover $P^{min}_0=P^{min}$, $P^{max}_0=P^{max}$, $Q^{min}_0=Q^{min}$, $Q^{max}_0=Q^{max}$ from \refcoordEEPoints.
\end{proposition}

This will be proven in Section \ref{section posRankZero}. 
Let us now shed some light on the situation at $t=\frac{1}{2}$.

\begin{proposition}
\label{doublePinchParam}
 At $t=\frac{1}{2}$, the focus-focus points $A$ and $B$ lie both in $F_{\frac{1}{2}}^{-1}(1,0)$ and the focus-focus points $C$ and $D$ lie both in $F_{\frac{1}{2}}^{-1}(2,0)$ and both fibres have the form of a double pinched torus as displayed in Figure \ref{Fig_double_pinched}.
Exemplarily, we compute the fibre $ F_{\frac{1}{2}}^{-1}\left(1,0\right)$ as
$$ \left\{ \left. \left[\sqrt{2}, \sqrt{6 - \vert z_8 \vert^2}, \sqrt{6 - \vert z_8 \vert^2}, \sqrt{8 - \vert z_8 \vert^2},\ 2, \sqrt{2 + \vert z_8 \vert^2}, \ \pm i \overline{z_8}, \ z_8 \right] \in M \ \right| \vert z_8 \vert \in \left[0, \sqrt{6}\right] \right\}. $$
Using polar coordinates, this yields for $ F_{\frac{1}{2}}^{-1}\left(1,0\right)$ the parametrisation
$$
\left\{ \left. \left[\sqrt{2}, \sqrt{6 - r^2}, \sqrt{6 - r^2}, \sqrt{8 - r^2}, 2, \sqrt{2 + r^2}, \pm r i e^{-i\theta}, r e^{i\theta}\right] \in M \ \right| \theta \in [0, 2\pi[, \ r \in \left[0, \sqrt{6} \right]\right\}
$$
\end{proposition}

This statement will be proven in Section \ref{section pinched}.


\section{{\bf The whereabouts of the fixed points of $\mathbf{F_t=(J, H_t)}$}} 


\label{section posRankZero}

In this section, we will determine the explicit coordinates of the eight fixed points of $F_t=(J, H_t)$ and hereby prove \refrankOpoints.

We only have to consider the case $t>0$ since $t=0$ is already treated in \refcoordEEPoints.
For $t>0$, the family $F_t=(J, H_t)$ will turn out to be of toric type or semitoric --- apart from the two transition times where the four transition points pass through a degeneracy.

\begin{proof}[Proof of \refrankOpoints]
$p=[p_1, \dots, p_8] \in M$ can only be a fixed point of $F_t = (J, H_t)$, if it is a fixed point of $J$. Thus, according to \refcoordEEPoints, $p$ lies in $ \{A, B, C, D\}$ or it satisfies $p_1 = 0$ or $p_5 = 0$.

{\bf Case $\mathbf{p \in \{A, B, C, D\}}$:}
If $p$ is one of these fixed points, $p$ is also a fixed point of $H$ according to \refcoordEEPoints. Write $z_k = x_k + i y_k $ and recall $X(z_1, \dots, z_8) = \mfR(\overline{z_2}\ \overline{z_3}\ \overline{z_4} \ z_6 \ z_7 \ z_8)  $. Let $\zhat_k$ stand for the omission of $z_k$ resp.\ $\overline{z_k}$ and define $\sign(\zhat_k)=1$ if $\zhat_k=z_k$ and $\sign(\zhat_k)=-1$ if $\zhat_k=\overline{z_k}$. We calculate
\begin{align*}
\del_{x_k}X(z_1, \dots, z_8) 
& =  \del_{x_k}\bigl(x_k \ \mfR(\overline{z_2}\ \cdots \zhat_k \cdots \ z_8) -\sign(\zhat_k) \ y_k \ \mfI(\overline{z_2}\ \cdots \zhat_k \cdots \ z_8)\bigr) \\
& =  \mfR(\overline{z_2}\ \cdots \zhat_k \cdots \ z_8)
\end{align*}
and 
\begin{align*}
\del_{y_k}X(z_1, \dots, z_8) & = -\sign(\zhat_k) \ \mfI(\overline{z_2}\ \cdots \zhat_k \cdots \ z_8).
\end{align*}
Note that each of the points $A,B,C,D$ has two vanishing coordinate entries, thus we find for $p \in \{A, B, C, D\}$
$$ 
dX(p) = \left(\partial_{x_1} X(p), \partial_{y_1} X(p), \partial_{x_2} X(p), \ldots, \partial_{y_8} X (p) \right) = (0, 0, 0,  \ldots, 0).
$$
Thus $p$ is also a fixed point of $X$. Since $H_t = (1-2t)H + t\gamma X$ is a linear combination, $p$ is also a fixed point of $H_t$.

{\bf Case: $\mathbf p$ has ${\mathbf p_1 = 0}$:}
First we show that any fixed point $p$ with $p_1=0$ has to lie in the set $U_1 \subseteq M$. We argue by contradiction: $p_1=0$ would also allow $p \in U_8$, i.e., $p_8=0$. By definition of $U_8$, we have $z_2, \dots, z_7 \neq 0$ and we may choose a representative with nonvanishing real part, thus $\del_{x_8}X(p) =Re(\overline{p_2} \: \overline{p_3} \:\overline{p_4}\: p_6 \: p_7) \neq 0$ and thus the derivative of $X$ does not vanish in $p$. Picking in particular $p \in U_8$ to be the fixed point 
$P^{max}= [0, \ \sqrt{2}, \ 2, \ 2\sqrt{2}, \ \sqrt{6}, \ \sqrt{6}, \ \sqrt{2}, \ 0]$ of the toric system $F_0=(J, H)$, we find 
$$dH_t(p) = (1-2t) dH(p) + t\gamma dX(p) = 0 + t\gamma d X(p) \neq 0$$
for $t \neq 0$. This shows that any fixed point $p$ of $F_t$ with $p_1=0$ has to lie in $U_1$. 

Let us now work in the coordinate chart $(U_1, \psi_1)$. By assumption, we have $0= z_1= x_1+iy_1$ and we now determine the coordinates $z_2 = x_2 +i y_2$ of any possible fixed point $p=\psi_1(0, 0, x_2, y_2) \in U_1$. We already calculated $H\circ \psi_1^{-1}$ in \eqref{HinCoord} and, plugging in $(0, 0, x_2, y_2)$, yields
$$
(H \circ \psi_1^{-1})(0, 0, x_2, y_2) = \frac{1}{2} \left( 2 + x_2^2 + y_2^2 \right) =  \frac{1}{2} \left( 2 + \abs{z_2}^2 \right).
$$
We also calculated $X \circ \psi_1^{-1}$ in \eqref{XinCoord} implying
\begin{align*}
 & (X \circ \psi^{-1}_1)(0, 0, x_2, y_2)  \\
 & \quad  = x_2 \sqrt{ \left(2 + (x_2^2 + y_2 ^2)\right) \left( 6 + ( x_2^2 + y_2 ^2 )\right) \left( 8 - ( x_2^2 + y_2 ^2) \right) \left( 4  - ( x_2^2 + y_2 ^2 ) \right) \left(2  -( x_2^2 + y_2 ^2 )\right)}  \\
 & \quad = x_2 \sqrt{(2 + \vert z_2 \vert^2)(6 + \vert z_2 \vert^2)(8 - \vert z_2 \vert^2)(4 - \vert z_2 \vert^2)(2 - \vert z_2 \vert^2)}.
\end{align*}
Now set 
\begin{align*}
  & \mfg: \R^2 \to \R, &&  \mfg(x_2, y_2):= x_2^2 + y_2^2, \\
  & \mff: \R \to \R, &&  \mff(\ze):= (2+ \ze) (6+ \ze) (8- \ze) ( 4- \ze) (2-\ze).
\end{align*}
Then we get for $H_t = (1-2t)H +t \ga X$ in local coordinates
\begin{align*}
& (H_t \circ \psi^{-1}_1)(0, 0, x_2, y_2) \\
& \quad =  \frac{1-2t}{2} \left( 2 + \abs{z_2}^2 \right)+ \gamma t x_2 \sqrt{(2 + \vert z_2 \vert^2)(6 + \vert z_2 \vert^2)(8 - \vert z_2 \vert^2)(4 - \vert z_2 \vert^2)(2 - \vert z_2 \vert^2)} \\
& \quad =  \frac{1-2t}{2} \left( 2 + \mfg(x_2, y_2) \right) + \gamma t x_2 \sqrt{\mff(\mfg(x_2, y_2))}.
\end{align*}
Now we look for critical points, i.e., $(x_2, y_2)$ such that $d (H_t \circ \psi^{-1}_1)\vert_{(0, 0, x_2, y_2)} = 0$. We compute
\begin{eqnarray}
\label{dH_x2}
\partial_{x_2} (H_t\circ \psi^{-1}_1) = 0  &\Leftrightarrow& (1-2t)x_2 + \gamma t \sqrt{\mff(\mfg(x_2, y_2))} + \gamma t x_2^2 \frac{ \mff'(\mfg(x_2, y_2)) }{ \sqrt{\mff(\mfg(x_2, y_2))}} = 0
\end{eqnarray}
and
\begin{eqnarray}
\label{dH_y2}
\partial_{y_2} (H_t\circ \psi^{-1}_1)  = 0  &\Leftrightarrow& (1-2t)y_2 + \gamma t x_2 y_2 \frac{ \mff'(\mfg(x_2, y_2))  }{\sqrt{\mff(\mfg(x_2, y_2))}} = 0 .
\end{eqnarray}
First we show that \eqref{dH_x2} and \eqref{dH_y2} force $y_2 = 0$ if $t>0$: We argue by contradiction and assume that $y_2 \neq 0$. Then dividing the right hand side of \eqref{dH_y2} by $y_2$ yields
$$  \gamma t x_2 \frac{ \mff'(\mfg(x_2, y_2))  }{\sqrt{\mff(\mfg(x_2, y_2))}} = 2t-1 .$$
We substitute this in the right hand side of \eqref{dH_x2} and get
\begin{equation*}
0 = (1-2t)x_2 + \gamma t \sqrt{\mff(\mfg(x_2, y_2))} + x_2(2t-1) = \gamma t \sqrt{\mff(\mfg(x_2, y_2))}.
\end{equation*}
But this cannot be true since $\sqrt{\mff(\mfg(x_2, y_2))} = \frac{1}{x_2} (X \circ \psi_1^{-1})(0, 0, x_2, y_2)\stackrel{\eqref{XinCoord}}{=} x_3 \: x_4 \: x_6 \: x_7 \: x_8 \neq 0$ in $U_1$ and $\gamma \neq 0$ and $t > 0$. 

Thus we conclude $y_2 = 0$ which implies $\abs{z_2}^2 = x_2^2= \mfg(x_2, y_2)$ and $p= \psi^{-1}_1(0, 0, x_2, 0)$. Thus we still need to determine the possible values of $x_2$ to determine $p$ completely. This we approach by inserting $y_2=0$ on the right hand side of \eqref{dH_x2} and by multiplying it with $\sqrt{\mff(\mfg(x_2, 0))}=\sqrt{\mff\bigl(x_2^2\bigr)}$ which yields
\begin{equation}
 \label{charEqu}
(1-2t) \ x_2 \sqrt{\mff\bigl(x_2^2\bigr)} + \gamma\ t \ \mff\bigl(x_2^2\bigr) + \gamma\ t \ x_2^2 \ \mff'\bigl(x_2^2\bigr) = 0.
\end{equation}
Note that equation \eqref{charEqu} is equation \eqref{Rank0eqnLem} in \refrankOpoints. Using the manifold equations \eqref{manifold eqn}, we get a restriction for $\abs{x_2}$: Since $ z_1 = 0$ we get $\abs{z_5}^2 = 6 $ and furthermore $ \vert z_2 \vert^2 + \vert z_7 \vert^2 = 4$ and $ \vert z_7 \vert^2 = 2 + \vert z_6 \vert^2 \geq 2 $ and thus
$$ \abs{x_2}^2 = \vert z_2 \vert ^2 \leq 2. $$
Renaming $x_2$ as $u$ to drop the index and simplify notation, we are thus looking for (smooth) solutions $u: [0,1] \to [- \sqrt{2}, \sqrt{2}]$ solving $\mcF(t, u(t))=0$ where
$$
\mcF: [0,1] \times [- \sqrt{2}, \sqrt{2}] \to \R, \quad \mcF(t, u):= (1-2t)u \sqrt{\mff\bigl(u^2\bigr)} + \gamma t \mff\bigl(u^2\bigr) + \gamma t u^2 \mff'\bigl(u^2\bigr).
$$
A look at the formulas of $\mff$ and $\mcF$ shows immediately that $\mcF(0, \pm \sqrt{2})=0=\mcF(0,0)$. This are the only zeros for $t=0$ as the plot of the graph of $\mcF$ in Figure \ref{Fig_Rank0eqn1} shows. Moreover, the zero $(0, - \sqrt{2})$ is isolated within $[0,1] \times  [- \sqrt{2}, \sqrt{2}]$ whereas $(0,0)$ and $(0,\sqrt{2})$ are the limit for $t \to 0$ of two unique smooth curves $t \mapsto (t, u_-(t))$ and $ t \mapsto (t, u_+(t))$ that satisfy $\mcF(t, u_\pm(t))=0$ and $u_-(t) \in\ ]- \sqrt{2}, 0]$ and $u_+(t) \in\ ] 0, \sqrt{2}]$ for $t \in [0,1]$ as displayed in Figure \ref{Fig_Rank0eqn2}.

\begin{figure}[h!]
\centering
\begin{subfigure}{0.33\textwidth}
\centering
\includegraphics[width=1.3\linewidth]{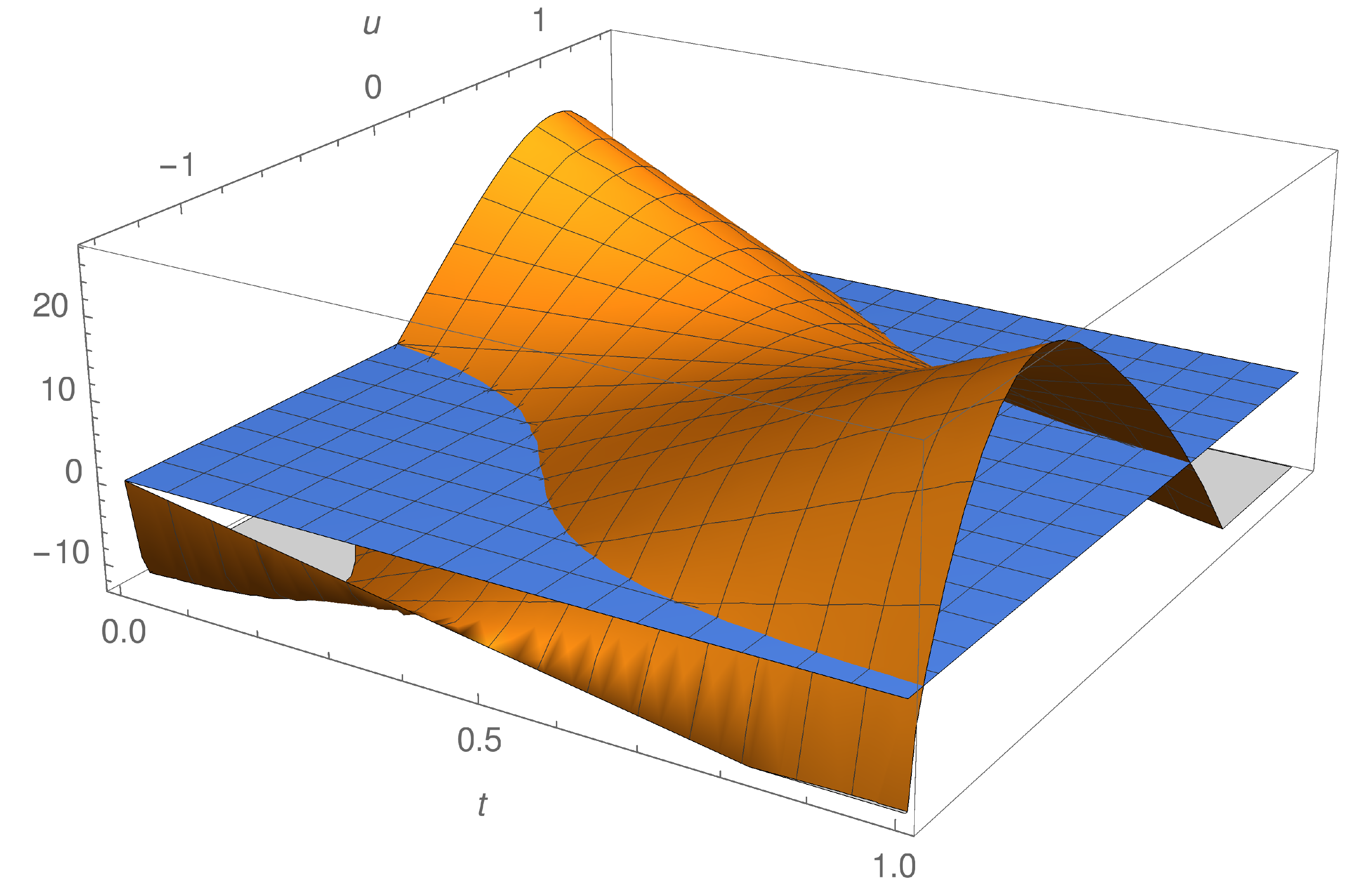}
\end{subfigure}%
\hspace{20mm}
\begin{subfigure}{0.33\textwidth}
\centering
\includegraphics[width=1.0\linewidth]{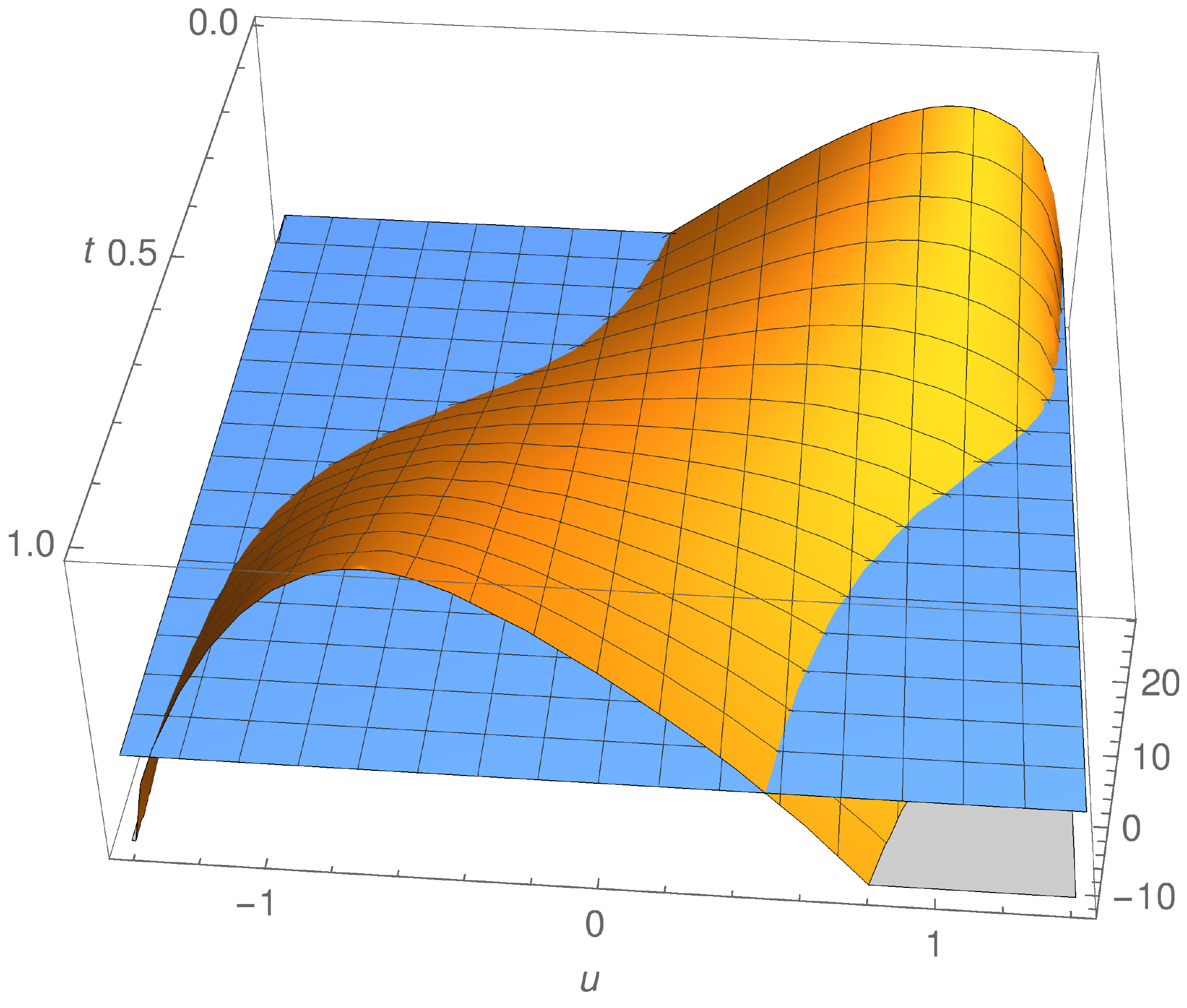}
\end{subfigure}%

\caption{The graph {\em (orange)} of $(t, u) \mapsto \mcF(t,u)$  intersected with a horizontal plane through zero {\em (blue)} seen from two different angles. Both plots are done with {\em Mathematica} for $\gamma = \frac{1}{50}$.}
\label{Fig_Rank0eqn1}
\end{figure}

\noindent
Letting $\ga>0$ tend to zero leads to a slimmer and slimmer `waist' of the graph of $\mcF$ and a steeper and steeper slope of $u_\pm$ as displayed in Figures \ref{Fig_Rank0eqn3} and \ref{Fig_Rank0eqn4}.

\begin{figure}[h!]
\centering
\begin{subfigure}{0.33\textwidth}
\centering
\includegraphics[width=.9\linewidth]{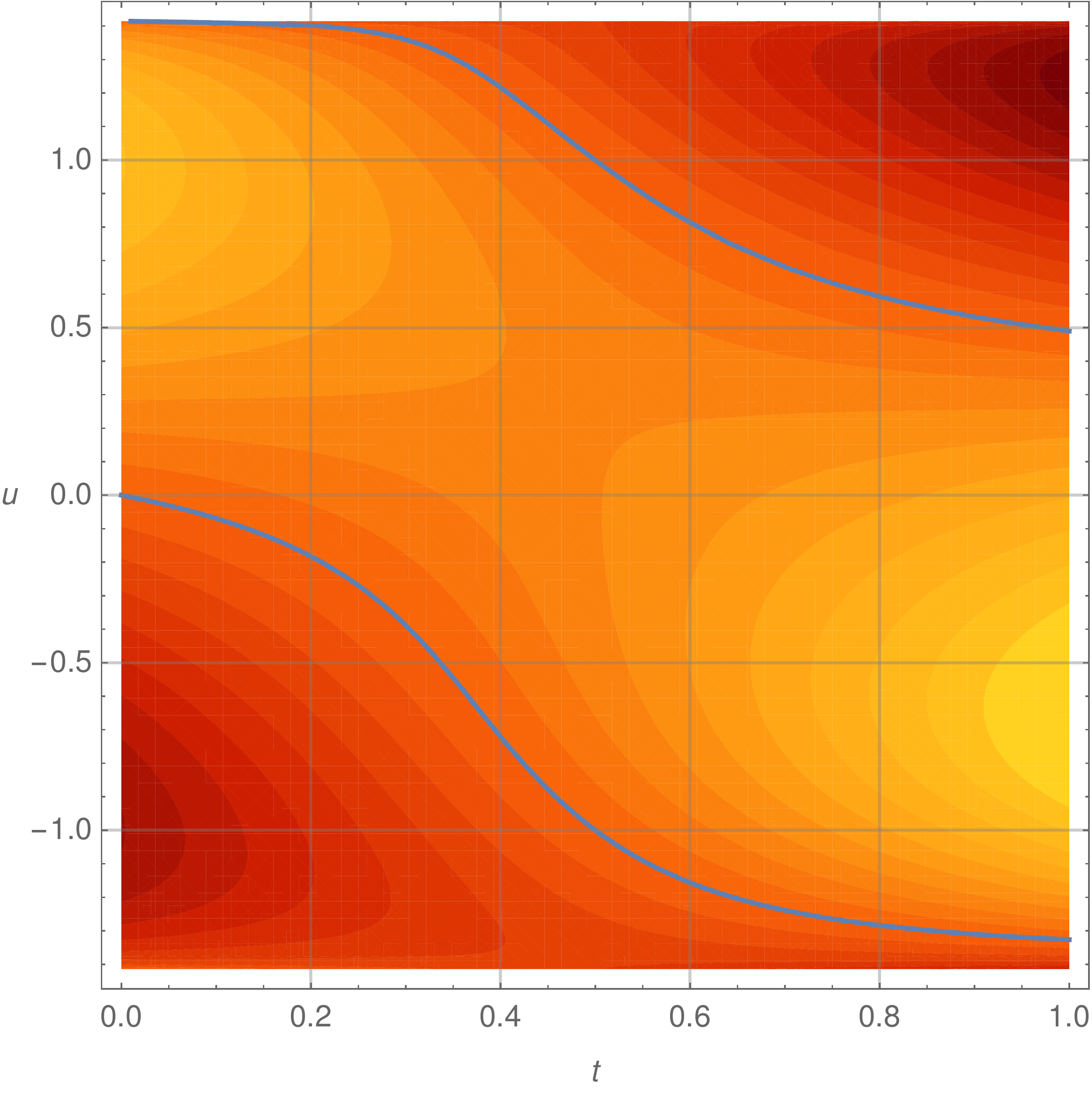}
\end{subfigure}%
\begin{subfigure}{0.33\textwidth}
\centering
\includegraphics[width=.9\linewidth]{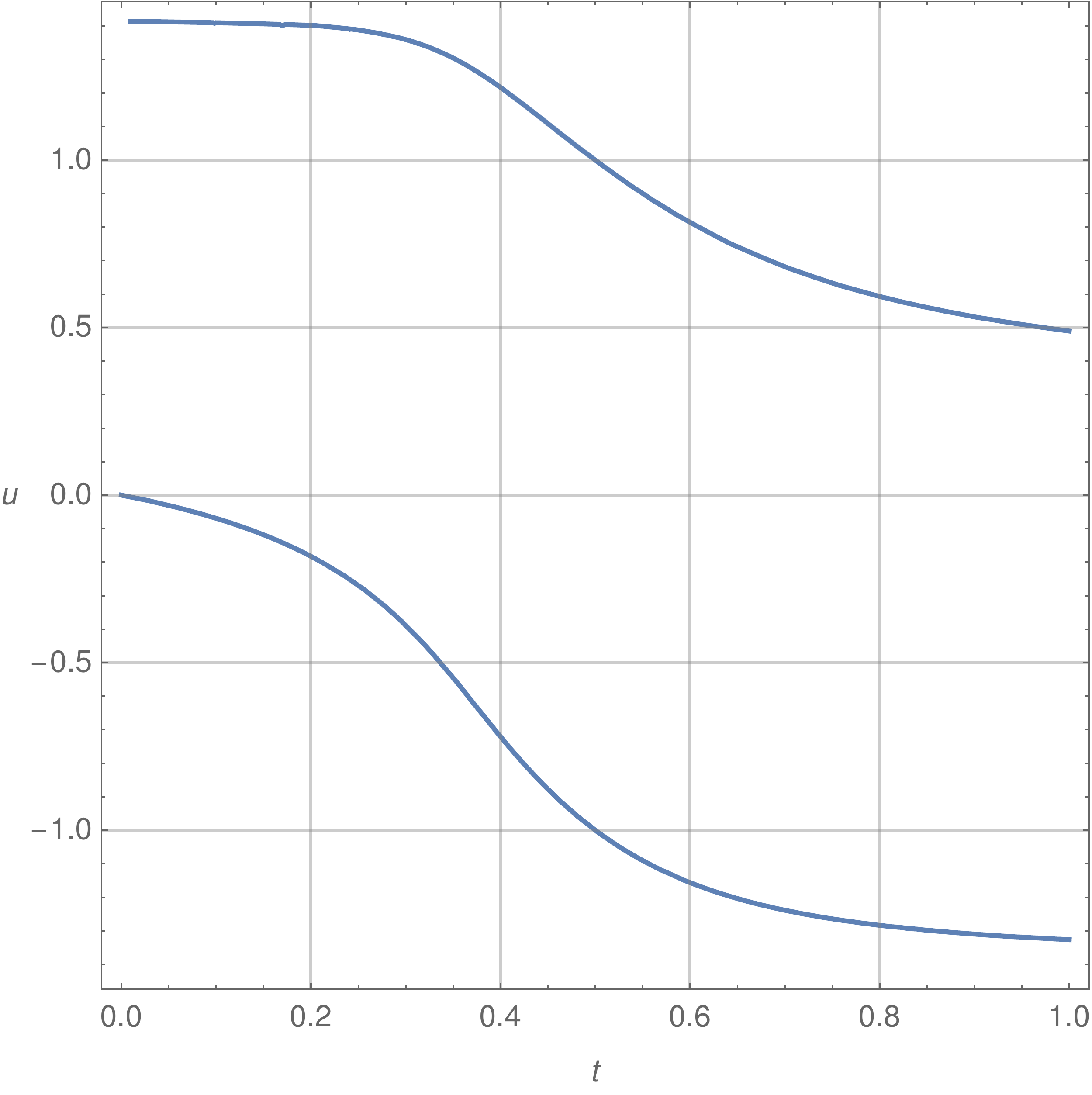}
\end{subfigure}

\caption{{\em Left:} Plot of the level sets of $\mcF(t, u)$ with the two solutions {\em (blue)} satisfying $\mcF(t, u_\pm(t))=0$. {\em Right:} The positive solution $u_+$ and the negative solution $u_-$ of $\mcF(t, u)=0$. Both plots are done with {\em Mathematica} for $\gamma = \frac{1}{50}$.}
\label{Fig_Rank0eqn2}
\end{figure}

\begin{figure}[h!]
\centering
\begin{subfigure}{0.33\textwidth}
\centering
\includegraphics[width=1.2\linewidth]{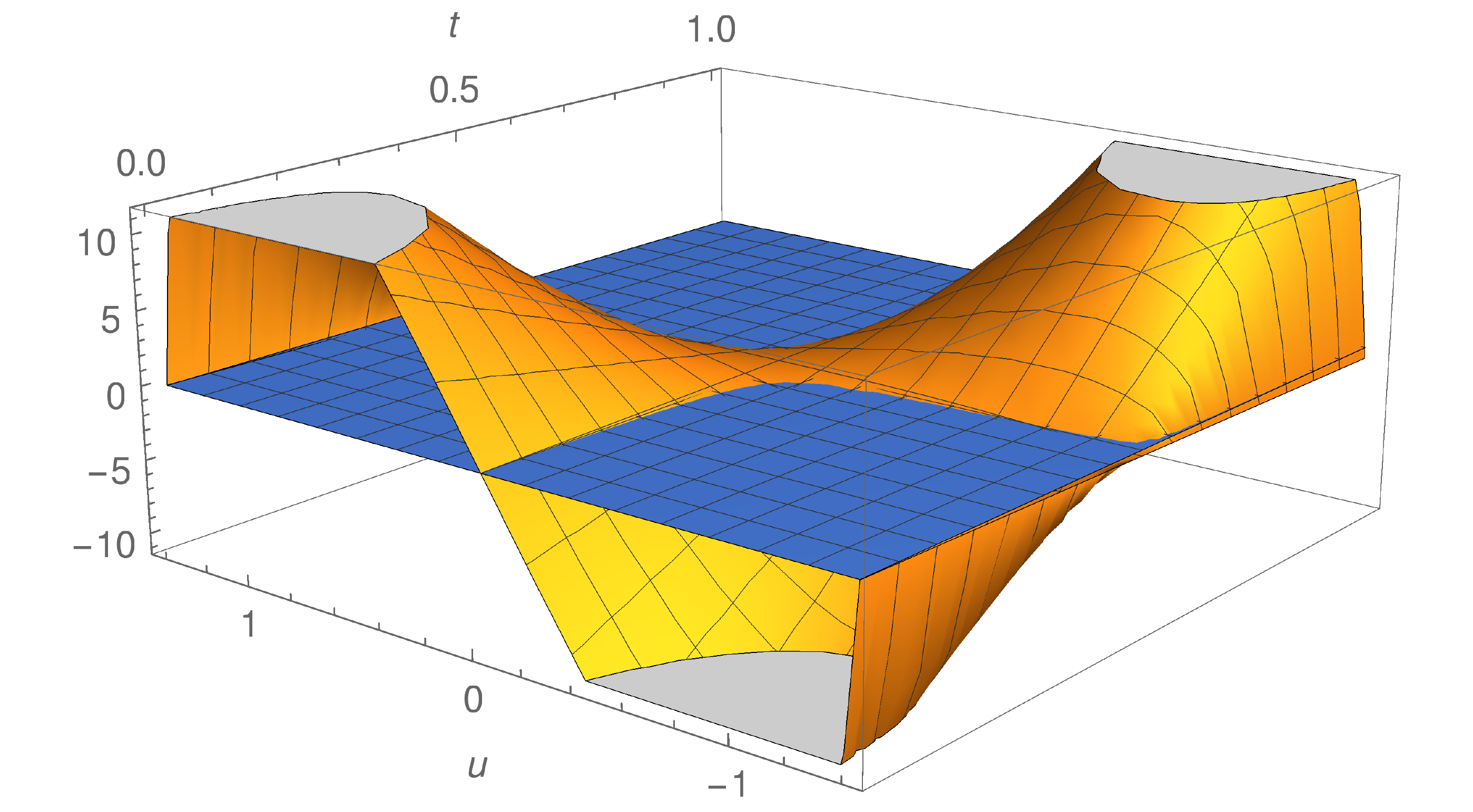}
\end{subfigure}%
\hspace{7mm}
\begin{subfigure}{0.33\textwidth}
\centering
\includegraphics[width=.9\linewidth]{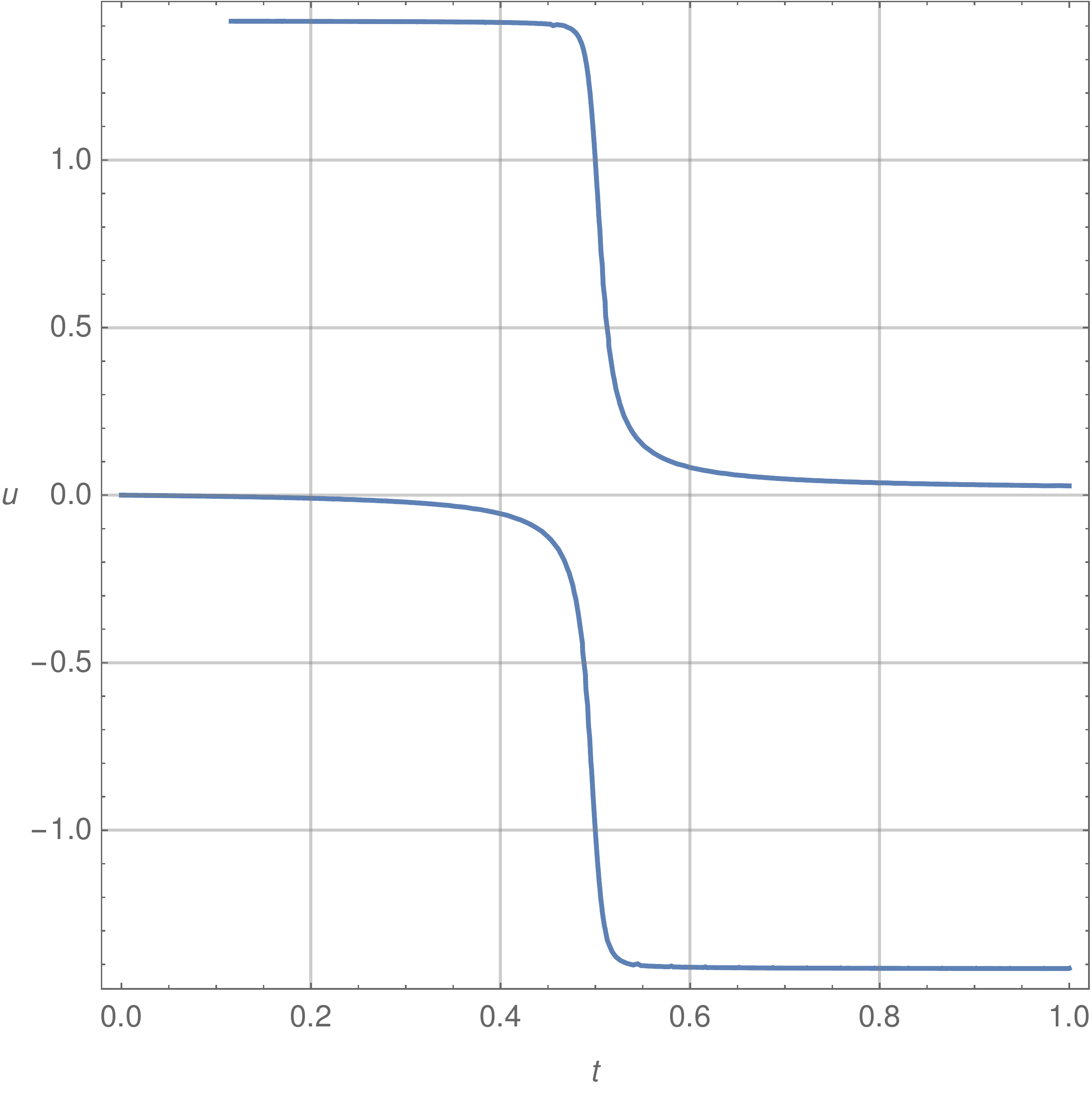}
\end{subfigure}

\caption{{\em Left:} The graph {\em (orange)} of $(t, u) \mapsto \mcF(t,u)$  intersected with a horizontal plane through zero {\em (blue)} seen from two different angles. {\em Right:} The positive solution $u_+$ and the negative solution $u_-$ of $\mcF(t, u)=0$. Both plots are done with {\em Mathematica} for $\gamma = \frac{1}{1000}$.}
\label{Fig_Rank0eqn3}
\end{figure}

\begin{figure}[h!]
\centering
\begin{subfigure}{0.33\textwidth}
\centering
\includegraphics[width=1.1\linewidth]{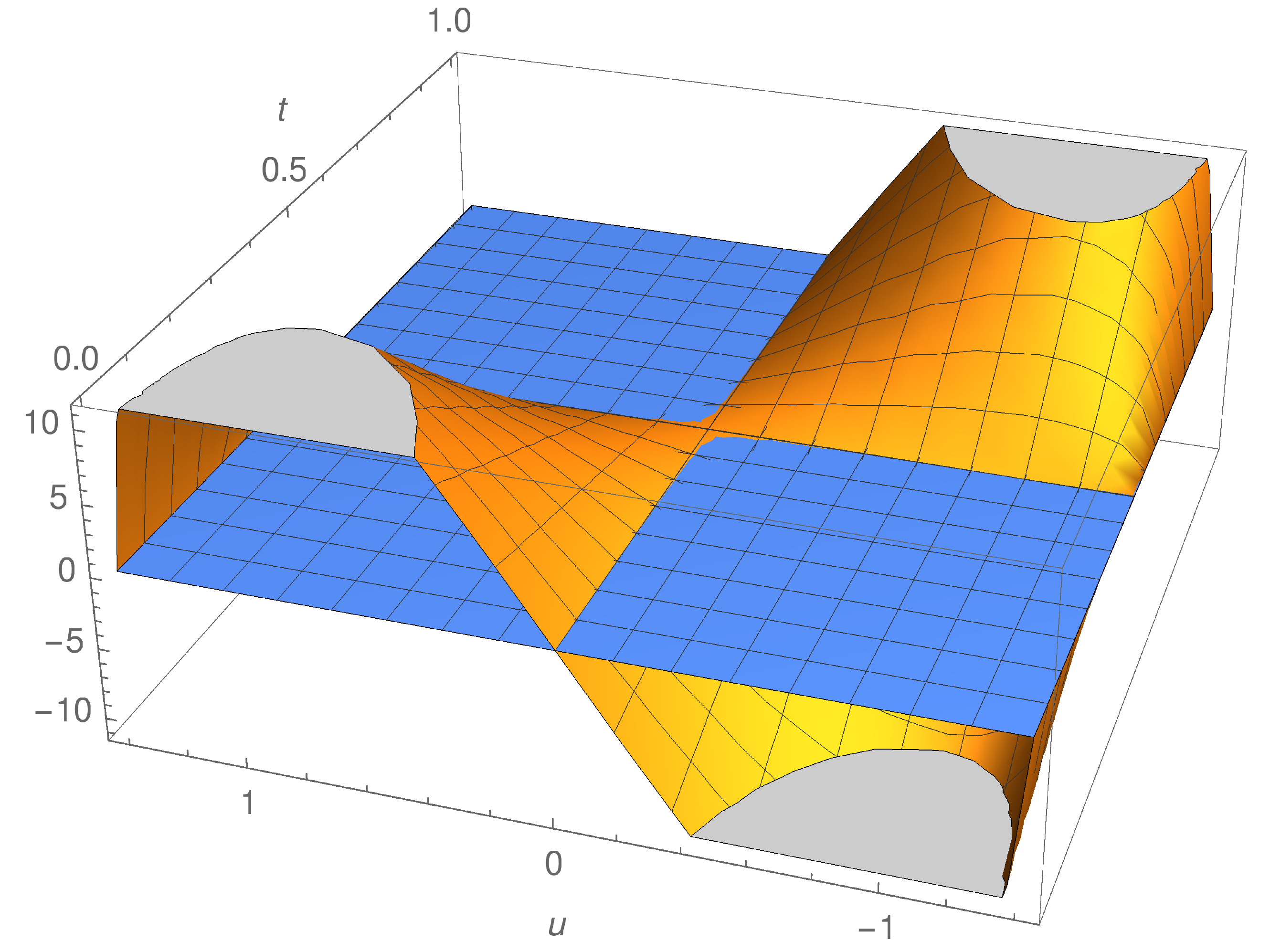}
\end{subfigure}%
\hspace{7mm}
\begin{subfigure}{0.33\textwidth}
\centering
\includegraphics[width=.9\linewidth]{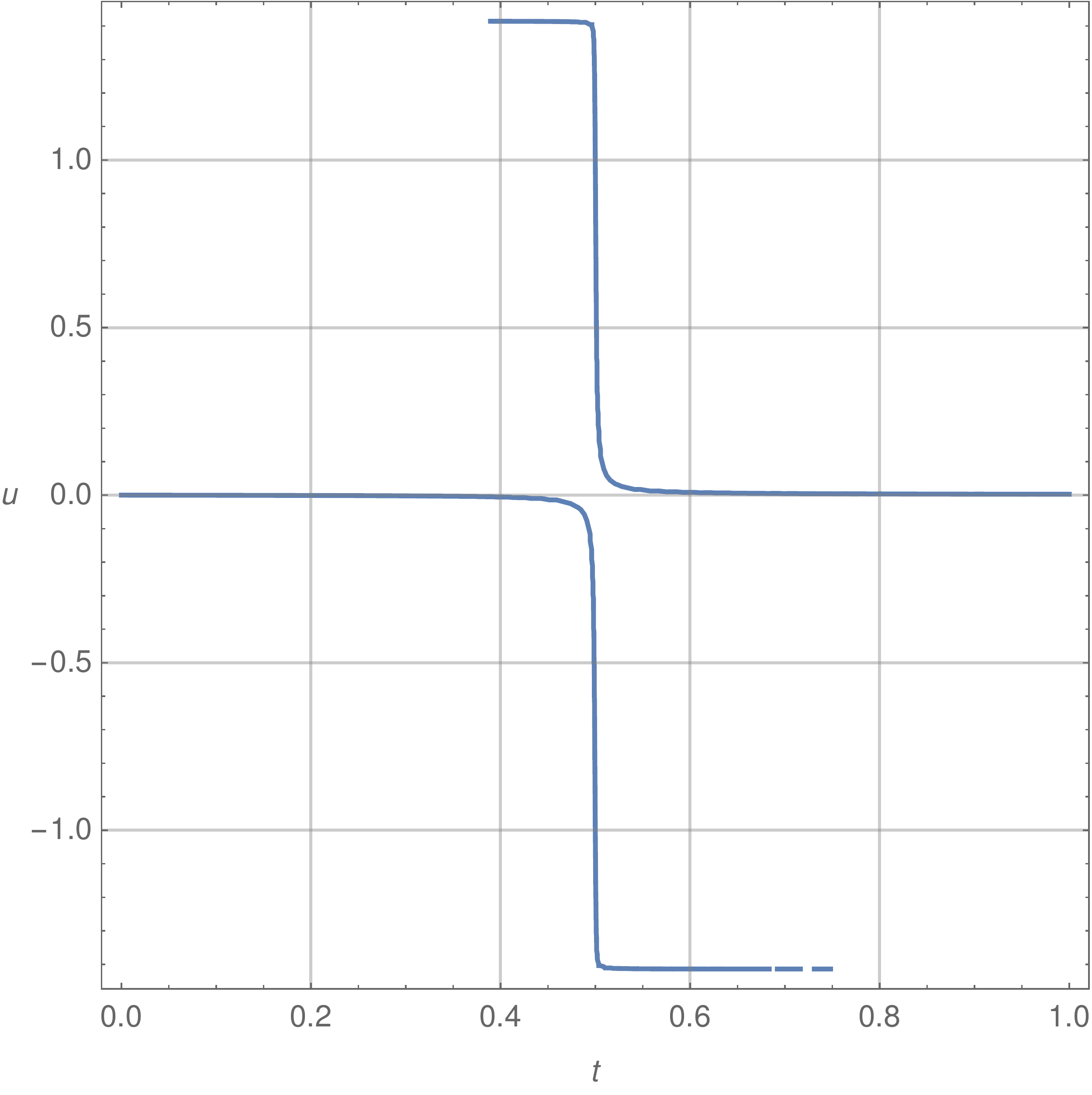}
\end{subfigure}

\caption{{\em Left:} The graph {\em (orange)} of $(t, u) \mapsto \mcF(t,u)$  intersected with a horizontal plane through zero {\em (blue)} seen from two different angles. {\em Right:} The positive solution $u_+$ and the negative solution $u_-$ of $\mcF(t, u)=0$. Both plots are done with {\em Mathematica} for $\gamma = \frac{1}{10 000}$.}
\label{Fig_Rank0eqn4}
\end{figure}

\noindent
Substituting the variable $x_2$ in \eqref{charEqu} with $u_{\pm}$, we obtain the fixed points 
\begin{align*}
 P^{min}_t  :=\psi_1^{-1}(0,0,u_-(t), 0) \quad \mbox{and} \quad  P^{max}_t & :=\psi_1^{-1}(0,0,u_+(t), 0).
\end{align*}
Using the definition of $\psi_1^{-1}$, i.e., the formulas in \eqref{sixVarEq}, we obtain the coordinates of $P^{min}_t$ and $P^{max}_t $ as given in \refrankOpoints. Letting $t \to 0$, we recover $P^{max}=P^{max}_0$ and $P^{min}=P^{min}_0$ of the toric system, as computed in \refcoordEEPoints.

{\bf Case: $\mathbf p$ has $\mathbf{p_5 = 0}$:}
We proceed analogously to the case $p_1 = 0$, i.e., first, we show $p$ to live in $U_5$ and then we work in the chart
$$\psi_5^{-1}(x_5, y_5, x_6, y_6)= [x_1, 0, x_2, 0, x_3, 0, x_4, 0, x_5, y_5, x_6, y_6, x_7, 0, x_8, 0]$$
where $x_1, x_2, x_3, x_4, x_7, x_8$ are determined by the manifold equations from \eqref{manifold eqn} as
\begin{align*}
x_1 &= \sqrt{6 - \vert z_5 \vert^2}, & x_3 &= \sqrt{4 + \vert z_5 \vert^2 - \vert z_6 \vert^2}, & x_7 &= \sqrt{2 - \vert z_5 \vert^2 + \vert z_6 \vert^2}, \\
x_2 &= \sqrt{8 - \vert z_6 \vert^2},  & x_4 &= \sqrt{2 + 2\vert z_5 \vert^2 - \vert z_6 \vert^2},  &  x_8 &= \sqrt{6 - 2\vert z_5 \vert^2 + \vert z_6 \vert^2}
\end{align*}
where $\vert z_5 \vert^2 = x_5^2 + y_5^2$ and $\vert z_6 \vert^2 = x_6^2 + y_6^2$.
We find $ p = \psi_5^{-1}(0,0,x_6,y_6) $ and
$$ H_t(p) = \frac{1-2t}{2}(4 - \vert z_6 \vert^2) + \gamma t x_6 \sqrt{(8 - \vert z_6 \vert^2)(4 - \vert z_6 \vert^2)(2 - \vert z_6 \vert^2)(2 + \vert z_6 \vert^2)(6 + \vert z_6 \vert^2)} .$$
The system of equations describing $d( H_t \circ \psi_5^{-1})(0,0,x_6, y_6) = 0$ leads to $y_6 = 0$ and forces $x_6$ to satisfy for $t>0$
\begin{equation}
 \label{charEqu2}
- (1-2t) \ x_6 \sqrt{\mff\bigl(x_6^2\bigr)} + \gamma\ t \ \mff\bigl(x_6^2\bigr) + \gamma\ t \ x_6^2 \ \mff'\bigl(x_6^2\bigr) = 0.
\end{equation}
$z_5 = 0$ and the manifold equations \eqref{manifold eqn} imply $ \vert z_6 \vert^2 = \vert z_7 \vert^2 - 2$ and $ \vert z_7 \vert^2 = 4 - \vert z_4 \vert^2 \leq 4 $ which leads to $ \vert z_6 \vert^2 = x_6^2 \leq 2$.

\begin{figure}[h!]
\centering
\begin{subfigure}{0.33\textwidth}
\centering
\includegraphics[width=1.3\linewidth]{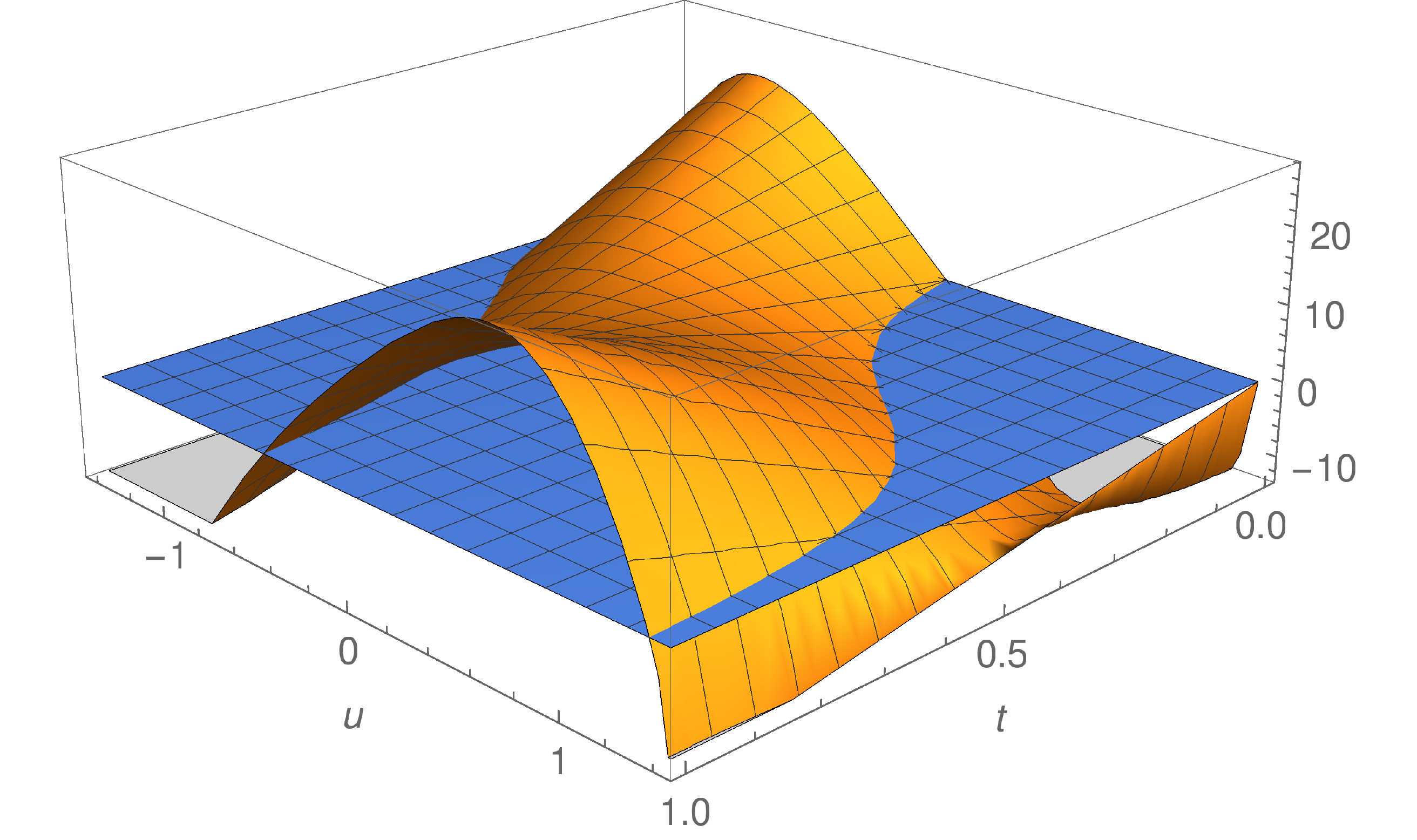}
\end{subfigure}%
\hspace{20mm}
\begin{subfigure}{0.33\textwidth}
\centering
\includegraphics[width=.9\linewidth]{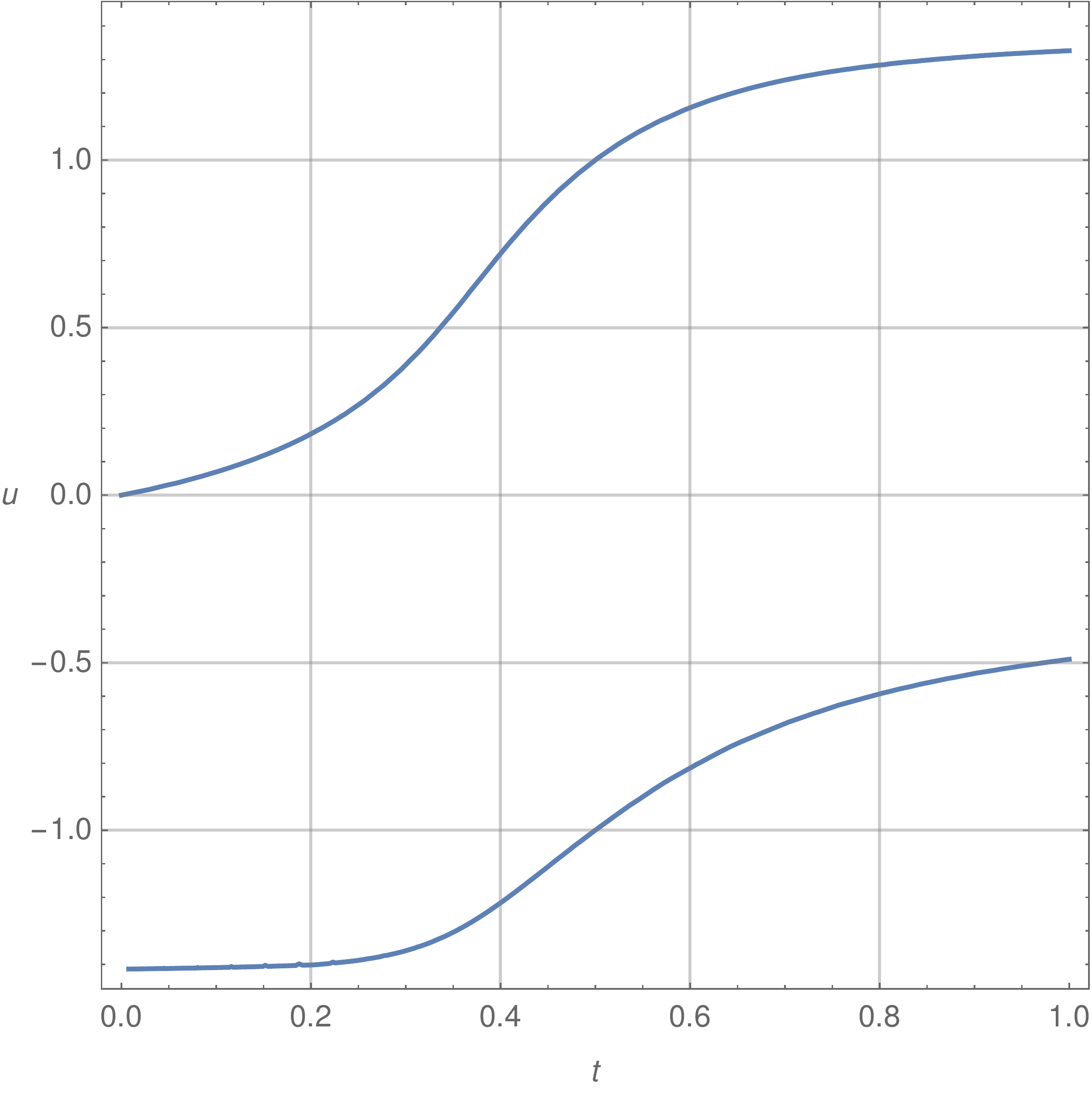}
\end{subfigure}

\caption{{\em Left:} The graph {\em (orange)} of $(t, v) \mapsto \widetilde{\mcF}(t,v)$ intersected with a horizontal plane through zero {\em (blue)}. {\em Right:} The positive solution $v_+$ and the negative solution $v_-$ of $\widetilde{\mcF}(t, v)=0$. Both plots are done with {\em Mathematica} for $\gamma = \frac{1}{50}$.}
\label{Fig_Rank0eqn5}
\end{figure}

Renaming $x_6$ as $v$, we are thus looking for smooth solutions $v: [0,1] \to [- \sqrt{2}, \sqrt{2}]$ solving $\widetilde{\mcF}(t, v(t))=0$ where
$$
\widetilde{\mcF}: [0,1] \times [- \sqrt{2}, \sqrt{2}] \to \R, \quad \widetilde{\mcF}(t, v):= -(1-2t)v \sqrt{\mff\bigl(v^2\bigr)} + \gamma t \mff\bigl(v^2\bigr) + \gamma t v^2 \mff'\bigl(v^2\bigr).
$$
The situation is plotted in Figure \ref{Fig_Rank0eqn5}: similar as above, we obtain two smooth solutions $v_-:[0,1] \to [- \sqrt{2}, 0[$ and $v_+:[0,1] \to [0, \sqrt{2}[$ satisfying $\widetilde{\mcF}(t, v_\pm(t))=0$ and $v_-(0)=- \sqrt{2}$ and $v_+(0) = 0$. This in turn gives us the fixed points 
\begin{align*}
 Q^{min}_t  :=\psi_5^{-1}(0,0,v_-(t), 0) \quad \mbox{and} \quad  Q^{max}_t & :=\psi_5^{-1}(0,0,v_+(t), 0).
\end{align*}
Using the definition of $\psi_5^{-1}$, we obtain the coordinates of the $Q^{min}_t$ and $Q^{max}_t $ given in \refrankOpoints. Moreover, for $t \to 0$, we recover $Q^{max}=Q^{max}_0$ and $Q^{min}=Q^{min}_0$ of the toric system, see \refcoordEEPoints.
\end{proof}


\section{{\bf The types of the fixed points of $\mathbf{F_t=(J, H_t)}$}}

\label{section typeRankZero}

In this section, we will show that $P_t^{min}, P_t^{max}, Q_t^{min}, Q_t^{max}$ are elliptic-elliptic singular points for all $t \in [0,1]$ and that $A, B, C, D$ are turning from elliptic-elliptic to focus-focus at time $t^-$ and back to elliptic-elliptic at $t^+$ passing through a degeneracy at $t^\pm$.

\begin{proposition} 
\label{Aeeffee}
Let $ 0 < \gamma < \frac{1}{48}$ and set
$$ t^- := \frac{1}{2(1+24\gamma)} \quad \mbox{and} \quad t^+ := \frac{1}{2(1-24\gamma)}. $$
Then the rank zero points $A$, $B$, $C$, $D$ are of elliptic-elliptic type for $0 \leq t  < t^-$ as well as for $t^+ < t \leq 1$ and of focus-focus type for $t^- < t < t^+$.

Moreover, for $t = t^\pm$, the points $A$, $B$, $C$, $D$ are degenerate (see \refdegPoint).
\end{proposition}

\begin{proof}
{\it Step 1: Type of the point ${A}$:} 
Since the point $A$ has coordinates $A_7 = 0 = A_8$, we work with the chart $(U_7, \psi_7)$ where
$$
\psi_7^{-1} : \R^4 \to U_7, \quad \psi^{-1}_7(x_7, y_7, x_8, y_8) = [x_1, 0, x_2, 0, \ldots, x_7, y_7, x_8, y_8]
$$
with $\vert z_7 \vert^2 = x_7^2 + y_7^2$ and $\vert z_8 \vert^2 = x_8^2 + y_8^2$ and
\begin{align*}
x_1 &= \sqrt{2 - \vert z_7 \vert^2 + \vert z_8 \vert^2}, &&&x_3 &= \sqrt{6 - \vert z_7 \vert^2}, &&& x_5 &= \sqrt{4 + \vert z_7 \vert^2 - \vert z_8 \vert^2}, \\
x_2 &= \sqrt{6 - 2\vert z_7 \vert^2 + \vert z_8 \vert^2} , &&& x_4 &= \sqrt{8 - \vert z_8 \vert^2}, &&& x_6 &= \sqrt{2 + 2\vert z_7 \vert^2 - \vert z_8 \vert^2}.
\end{align*}
 We have $A = \psi_7^{-1}(0,0,0,0)$ in this chart and we find
\begin{align}
\notag
 & (H_t \circ \psi^{-1}_7)(x_7, y_7, x_8, y_8)   = \frac{(1-2t)}{2}(6 - \vert z_7 \vert^2)  \\ \notag
& \qquad \quad + t\gamma (x_7x_8 - y_7y_8) \sqrt{(6 - 2\vert z_7 \vert^2 + \vert z_8 \vert^2) (6 - \vert z_7 \vert^2)(8-\vert z_8 \vert^2)(2 + 2\vert z_7 \vert^2 -\vert z_8 \vert^2)} \\ 
\label{klm1}
& =: \mcK(x_7, y_7, x_8, y_8) + \mcL(x_7, y_7, x_8, y_8) \sqrt{\mcM(x_7, y_7, x_8, y_8)}.
\end{align}
To determine the Hessian $d^2(H_t \circ \psi^{-1}_7)(0,0,0,0)$, we calculate 
\begin{align*}
 d^2 \mcK(x_7, y_7, x_8, y_8) =
 \begin{pmatrix}
  -(1-2t) & 0 & 0 &0 \\
  0 & -(1-2t) & 0 &0  \\
  0 & 0& 0 & 0 \\
  0 & 0& 0 & 0
 \end{pmatrix}
\end{align*}
and, for $\si, \tau \in \{x_7, y_7, x_8, y_8\}$, we get
\begin{align}
\notag
 \del_\si(\mcL \sqrt{\mcM}) & = (\del_\si \mcL) \sqrt{\mcM} + \mcL \ \frac{1}{2 \sqrt{\mcM}} \ \del_\si \mcM, \\
 \label{klm2}
 \del^2_{\si \tau} (\mcL \sqrt{\mcM}) & = (\del^2_{\si \tau} \mcL) \sqrt{\mcM} + (\del_\si \mcL) \ \frac{1}{2 \sqrt{\mcM}} \ \del_\tau \mcM + (\del_\tau \mcL ) \ \frac{1}{2 \sqrt{\mcM}} \ \del_\si \mcM  \\
 \notag
 & \quad + \mcL \ \del_\tau\left(\frac{1}{2 \sqrt{\mcM}} \right)\ \del_\si \mcM + \mcL \ \frac{1}{2 \sqrt{\mcM}} \ \del^2_{\si\tau} \mcM.
\end{align}
We calculate
\begin{align*}
d\mcL(x_7, y_7, x_8, y_8) = ( t \ga x_8, - t \ga y_8, t \ga x_7, -t \ga y_7 ), 
\quad 
  d^2 \mcL(x_7, y_7, x_8, y_8) =
 \begin{pmatrix}
  0 & 0 & t\ga &0 \\
  0 & 0 & 0 & - t\ga  \\
  t \ga & 0& 0 & 0 \\
  0 & t \ga & 0 & 0
 \end{pmatrix}.
\end{align*}
Now note that $\del_\si \mcL$ and $\del_\si \mcM$ and $\del^2_{\si, \tau} \mcM$ contain at least one of the variables $x_7, y_7, x_8, y_8$ as a factor and that $\mcL(0,0,0,0)=0$ and $\sqrt{\mcM(0,0,0,0)}=24$. This leads to 
\begin{align*}
 d^2(\mcL \sqrt{\mcM})(0,0,0,0)= 
 \begin{pmatrix}
  0 & 0 & 24t\ga &0 \\
  0 & 0 & 0 & - 24t\ga  \\
  24 t \ga & 0& 0 & 0 \\
  0 & - 24 t \ga & 0 & 0
 \end{pmatrix}
\end{align*}
and therefore we get at $\psi_7(A) = (0,0,0,0)$
$$ d^2(H_t \circ \psi_7^{-1})(0,0,0,0) = \begin{pmatrix}
2t-1 & 0 & 24t\gamma & 0 \\
0 & 2t-1 & 0 & -24t\gamma \\
24t\gamma & 0 & 0 & 0 \\
0 & -24t\gamma & 0 & 0
\end{pmatrix}. $$
According to \eqref{omIsStandard}, in the chart $(U_7, \psi_7)$, the symplectic form is denoted by $\om_7$ and shown to be the standard symplectic form in \eqref{omIsStandard}. We get
\begin{eqnarray*}
(\om_7)^{-1}_{(0,0,0,0,)} d^2 ( H_t \circ \psi_7^{-1}) {(0,0,0,0)} = \begin{pmatrix}
0 & 2t-1 & 0 & -24t\gamma \\
1-2t & 0 & -24t\gamma & 0 \\
0 & -24t\gamma & 0 & 0 \\
-24t\gamma & 0 & 0 & 0
\end{pmatrix}. \\
\end{eqnarray*}
{\it First, consider the case ${t \neq \frac{1}{2}}$:} Set $c := 24\gamma$ and compute the characteristic polynomial $\chi$ of $\om_7^{-1} d^2(H_t \circ \psi^{-1}_7)$ in $(0, 0, 0, 0)$ as
$$ \chi(\xi) = \xi^4 + (1 - 4t + 2(2-c^2)t^2)\xi^2 + c^4t^4. $$
Now look at the polynomial $\widetilde{\chi}$ defined by requiring $\widetilde{\chi}(\xi) = \chi(\xi^2)$, i.e., 
$$ \widetilde{\chi}(\xi) = \xi^2 + (1 - 4t + 2(2-c^2)t^2)\xi + c^4t^4. $$
$\widetilde{\chi}$ has discriminant $\de$ given by
\begin{eqnarray*}
\de := -(2t-1)^2(4(c^2-1)t^2 + 4t-1) 
= -4(c^2-1)(2t-1)^2 \left( t-\frac{1}{2(1-c)} \right) \left( t-\frac{1}{2(1+c)} \right).
\end{eqnarray*}
Setting
$$ t^-: = \frac{1}{2(1+c)}= \frac{1}{2(1+24\gamma)} \qquad \mbox{ and } \qquad t^+ :=\frac{1}{2(1-c)} = \frac{1}{2(1-24\gamma)} $$
yields
$$ \de = 4(1-c^2)(2t-1)^2(t-t^+)(t-t^-). $$
The assumption $0 < \gamma < \frac{1}{48}$ implies that $0 < t^- < \frac{1}{2} < t^+ < 1$ and in particular that $c < \frac{1}{2}$ and hence $1-c^2 > 0$. Moreover, since $t \neq \frac{1}{2}$, the factor $(2t-1)^2$ is strictly positive.
Thus the sign of the discriminant $\de$ depends on the position of $t$ relative to $t^-$ and $t^+$, leaving us with two cases:

{\it Consider the case ${ t^- < t < t^+}$:} Here we have $\de < 0$, so that $\widetilde{\chi}$ has two complex conjugate roots $a + ib$ and $a - ib$ with nonzero imaginary part $b$. Hence, the four zeros of $\chi$ are of the form $ \pm \alpha \pm i\beta$, i.e., $A$ is nondegenerate and of focus-focus type according to the list after \refnondeg.

{\it Consider the case ${0 < t < t^-}$ or ${t^+ < t \leq 1}$:} Here we have $\de > 0$ so that $\widetilde{\chi}$ has two real roots
$$ \lambda_\pm := \frac{-(1-4t+2(2-c^2)t^2) \pm \sqrt{\de}}{2}. $$
We have
$$ \bti: = 2(2-c^2)t^2 - 4t + 1 = 2(2-c^2) \left(t-\frac{1}{2-c\sqrt{2}} \right) \left(t-\frac{1}{2+c\sqrt{2}} \right) $$
which is positive since
$$ t^- < \frac{1}{2+c\sqrt{2}} < \frac{1}{2-c\sqrt{2}} < t^+ $$
and $t < t^-$ or $t^+ < t$. Hence we find $\lambda_- < 0$. Moreover, $ \bti > \sqrt{\de}$ since $b^2 - \Delta = 4c^4t^4 > 0$. This implies $\lambda_+ < 0$ as well. Therefore, the zeros of $\chi$ are of the form $\pm i\alpha, \pm i\beta$ and $A$ is nondegenerate of elliptic-elliptic type according to the list after \refnondeg.

{\it Second, consider the case ${ t = \frac{1}{2}}$:} Here we find $\de = 0$, so the eigenvalues are not distinct and we cannot conclude the type of $A$ in the same way as above. We now start with
$$ (\om_7^{-1})_{(0,0,0,0)} d^2(H_t \circ \psi^{-1}_7){(0,0,0,0)}
= \begin{pmatrix}
0 & 0 & 0 & -12\gamma \\
0 & 0 & -12\gamma & 0 \\
0 & -12\gamma & 0 & 0 \\
-12\gamma & 0 & 0 & 0
\end{pmatrix} $$
and consider in addition $(J \circ \psi_7^{-1})(x_7, y_7 x_8, y_8) = \frac{1}{2}(2 - x_7^2 - y_7^2 + x_8^2 + y_8^2)$. We find
$$ d^2(J \circ \psi_7^{-1}){(0,0,0,0)} = \begin{pmatrix}
-1 & 0 & 0 & 0 \\
0 & -1 & 0 & 0 \\
0 & 0 & 1 & 0 \\
0 & 0 & 0 & 1
\end{pmatrix}
$$
and thus
$$
(\om_7^{-1})_{(0,0,0,0)} d^2(J \circ \psi^{-1}_7){(0,0,0,0)} =  \begin{pmatrix}
0 & -1 & 0 & 0 \\
1 & 0 & 0 & 0 \\
0 & 0 & 0 & 1 \\
0 & 0 & -1 & 0
\end{pmatrix} .$$
The linear combination
$$ (\om_7)^{-1}_{(0,0,0,0)} d^2( J  \circ \psi^{-1}_7) (0,0,0,0) - \frac{1}{12\gamma}  (\om_7)^{-1}_{(0,0,0,0)} (d^2H_t \circ \psi^{-1}_7)(0,0,0,0) = \begin{pmatrix}
0 & -1 & 0 & 1 \\
1 & 0 & 1 & 0 \\
0 & 1 & 0 & 1 \\
1 & 0 & -1 & 0
\end{pmatrix} $$
has four distinct eigenvalues $\pm 1 \pm i$, so $A=\psi^{-1}_7(0,0,0,0)$ is nondegenerate of focus-focus type.

{\it Proof for the point ${B}$:} 
Note that $B = [\sqrt{2}, \ 0, \ 0, \ \sqrt{2}, \ 2, \ 2\sqrt{2}, \ \sqrt{6}, \ \sqrt{6}]$ lies in the chart $(U_2, \psi_2)$ given by
$$
\psi_2^{-1} : \R^4 \to U_2, \quad \psi_2^{-1}(x_2, y_2, x_3, y_3) = [x_1, 0, x_2, y_2, x_3, y_3, x_4, 0, \ldots, x_8, 0]
$$
with $\vert z_2 \vert^2 = x_2^2 + y_2^2$ and $\vert z_3 \vert^2 = x_3^2 + y_3^2$ and
\begin{align*}
x_1 &= \sqrt{2 + \vert z_2 \vert^2 - \vert z_3 \vert^2}, &&& x_5 &= \sqrt{4 - \vert z_2 \vert^2 + \vert z_3 \vert^2}, &&& x_7 &= \sqrt{6 - \vert z_3 \vert^2}, \\
x_4 &= \sqrt{2 - \vert z_2 \vert^2 + 2\vert z_3 \vert^2}, &&& x_6 &= \sqrt{8 - \vert z_2 \vert^2} , &&& x_8 &= \sqrt{6 + \vert z_2 \vert^2 - 2\vert z_3 \vert^2} .
\end{align*}
and is given by $B=\psi_2^{-1}(0,0,0,0)$. Proceeding as above for the fixed point $A$, we find 
$$ 
(\om_2^{-1})_{(0,0,0,0)} d^2(H_t \circ \psi^{-1}_2){(0,0,0,0)}
= 
\begin{pmatrix}
0 & 0 & 0 & -24\gamma t \\
0 & 0 & -24\gamma t & 0 \\
0 & -24\gamma t & 0 & 1-2t \\
-24 \gamma t & 0 & 2t-1 & 0
\end{pmatrix}. 
$$
This matrix has the same characteristic polynomial as the corresponding matrix in the calculations for the point $A$ above. Hence, all calculations and implications are identical to the ones made for the point $A$.

{\it Proof for the points ${C}$ and ${D}$:} Analogous to the ones for $A$ and $B$.
\end{proof}

We now study the type of the remaining four fixed points.

\begin{proposition}
\label{eeFixedPoints}
The points $P_t^{min}, P_t^{max}, Q_t^{min}$ and $Q_t^{max}$ are elliptic-elliptic for all $t \in [0, 1]$.
\end{proposition}

\begin{proof}
We showed in \refrankOpoints\ that we recover via $P_0^{min}=P^{min}$, $P_0^{max}=P^{max}$, $Q_0^{min}=Q^{min}$, $Q_0^{max}=Q^{max}$ the fixed points of the toric system $(M, \om, F=(J, H))$ which are all elliptic-elliptic. Thus it is sufficient to prove \refeeFixedPoints\ for $t>0$.

{\it First, consider ${P_t^{min}}$ and ${P_t^{max}}$:} 
As in the proof of \refrankOpoints, we work in the chart $(U_1, \psi_1)$. Here we have $\psi_1(P_t^{min}) = (0,0,u_-(t),0)$ and $\psi_1(P_t^{max}) = (0,0,u_+(t),0)$ where $t \mapsto u_{\pm}(t) \in [-\sqrt{2}, \sqrt{2}]$ are the two real solutions of \eqref{Rank0eqnLem}.
In the chart $(U_1, \psi_1)$ with identification $ x_1^2 + y_1^2=\abs{z_1}^2 $ and $ x_2^2 + y_2 ^2= \abs{z_2}^2 $, we get 
\begin{align*}
 & (H_t \circ \psi^{-1}_1)(x_1, y_1, x_2, y_2)  = \frac{1-2t}{2}(2-\vert z_1 \vert^2 + \vert z_2 \vert^2) \\
 & \quad + \gamma t x_2 \sqrt{(2 - \vert z_1 \vert^2 + \vert z_2 \vert^2)(6 - 2\vert z_1 \vert^2 + \vert z_2 \vert^2)(8 - \vert z_2 \vert^2)(4 + \vert z_1 \vert^2 - \vert z_2 \vert^2)(2 + 2\vert z_1 \vert^2 - \vert z_2 \vert^2)}
 \end{align*}
 and set analogously to \eqref{klm1}
\begin{align*} 
  =: \mcK(x_1, y_1, x_2, y_2) + \mcL(x_1, y_1, x_2, y_2) \sqrt{\mcM(x_1, y_1, x_2, y_2)}.
\end{align*}
 We need to calculate the Hessian of $H_t\circ \psi^{-1}_1$ in the point $(0,0,u,0)$ for $u:=u_\pm(t)$. 
All following calculations only care about the variables $x_1, y_1, x_2, y_2$ of the chart while holding true for all $t \in [0,1]$. Thus, for sake of readability, we drop the parameter $t$ in some of the expressions.
We find 
$$
d^2 \mcK(x_1, y_1, x_2, y_2) = 
\begin{pmatrix}
 2t-1 & 0  & 0 & 0 \\
 0 & 2t-1 & 0  & 0 \\
 0 & 0 & (1-2t) & 0 \\
 0 & 0 & 0 & (1-2t)
\end{pmatrix}
$$
and 
$$
\mcL(x_1, y_1, x_2, y_2) = \ga t x_2, \quad d\mcL(x_1, y_1, x_2, y_2) = (0, 0, \ga t, 0), \quad d^2 \mcL = \mathbf 0. 
$$
The formula for $\del^2_{\si \tau} (\mcL \sqrt{\mcM})$ is the same as in \eqref{klm2}. For $\si, \tau \in \{x_1, y_1, x_2, y_2\}$, note that $\del^2_{\si \tau} \mcL =0$ and that $\del_\si \mcL=0$ for $\si \neq x_2$ and $\del_{x_2} \mcL = \ga t$. We have 
\begin{align*}
 \mcM(u):=\mcM(0,0,u,0) & = (2+u^2)(6+u^2)(8-u^2)(4-u^2)(2-u^2) \\
 & = - u^{10}+ 6 u^8+ 44 u^6 - 216 u^4- 160 u^2 + 768
\end{align*}
and the first partial derivatives of $\mcM$ are of the form 
$$
\del_\si \mcM (x_1, y_1, x_2, y_2) = \si \ Rest(x_1^2+ y^2_1, x^2_2 + y_2^2)
$$
for $\si \in \{x_1, y_1, x_2, y_2\}$ such that $\del_\si \mcM(0,0,u,0)=0$ for $\si \in \{x_1, y_1, y_2\}$. We compute for $\si=x_2$
$$
\del_{x_2} \mcM(0,0,u,0) = - 10 u^9+ 48 u^7+ 264 u^5- 864 u^3-320 u = \frac{d}{du} \mcM(u)=:\mcM'(u).
$$
The mixed second partial derivatives are of the form
$$
\del_{\si \tau} \mcM (x_1, y_1, x_2, y_2) = \si \ \tau \ Rest(x_1^2+ y^2_1, x^2_2 + y_2^2)
$$
for all $\si, \tau \in \{x_1, y_1, x_2, y_2\}$ with $\si \neq \tau$ so that $\del_{\si \tau} \mcM(0,0,u,0)=0$ for all $\si, \tau \in \{x_1, y_1, x_2, y_2\}$ with $\si \neq \tau$. Thus a look at \eqref{klm2} shows that
$$
\del^2_{\si \tau} \left(\mcL \sqrt{\mcM} \ \right)(0,0,u,0)=0 \qquad \forall\ \si, \tau \in\{x_1, y_1, x_2, y_2\} \ \mbox{ with }  \si \neq \tau.
$$
To get the remaining second derivatives, we first compute
\begin{align*}
 \del^2_{x_1 x_1} \mcM (0,0,u,0) & = 12 u^8 - 84 u^6  - 256 u^4 + 1200 u^2 +  640 \\
 & = \del^2_{y_1 y_1} \mcM (0,0,u,0), \\
\del^2_{x_2 x_2} \mcM (0,0,u,0) & =- 90 u^8+ 336 u^6+ 1320 u^4- 2592 u^2-320, \\
\del^2_{y_2 y_2} \mcM (0,0,u,0) & =- 10 u^8 + 48 u^6+ 264 u^4- 864 u^2 -320  
\end{align*}
and thus
\begin{align*}
\del^2_{x_1 x_1} \left(\mcL \sqrt{\mcM} \ \right)(0,0,u,0) & = \frac{\ga \ t\  u \ (6 u^8 -42 u^6 - 128 u^4  + 600 u^2 + 320)}{\sqrt{\mcM(u)}} \\
& = \del^2_{y_1 y_1} \left(\mcL \sqrt{\mcM} \ \right)(0,0,u,0), \\
\del^2_{y_2 y_2} \left(\mcL \sqrt{\mcM} \ \right)(0,0,u,0) & = \frac{\ga \ t\  u \  (- 5 u^8 + 24 u^6 + 132 u^4 - 432 u^2 -160) }{\sqrt{\mcM(u)}} .
\end{align*}
Setting 
\begin{align}
\label{gFunction}
g(u)& := - 15 u^{18} + 153 u^{16}+ 652 u^{14}- 9216 u^{12} + 3984 u^{10}  + 126288 u^8 - 159552 u^6 \\ \notag
& \qquad\ - 510336 u^4+ 803840 u^2 + 184320
\end{align}
we obtain
\begin{align*}
\del^2_{x_2 x_2} \left(\mcL \sqrt{\mcM} \ \right)(0,0,u,0) & = 
\frac{ 2\ \ga\ t \ u \ g(u)}{\sqrt{\mcM^3(u)}}.
\end{align*}
Setting 
\begin{align*}
 a(u)& := 2t-1 + \frac{\ga \ t\  u \ (6 u^8 -42 u^6 - 128 u^4  + 600 u^2 + 320)}{\sqrt{\mcM(u)}}, \\
 b(u) & := (1-2t) +\frac{ 2\ \ga\ t \ u \ g(u)}{\sqrt{\mcM^3(u)}}, \\
 c(u) & := (1-2t) +\frac{\ga \ t\  u \  (- 5 u^8 + 24 u^6 + 132 u^4 - 432 u^2 -160) }{\sqrt{\mcM(u)}} ,
\end{align*}
we obtain
$$ 
d^2(H_t \circ \psi_1^{-1})|_{(0,0,u,0)} 
= \begin{pmatrix}
a(u) & 0 & 0 & 0 \\
0 & a(u) & 0 & 0 \\
0 & 0 & b(u) & 0 \\
0 & 0 & 0 & c(u)
\end{pmatrix} .
$$
We calculate the type of the singular points by finding the eigenvalues of the matrix
$$ (\om_1^{-1})|_{(0,0,u,0)} d^2(H_t \circ \psi_1^{-1})|_{(0,0,u,0)} = \begin{pmatrix}
0 & a(u) & 0 & 0 \\
-a(u) & 0 & 0 & 0 \\
0 & 0 & 0 & c(u) \\
0 & 0 & -b(u) & 0
\end{pmatrix}. $$
The characteristic polynomial is given by
$$ \chi(\xi) = \bigl(\xi^2 + a^2(u)\bigr)\bigl(b(u)c(u)+\xi^2 \bigr) .$$
Hence the eigenvalues are of the form
$$ \lambda^\pm = \pm i a(u) \quad \mbox{ and } \quad \mu^\pm = \pm i \sqrt{b(u)c(u)}. $$
Now we show that $ \lambda^\pm$ and $\mu^\pm$ are purely imaginary.
Let us first investigate $\lam^\pm$. 
Since $u \in\ ]-\sqrt{2}, \sqrt{2} \ [$ for $t>0$, we have $\mcM(u)>0$ so that $\sqrt{\mcM(u)}$ and $a(u)$ are real numbers.
Since $u=u_\pm(t)$ is also real valued, $\lambda^\pm=\lambda_t^\pm$ is purely imaginary for all $t \in\ ]0, 1]$.
Now we investigate $\mu^\pm$. First, recall the function $\mff$ from \eqref{Rank0eqnLem} and note that $\mff(u^2)=\mcM(u)$ and $2u \mff'(u^2)= \mcM'(u)$. Thus we can rewrite
$$ c(u) = 1 - 2t + \frac{\gamma t u \mcM'(u)}{2u\sqrt{\mcM(u)}} \stackrel{\eqref{Rank0eqnLem}}{=}-\frac{\gamma t \mcM(u)}{u\sqrt{\mcM(u)}} = -\frac{\gamma t \sqrt{\mcM(u)}}{u}. $$
The nominator $\gamma t \sqrt{\mcM(u)} $ is a strict positive number for $t > 0$ and hence, we conclude that $c(u)$ and $u$ have opposite signs. Moreover,
\begin{align*}
b(u) &= 1 - 2t + \frac{2t\gamma u \  g(u)}{\sqrt{\mcM^3(u)}} = 1 - 2t + \frac{\ga t \mcM'(u)}{2 \sqrt{\mcM(u)}} - \frac{\ga t \mcM'(u)}{2 \sqrt{\mcM(u)}} + \frac{2t\gamma u \  g(u)}{\sqrt{\mcM^3(u)}} \\
& \stackrel{\eqref{Rank0eqnLem}}{=} - \frac{\ga t \sqrt{M}}{u} - \frac{\ga t \mcM'(u)}{2 \sqrt{M}} + \frac{ 2 \ga t u \ g(u)}{\sqrt{\mcM^3(u)}} 
= \frac{\ga t}{u \sqrt{\mcM^3(u)}} \left( - \mcM^2(u) -\frac{\mcM(u) \mcM'(u)}{2} + 2 u^2 \ g(u) \right) \\
& =:  \frac{\ga t \ \mcP(u)}{u \sqrt{\mcM^3(u)}}.
\end{align*}
Explicitly, we have
\begin{align*}
 \mcP(u)& = -589824 - 995328 u^4 + 678912 u^6 + 193728 u^8 - 173952 u^{10} -  6624 u^{12} + 14112 u^{14} \\
 &\qquad  - 1044 u^{16} - 240 u^{18} + 24 u^{20}, \\
 \mcP'(u)& = 96 (-2 + u) (-1 + u) u^3 (1 + u) (2 + u) (-8 + u^2) (-2 + u^2) (2 + 
   u^2) (6 + u^2) \\
   & \qquad    (-54 - 10 u^2 + 5 u^4).
\end{align*}
$\mcP$ is strictly negative over $[-\sqrt{2}, \sqrt{2}]$ with maximum $\mcP(0)=\mcP(\pm \sqrt{2}) = -589 824<0$ and minimum $\mcP(\pm 1)=-880236<0$, see Figure \ref{PolynomP}.

\begin{figure}[h]
\centering
\begin{subfigure}{0.33\textwidth}
\centering
\includegraphics[width=1.1\linewidth]{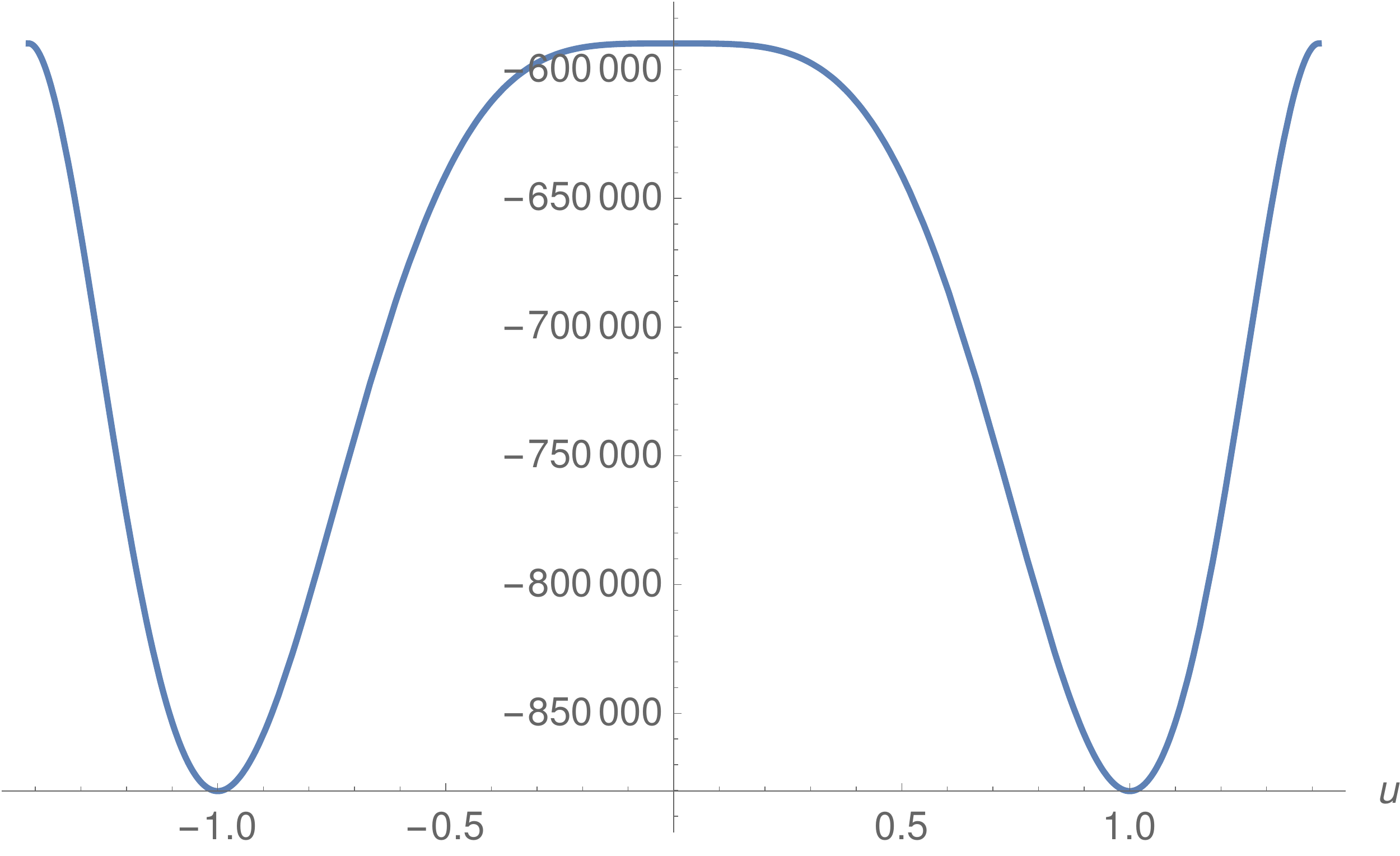}
\end{subfigure}%
\hspace{10mm}
\begin{subfigure}{0.33\textwidth}
\centering
\includegraphics[width=1.1\linewidth]{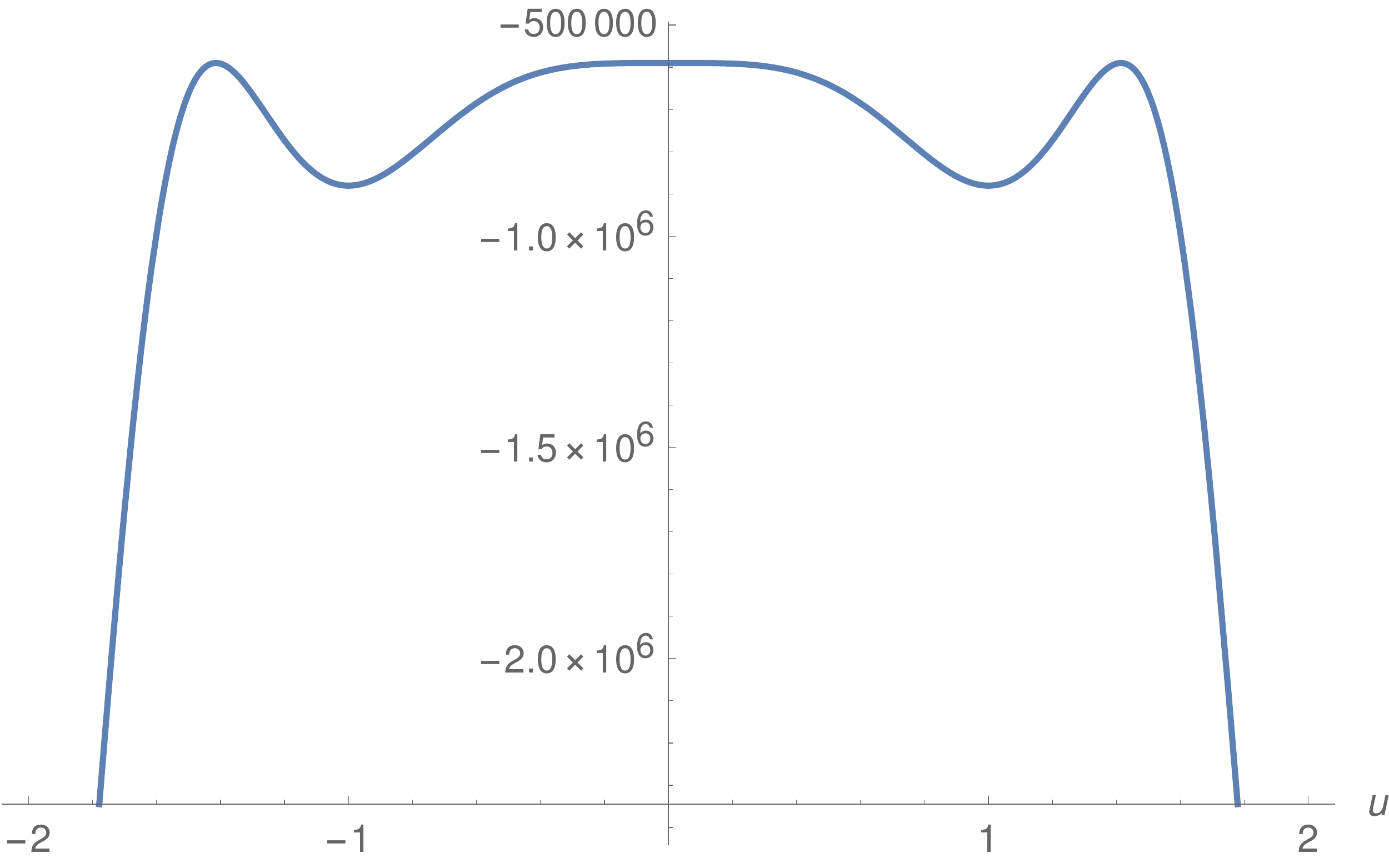}
\end{subfigure}%

\caption{{\em Left:} $\mcP$ plotted from $-\sqrt{2}$ to $\sqrt{2}$, displaying the maxima at $u=0, \pm \sqrt{2}$ and the minima at $u=\pm 1$. {\em Right:} $\mcP$ plotted over an interval strictly larger than $[-\sqrt{2}, \sqrt{2}]$. Already slightly outside $[-\sqrt{2}, \sqrt{2}]$, the function $\mcP$ descends rapidly, becoming `even negativer'. Both plots are done with {\em Mathematica}.}
\label{PolynomP}
\end{figure}

\noindent
Hence $b(u)$ and $u$ have opposite signs.
Altogether we conclude $b(u)c(u) > 0$ and thus $\mu^\pm=\mu_t^\pm$ is purely imaginary for all $t \in [0,1]$.

Finally, a look at the list after \refrankOpoints\ shows that $P_t^{min}$ and $P_t^{max}$ are nondegenerate points of elliptic-elliptic type for all $t \in [0, 1]$.

{\it Second, consider ${Q_t^{min}}$ and ${Q_t^{max}}$:} Here we work in the chart $(U_5, \psi_5)$. Keeping $x_5^2 + y_5^2 = \abs{z_5}^2$ and $x_6^2 + y_6^2 = \abs{z_6}^2$ in mind, we calculate
\begin{align*}
  & (H_t \circ \psi^{-1}_5)(x_5, y_5, x_6, y_6) = \frac{1-2t}{2}(4 + \vert z_5 \vert^2 - \vert z_6 \vert^2) \\
  & \quad + \gamma t x_6 \sqrt{(2 - \vert z_5 \vert^2 + \vert z_6 \vert^2)(6 - 2\vert z_5 \vert^2 + \vert z_6 \vert^2)(8 - \vert z_6 \vert^2)(4 + \vert z_5 \vert^2 - \vert z_6 \vert^2)(2 + 2\vert z_5 \vert^2 - \vert z_6 \vert^2)}
\end{align*}
and the Hessian of $(H_t \circ \psi^{-1}_5)$ in the point $(0,0,u,0) \in \psi_5(U_5)$ with $u:=u_\pm(t)$ is given by
$$ d^2( H_t \circ \psi^{-1}_5)|_{(0,0,u,0)}
= 
\begin{pmatrix}
\ati(u) & 0 & 0 & 0 \\
0 & \ati(u) & 0 & 0 \\
0 & 0 & \bti(u) & 0 \\
0 & 0 & 0 & \cti(u)
\end{pmatrix} $$
where
\begin{eqnarray*}
\ati(u) &:=& 1 - 2t + \frac{2 \gamma t u (160 + 300 u^2 - 64 u^4 - 21 u^6 + 3 u^8)}{\sqrt{768 - 160 u^2 - 216 u^4 + 44 u^6 + 6 u^8 - u^{10}}} \\
\bti(u) &:=& -1 + 2t + \frac{2 \gamma t u \ g(u)}{(768 - 160 u^2 - 216 u^4 + 44 u^6 + 6 u^8 - u^{10})^{3/2}}
\end{eqnarray*}
with $g(u)$ the same function as in \eqref{gFunction} and 
\begin{eqnarray*}
\cti(u) &:=& -1 + 2t + \frac{\gamma t u (-160 - 432 u^2 + 132 u^4 + 24 u^6 - 5 u^8)}{\sqrt{768 - 160 u^2 - 216 u^4 + 44 u^6 + 6 u^8 - u^{10}}}.
\end{eqnarray*}
Note that $\ati(u)$ coincides with $a(u)$ except for the minus sign at the term $(1-2t)$ right at the beginning. The same holds true for $\bti(u)$ and $b(u)$ and $\cti(u)$ and $c(u)$. Thus we obtain almost the same matrix up to these terms. Now use \eqref{Rank0eqn2} instead of \eqref{Rank0eqnLem}, where this opposite sign also appears. Similar calculations as above show that $Q_t^{min}$ and $Q_t^{max}$ are also nondegenerate singular points of elliptic-elliptic type.
\end{proof}


\section{{\bf The singular points of rank one and their types}}

\label{section rankOneType}

For $t=0$, the system $F_0=(J, H_0)=(J, H)$ is toric (see \refthOctagon) so that its rank one points are elliptic-regular and mapped to the edges of the momentum polytope $F_0(M)=\De$. In what follows, we will study the case $t > 0$: we will first investigate the case $dJ(p) \neq 0$ and then the case $dJ(p) = 0$.


\subsection{Case $\mathbf{dJ(p) \neq 0}$}

Geometrically this means that, at such a rank one point $p$, either $dH_t(p)=0$ or $dH_t(p)$ and $dJ(p)$ are linearly dependent.

\begin{lemma} 
\label{entriesRkOne}
Let $t \in\ ]0,1]$ and let $p = [p_1, p_2, \ldots, p_8]$ be a rank one point of $F_t = (J, H_t)$ with $dJ(p) \neq 0$.
Then $p_2, p_3, p_4, p_6, p_7, p_8$ are all non-zero.
\end{lemma}

\begin{proof}
Let $p$ be a rank one point of $F_t=(J, H_t)$. Then there exists $(\nu, \mu) \in \mathbb{R}^2 \setminus\{(0,0)\}$ with $\nu\ dH_t(p) + \mu\ dJ(p) = 0$. We conclude that $\nu \neq 0$ since $\nu=0$ would force also $\mu=0$ because of $dJ(p) \neq 0$, but $(\mu, \nu)=(0,0)$ is excluded. Thus we may divide by $\nu$ and rewrite the condition to be a rank one point as $ dH_t(p) - \lambda dJ(p) = 0 $ for some $\lambda \in \mathbb{R}$, i.e., we look for the extrema of the function $H_t - \lambda J: M \to \R$. Now summarise the six manifold equations of $M$ in \eqref{manifold eqn} in a function $K = (K_1, \ldots, K_6): \mathbb{C}^8 \to \R^6$.
Thus we may see $p$ as critical point of $H_t - \lambda J$ in $\C^8$ having to satisfy $K(p)=0$.

This is an optimisation problem with constraints which can be solved by using Lagrange multipliers: Thus a rank one singular point is a point $p = [p_1, \ldots, p_8] \in \mathbb{C}^8$ for which there are $\lambda, \si_1, \ldots, \si_6 \in \mathbb{R}$ such that $p$ satisfies
\begin{equation}
\label{LagMultCond}
 \left\{ 
 \begin{aligned}
 & \ \nabla H_t(p) = \lambda \nabla J(p) + \si_1\nabla K_1(p) + \ldots + \si_6 \nabla K_6(p) \\
& \ \vert p_1 \vert^2 + \vert p_5 \vert^2 = 6 \\
& \ \vert p_2 \vert^2 + \vert p_5 \vert^2 + \vert p_7 \vert^2 = 10 \\
& \ \vert p_3 \vert^2 + \vert p_7 \vert^2 = 6 \\
& \ \vert p_4 \vert^2 - \vert p_5 \vert^2 + \vert p_7 \vert^2 = 4 \\
& \ \vert p_5 \vert^2 - \vert p_6 \vert^2 + \vert p_7 \vert^2 = 2 \\
& \ \vert p_5 \vert^2 - \vert p_7 \vert^2 + \vert p_8 \vert^2 = 4  
 \end{aligned}
 \right.
\end{equation}
Write the gradient in $\mathbb{C}^8$ as $(\partial_{z_1}, \partial_{\overline{z_1}}, \ldots, \partial_{ z_8}, \partial_{ \overline{z_8}})$ and rewrite $H_t$ and $J$ as
$$H_t(z_1, \dots, z_8) = \frac{1-2t}{2} z_3 \: \overline{z_3} + t\ \gamma\ \mfR(\overline{z_2} \: \overline{z_3} \: \overline{z_4} \: z_6 \: z_7 \: z_8) \quad \mbox{and} \quad J(z_1, \dots, z_8) = \frac{1}{2} z_1 \: \overline{z_1}$$ 
Then the first equation of \eqref{LagMultCond} yields w.r.t.\ derivatives $\partial_{ z_1}, \dots, \partial_{z_8}$
\begin{eqnarray*}
0 &=& \lambda \overline{p_1} + \si_1 \overline{p_1} \\
t \gamma p_3 \: p_4 \: \overline{p_6} \: \overline{p_7} \: \overline{p_8} &=& \si_2 \overline{p_2} \\
(1-2t)\overline{p_3} + t \gamma p_2 \: p_4 \: \overline{p_6} \: \overline{p_7} \: \overline{p_8} &=& \si_3 \overline{p_3} \\
t \gamma p_2 \: p_3 \: \overline{p_6} \: \overline{p_7} \: \overline{p_8} &=& \si_4 \overline{p_4} \\
0 &=& (\si_1 + \si_2 - \si_4 - \si_5 + \si_6)\overline{p_5} \\
t \gamma \overline{p_2} \: \overline{p_3} \: \overline{p_4} \: p_7 \: p_8 &=& \si_5 \overline{p_6} \\
t \gamma \overline{p_2} \: \overline{p_3} \: \overline{p_4} \: p_6 \: p_8 &=& (\si_2 + \si_3 + \si_4 - \si_5 - \si_6) \overline{p_7} \\
t \gamma \overline{p_2} \: \overline{p_3} \: \overline{p_4} \: p_6 \: p_7 &=& \si_6 \overline{p_8}
\end{eqnarray*}
and we get the same equations with all $p_i$ conjugate for the derivatives $\partial_{ \overline{z_1}}, \dots, \partial_{\overline{z_8}}$.

Suppose now that $p_2 = 0$. Since $t \neq 0$ and $\gamma > 0$, it follows from the second equation that at least one of $p_3, p_4, p_6$, $p_7$, $p_8$ must be zero. A similar conclusion is true if we assume one of $p_3$, $p_4$, $p_6$, $p_7$, $p_8$ to vanish. \refparamU\ implies that never more than two subsequent coordinates $p_k$ for $k \in \{2,3,4,6,7,8\}$ are zero. But in this case, we get one of the fixed points $A, B, C$, $D$ which are rank zero points instead of rank one points with $dJ(p) = 0$.
Thus a singular point $p$ of rank one with $dJ(p) \neq 0$ has entries $p_2, p_3, p_4, p_6, p_7, p_8 \neq 0$.
\end{proof}

Recall from the proof of \refcoordEEPoints\ that the set of points with $dJ=0$ consists of the fixed points $A,B,C,D$ together with all points $z$ satisfying $z_1 = 0$ or $z_5 = 0$. Thus, if $dJ(z) \neq 0$ for a point $z \in M$, we have $z_1, z_5 \neq 0$. Together with \refentriesRkOne, this means that the coordinate entries of a singular point of rank one never vanish.

W.l.o.g.\ let us work in the chart $(U_1, \psi_1)$. Given $z \in U_1$ with $dJ(z) \neq 0$ (and thus no entries equal to zero), write it as $z= \psi^{-1}_1(x_1, y_1, x_2, y_2)=[x_1, y_1, x_2, y_2, x_3, 0, \dots, x_8, 0]$ and rotate $z_1=x_1+i y_1$ with the $\mbS^1$-action of $J$ until $y_1=0$. Abbreviate $J(z)=:j$ and note that we have in this situation $x_1=\sqrt{2j}$. Now consider the image of $z$ under the quotient map to the reduced space $M^{red, j}$. Note that the possible choices of $z_2=x_2 + i y_2$ in the chart $z=\psi^{-1}_1(x_1, y_1, x_2, y_2)$ now describe $M^{red, j}$ completely. Since $z_2 \neq 0$, we may write $z_2=  \rho e^{i\theta}$ in polar coordinates. We have $\theta \in [0, 2\pi[$, but the actual values of $\rho>0$ are implicitly determined by the manifold equations \eqref{manifold eqn} which read in the new coordinates
\begin{equation}
\label{rhojcoord}
\left\{ \qquad 
\begin{aligned}
x_1 &= \sqrt{2j} , &&& x_5 &= \sqrt{6 - 2j}, \\
x_2 &= \rho \cos (\theta) , &&& x_6 &= \sqrt{8 - \rho^2}, \\
y_2 &= \rho \sin (\theta) , &&& x_7 &= \sqrt{4 + 2j - \rho^2}, \\
x_3 &= \sqrt{2 - 2j + \rho^2}, &&& x_8 &= \sqrt{2 + 4j - \rho^2}, \\
x_4 &= \sqrt{6 - 4j + \rho^2}, &&& y_1  & =  y_3 = y_4 = y_5 = y_6 = y_7 = y_8 =0.
\end{aligned}
\right. 
\end{equation}
\begin{figure}[h]
\centering
\includegraphics[scale=.2]{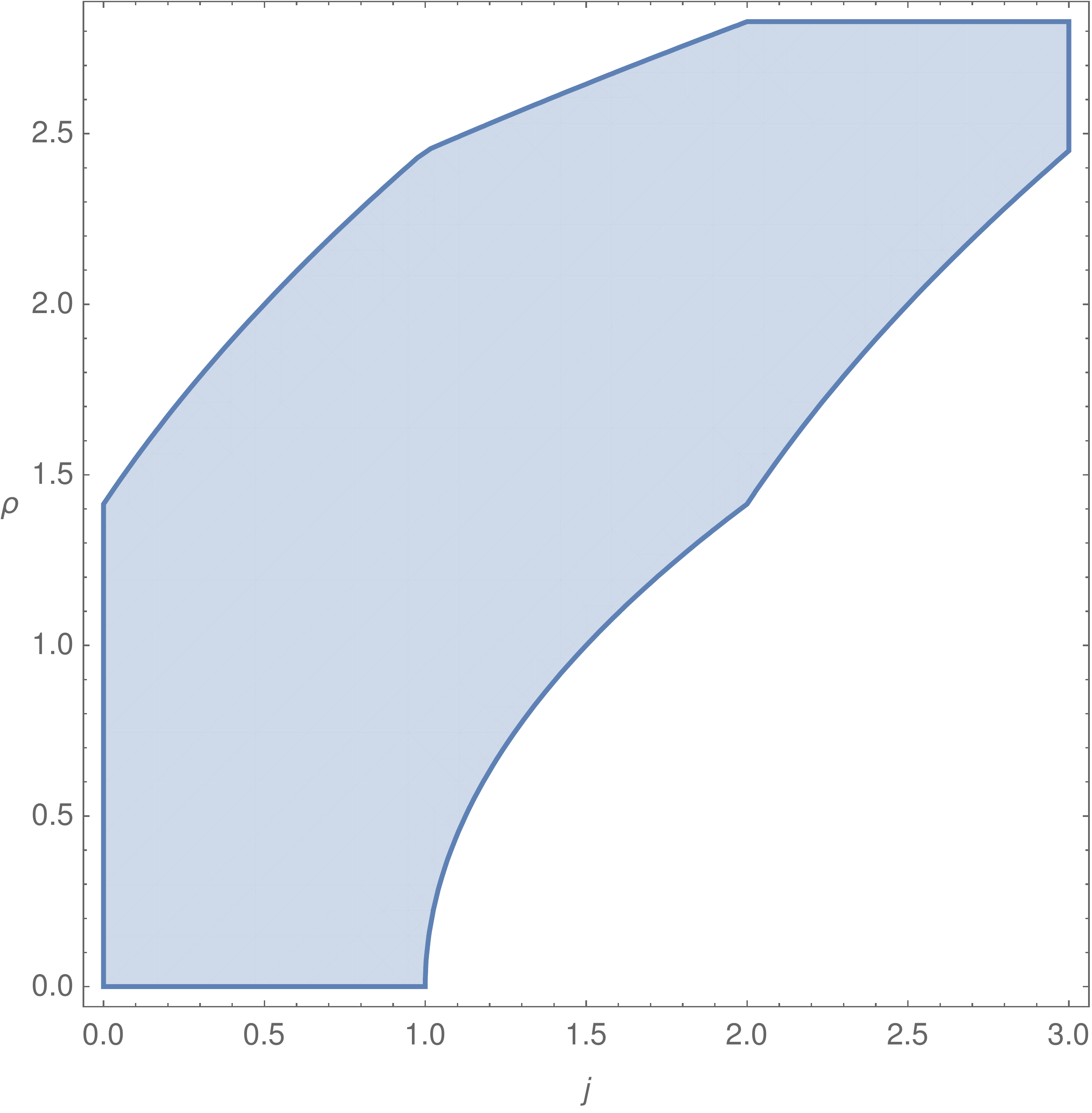}
\caption{The domain of definition of $\rho$ depending on the value of $j$ plotted with {\em Mathematica} between the maximally admissible bounds $0 \leq j \leq 3$ and $0\leq \rho \leq \sqrt{8}$ for the constraints given in \eqref{rhojcoord}.}
\label{DomainRhoJ}
\end{figure}
Thus we have always in particular $0 < \rho< \sqrt{8}$ and, by definition of $J$, also $0 \leq j \leq 3$. The possible values for $\rho$ depend on the value of $j$ and are plotted in Figure \ref{DomainRhoJ}. The function $H_t^{red, j} $ on $M^{red, j} $ becomes in these coordinates  
\begin{eqnarray*}
H_t^{j, red} ( \rho, \theta) = \frac{1-2t}{2}(2 - 2j + \rho^2) + \gamma t \rho \cos (\theta) \sqrt{g(\rho, j)}
\end{eqnarray*}
where
$$ 
g(\rho,j) := (2 - 2j + \rho^2)(6 - 4j + \rho^2)(8 - \rho^2)(4 + 2j - \rho^2)(2 + 4j - \rho^2) = (x_3 x_4 x_6 x_7 x_8)^2. 
$$
This change of coordinates is of use for 

\begin{proposition}
\label{ellRegVertical}
For $t \in\ ]0,1]$, the rank one points of $F_t = (J, H_t)$ on $M \setminus \{ J^{-1}(0) \cup\ J^{-1}(3) \}$ are nondegenerate of elliptic-regular type.
\end{proposition}

\begin{proof}
By assumption, $dJ$ does not vanish in the rank one points under consideration. According to \refredRankOne, the point $(\rho,\theta)$ is a rank one singular point of $F_t$ if $dH_t^{red,j}(\rho, \theta) = 0$, i.e., the following equations hold true:
\begin{eqnarray}
\partial_\theta H_t^{red,j} &=& -\gamma t \rho \sin (\theta) \sqrt{g(\rho, j)} = 0 \label{Rank1eqn1}, \\
\partial_\rho H_t^{red,j} &=& (1-2t)\rho + \gamma t \cos (\theta) \left( \sqrt{g(\rho,j)} + \frac{\rho\ \del_\rho g(\rho,j)}{2 \sqrt{g(\rho, j)}} \right) = 0 \label{Rank1eqn2}.
\end{eqnarray}
First consider \eqref{Rank1eqn1}. We have $\gamma, t, \rho > 0$ and, furthermore, by \refentriesRkOne, $\sqrt{g(\rho, j)} = x_3 x_4 x_6 x_7 x_8 \neq 0$. Thus the first equation holds true for $ \sin \theta = 0$, i.e., $\theta \in \{ 0,  \pi\} \subset [0, 2\pi[$.
Now abbreviate
$$ h(\rho, j) : = \sqrt{g(\rho, j)} + \frac{\rho\ \del_\rho g(\rho, j)}{2 \sqrt{g(\rho, j)}} $$
and calculate the entries of the Hessian $d^2 H_t^{j, red}$ of $H_t^{j, red}$ as
\begin{eqnarray*}
\partial^2_{\rho \theta} H_t^{red,j} &=& -\gamma t \sin (\theta) h(\rho, j), \\
\partial^2_{\theta \theta} H_t^{red,j} &=& -\gamma t \cos (\theta) \sqrt{g(\rho, j)}, \\
\partial^2_{\rho \rho} H_t^{red,j} &=& 1 - 2t + \gamma t \cos (\theta) \; \del_\rho h(\rho, j).
\end{eqnarray*}
For $\theta \in \{0, \pi\}$, we find $\partial^2_{\rho \theta} H_t^{red,j}=0$ and $\partial^2_{\theta \theta} H_t^{red,j} < 0 $ for $ \theta = 0$ and $\partial^2_{\theta \theta} H_t^{red,j} > 0 $ for $ \theta = \pi$.
To determine the sign of $\partial^2_{\rho \rho} H_t^{red,j}$, we rewrite equation \eqref{Rank1eqn2} as
$$ 1 - 2t = \frac{- \gamma t \cos (\theta) \bigl(2g(\rho, j) + \rho\ \del_\rho g(\rho, j)\bigr)}{2 \rho \sqrt{g(\rho, j)}} $$
and calculate
$$ \del_\rho h(\rho, j) = \frac{2g(\rho, j) \ \bigl(\rho\ \del^2_{\rho \rho} g(\rho, j) + 2 \del_\rho g(\rho, j) \bigr) - \rho \ \del_\rho g(\rho, j)^2}{4 g(\rho, j)^{\frac{3}{2}}}. $$
Combining those results yields
$$ \partial^2_{\rho \rho} H_t^{red,j} = \frac{\gamma t \cos(\theta) f(\rho, j)}{4 \rho g(\rho, j)^{\frac{3}{2}}} $$
with
$$ f(\rho, j) := 2\rho^2 g(\rho, j) \ \del^2_{\rho \rho} g(\rho, j) + 2\rho\ g(\rho, j) \ \del_\rho g(\rho, j) - \rho^2 \ \bigl(\del_\rho g(\rho, j)\bigr)^2  - 4g(\rho, j)^2. $$
As Figure \ref{FctDomainRhoJ} shows, we have $f(\rho,j) < 0$ for all admissible $(j, \rho)$ in the interior of the domain plotted in Figure \ref{DomainRhoJ}.
\begin{figure}[h]
\centering
\begin{subfigure}{0.33\textwidth}
\centering
\includegraphics[width=.9\linewidth]{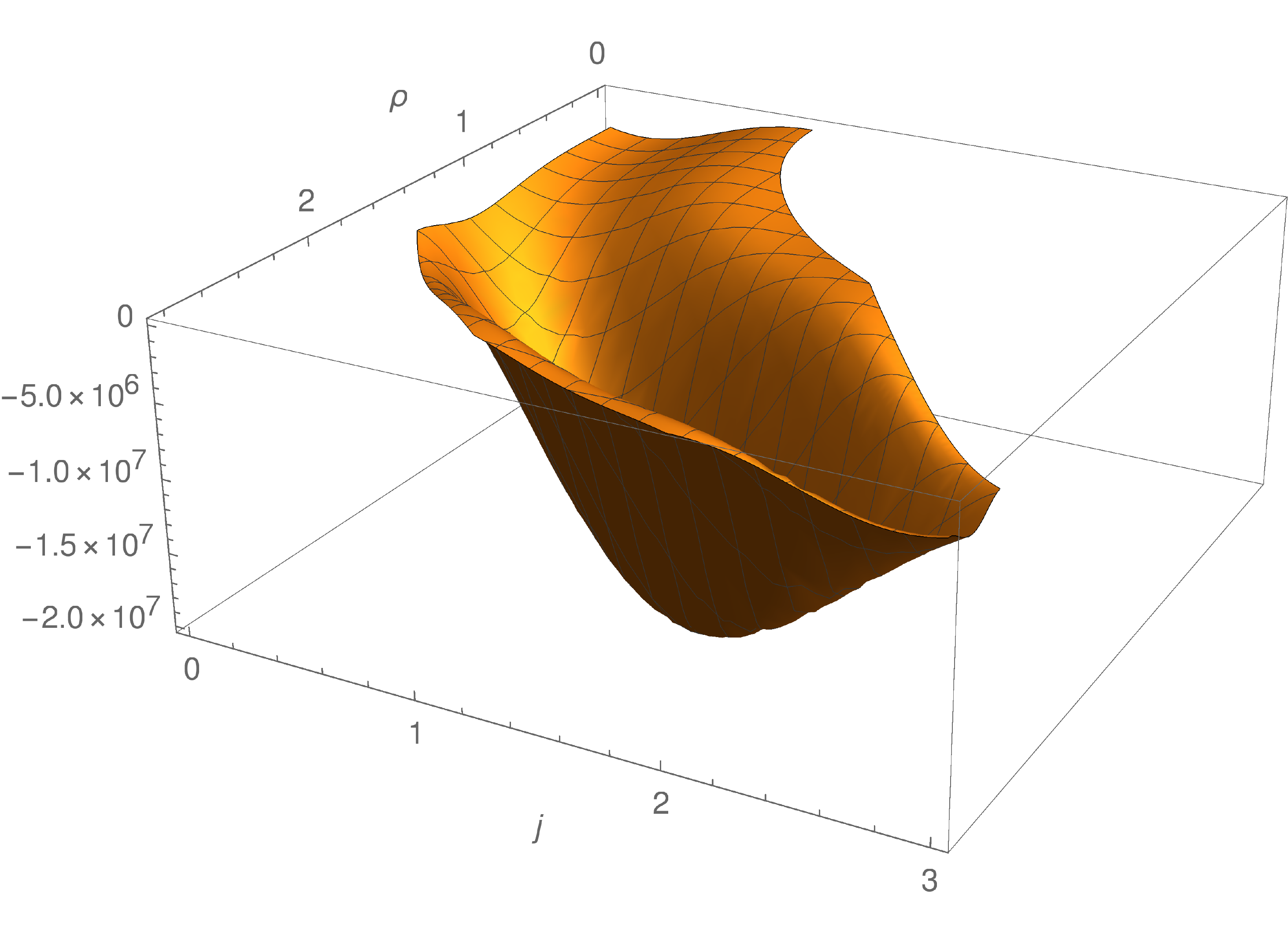}
\end{subfigure}%
\hspace{-12mm}
\begin{subfigure}{0.33\textwidth}
\centering
\includegraphics[width=1.3
\linewidth]{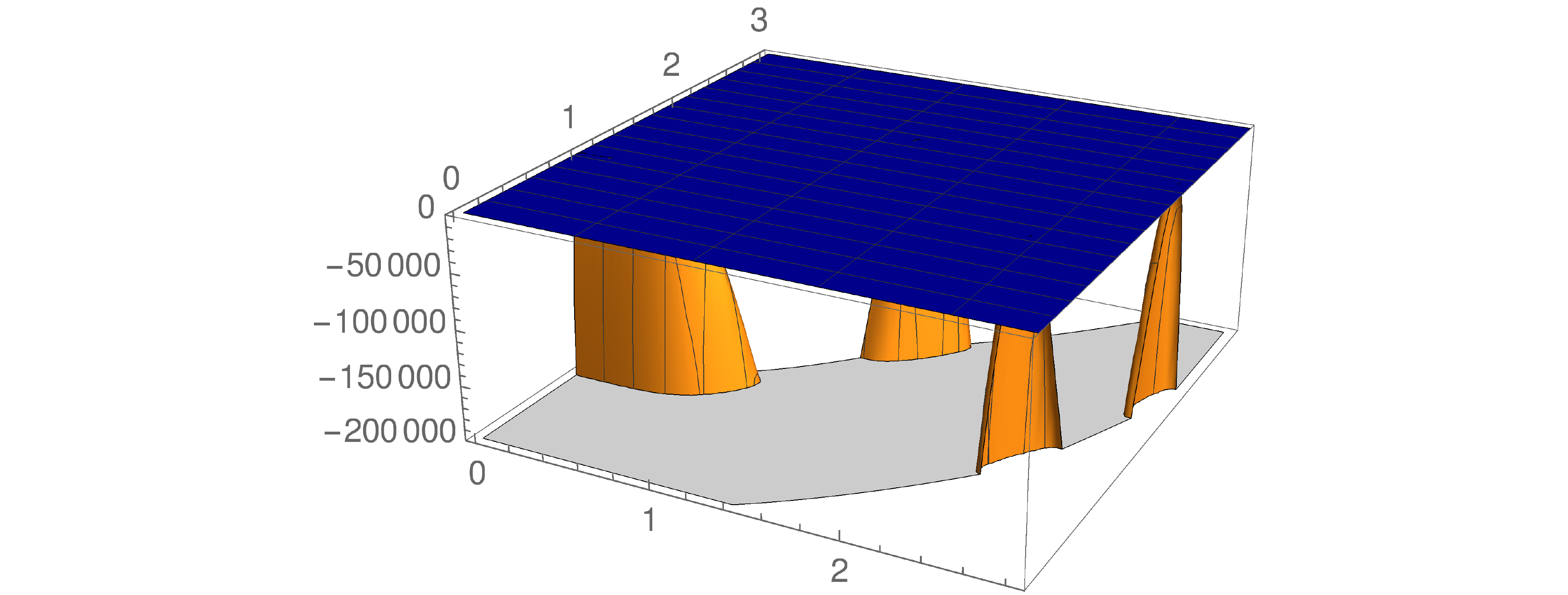}
\end{subfigure}%
\begin{subfigure}{0.33\textwidth}
\centering
\includegraphics[width=1.3
\linewidth]{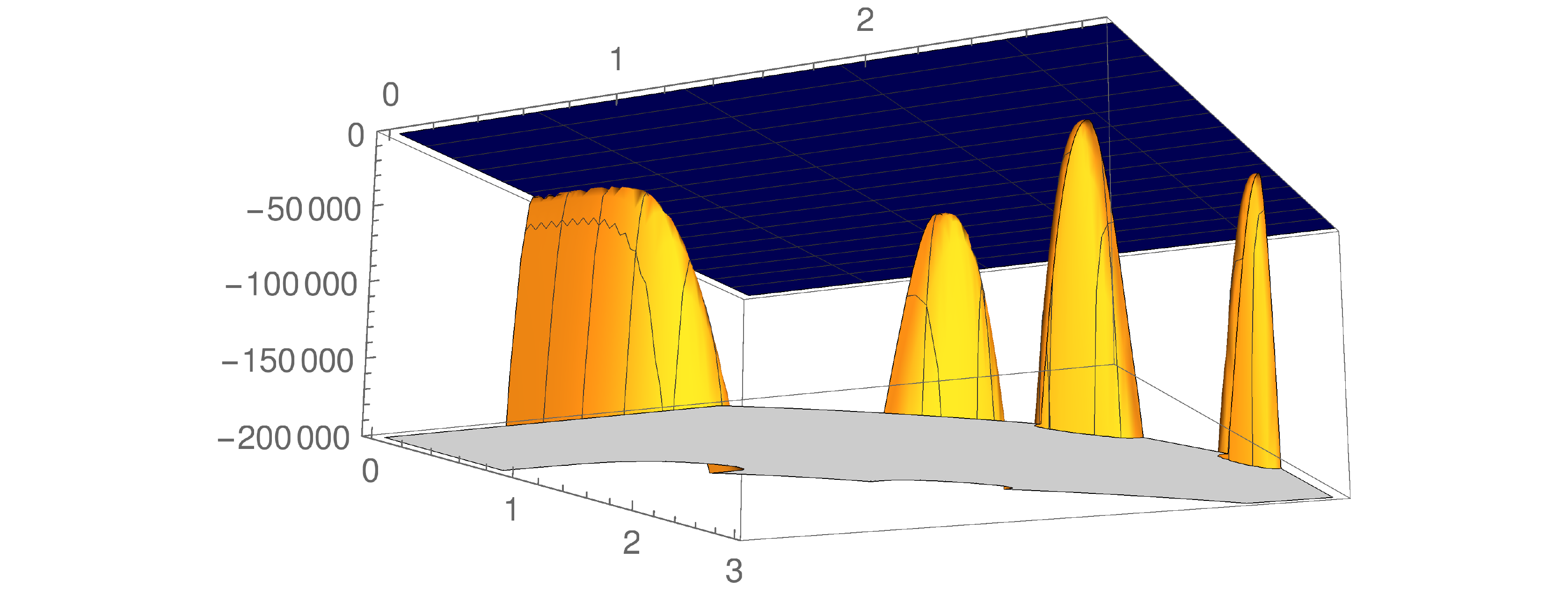}
\end{subfigure}

\caption{{\em Left:} Plot of $f(\rho, j)$ over the region of admissible $(\rho, j)$ from Figure \ref{DomainRhoJ}. The graph is approaching the horizontal plane of height zero from below. {\em Center and Right:} The horizontal plane of height zero plotted in {\em blue} and $f(\rho, j)$ plotted in {\em orange} for the range between $-200000$ and $0$. The `orange peaks' rise at the boundary of the admissible region. On the boundary itself, the function may vanish (e.g.\ $f(0,1)=0$). All three plots are realised with {\em Mathematica}.
}
\label{FctDomainRhoJ}
\end{figure}
Hence we conclude $\partial^2_{\rho \rho} H_t^{red,j} <0 $ for $\theta = 0$ and $\partial^2_{\rho \rho} H_t^{red,j} >0 $ for $\theta = \pi$.

Putting all results together, we see that the Hessian $d^2H_t^{j, red}(\rho, \theta)$ has strictly positive determinant in both cases. Therefore by \refredRankOne, the point $(\rho, \theta)$ is a nondegenerate rank one point of $F_t$ of elliptic-regular type.
\end{proof}


\subsection{Case $\mathbf{dJ(p) = 0}$}

Geometrically this means that, at such a rank one point $p$, we have $dH_t(p)\neq 0$, i.e., we are dealing with points that are fixed under the flow of $J$, but are non-fixed by the flow of $H_t$. Recall from the proof of \refcoordEEPoints\ that the set of points with $dJ=0$ consists of the fixed points $A,B,C,D$ together with all points $z$ satisfying $z_1 = 0$ or $z_5 = 0$.

\begin{proposition}
\label{ellRegHorizontal}
For $t \in \ ]0,1]$, the rank one points of $F_t = (J, H_t)$ in $J^{-1}(0) \cup J^{-1}(3)$ are nondegenerate of elliptic-regular type .
\end{proposition}

\begin{proof}
{\it Consider first $J^{-1}(0)$:} Here a rank one point $p$ satisfies $0=p_1$ and hence $p \in U_1$ or $p \in U_8$. Let us first assume $p \in U_1$ and write $p = \psi_1^{-1}(0,0,x_2, y_2)$. According to the discussion around \refnondegRankOne, we need to consider the Hessian of $J$ to determine nondegeneracy and type of rank one points. In the chart $(U_1, \psi_1)$ evaluated in $(0,0,x_2, y_2)$, we have
$$ d^2(J \circ \psi_1^{-1}){(0, 0, x_2, y_2)}= \begin{pmatrix}
1 & 0 & 0 & 0 \\
0 & 1 & 0 & 0 \\
0 & 0 & 0 & 0 \\
0 & 0 & 0 & 0
\end{pmatrix} 
\quad \mbox{ and } \quad 
(\om_1^{-1})_{(0, 0, x_2, y_2)} = 
\begin{pmatrix}
0 & 1 & 0 & 0 \\
-1 & 0 & 0 & 0 \\
0 & 0 & 0 & 1 \\
0 & 0 & -1 & 0
\end{pmatrix} $$
and thus
$$  (\om_1^{-1})_{(0, 0, x_2, y_2)}d^2(J \circ \psi_1^{-1}){(0, 0, x_2, y_2)} = \begin{pmatrix}
0 & 1 & 0 & 0 \\
-1 & 0 & 0 & 0 \\
0 & 0 & 0 & 0 \\
0 & 0 & 0 & 0
\end{pmatrix}. $$
Now we determine the space $L_p \subseteq T_p M$ given by the tangent line of the orbit through $p$. Since $dJ(p) = 0$, we have $ L_p = \mbox{Span}\left\lbrace \mathcal{X}^{H_t}(p) \right\rbrace$. 
Similar to in the proof of \refHtFt\ we obtain $ d(H_t \circ \psi_1^{-1})(0,0,x_2, y_2) = (0, 0, f_1, f_2) $ for some functions $f_1$ and $f_2$ depending on $x_2, y_2$. We calculate the Hamiltonian vector field
\begin{align*}
\left(\mathcal{X}^{H_t \circ \psi_1^{-1}}\right)^T = -d (H_t \circ \psi_1^{-1}) \om_1^{-1} = (0,0,f_2, -f_1).
\end{align*}
Therefore, we get
$$ L_p
= \mbox{Span}\left\lbrace \; \begin{pmatrix}
0 \\ 0 \\ f_2 \\ -f_1
\end{pmatrix} \; \right\rbrace 
\quad \mbox{and} \quad 
L_p^{\bot} = \mbox{Span}\left\lbrace \; \begin{pmatrix}
1 \\ 0 \\ 0 \\ 0
\end{pmatrix} \; , \; \begin{pmatrix}
0 \\ 1 \\ 0 \\ 0
\end{pmatrix} \; , \; \begin{pmatrix}
0 \\ 0 \\ f_2 \\ -f_1
\end{pmatrix} \; \right\rbrace 
$$
so that
$$
L^{\bot}_p/L_p \simeq \mbox{Span}\left\lbrace \; \begin{pmatrix}
1 \\ 0 \\ 0 \\ 0
\end{pmatrix} \;, \; \begin{pmatrix}
0 \\ 1 \\ 0 \\ 0
\end{pmatrix} \; \right\rbrace. $$
Therefore $\om_1^{-1}d^2(J \circ \psi_1^{-1})$ descends to $L_p^{\bot}/L_p$ as 
$\left( \begin{smallmatrix}
0 & 1 \\ -1 & 0
\end{smallmatrix} \right) $
which has purely imaginary eigenvalues $\pm i$. Thus, by \refnondegRankOne\ and the discussion afterward, the rank one points are nondegenerate of elliptic-regular type.

The only rank one points in $J^{-1}(0)$ that lie in $U_8$ but not in $U_1$ are points $z$ with $z_1=0=z_8$. When we insert this into the manifold equations \eqref{manifold eqn}, we get immediately $\abs{z_1}^2 =0= \abs{z_8}^2$ and then find successively $\abs{z_5}^2=6$, $\abs{z_7}^2=2$, $\abs{z_3}^2=4$, $\abs{z_4}^2 = 4$, $\abs{z_6}^2 = 6$ and eventually $\abs{z_2}^2 = 2$. This determines a unique point $p:=[ 0, \ \sqrt{2},\ 2,\ \sqrt{2},\ \sqrt{6}, \ \sqrt{6},\ \sqrt{2},\ 0] \in U_8$. Similar calculations as above now in the chart $(U_8, \psi_8)$ for $p=\psi_8^{-1}(0,0,0,0)$ lead to
$$ 
(\om_8^{-1})_{(0,0,0,0)} d^2(J\circ \psi_8^{-1})(0,0,0,0) = 
\begin{pmatrix}
0 & 0 & 0 & 0 \\
0 & 0 & 0 & 0 \\
0 & 0 & 0 & 1 \\
0 & 0 & -1 & 0
\end{pmatrix}
$$
and $\mathcal{X}^{H_t \circ \psi_8^{-1}}(0,0,0,0) = (f_2, -f_1, 0, 0)^T $ for some functions $f_1$ and $f_2$. Then $(\om_8^{-1})_{(0,0,0,0)} d^2(J \circ \psi_8^{-1})(0,0,0,0)$ descends to the associated quotient $L_p^\perp \slash L_p$ as 
$
\left(\begin{smallmatrix}
0 & 1 \\ -1 & 0
\end{smallmatrix} \right)
$. 
Thus $p$ also is of elliptic-regular type.

{\it Now consider $J^{-1}(3)$:} Here a rank one point $p$ satisfies $0= p_5$. We work in the chart $(U_5, \psi_5)$ where $p = \psi_5^{-1}(0,0,x_6, y_6)$.
Using the relation $\vert z_1 \vert^2 + \vert z_5 \vert^2 = 6$ we get
$$ 
(J  \circ \psi_5^{-1})(x_5, y_5, x_6, y_6)= \frac{6 - x_5^2 - y_5^2}{2} 
\quad \mbox{and} \quad 
d^2(J \circ\psi_5^{-1})|_{(0,0,x_6, y_6)} = \begin{pmatrix}
-1 & 0 & 0 & 0 \\
0 & -1 & 0 & 0 \\
0 & 0 & 0 & 0 \\
0 & 0 & 0 & 0
\end{pmatrix}
$$
and hence
$$
(\om_5^{-1})_{(0,0,x_6, y_6)}d^2(J \circ\psi_5^{-1}){(0,0,x_6, y_6)} = \begin{pmatrix}
0 & -1 & 0 & 0 \\
1 & 0 & 0 & 0 \\
0 & 0 & 0 & 0 \\
0 & 0 & 0 & 0
\end{pmatrix} $$
Analogously to above, we find $d(H_t \circ \psi_5^{-1}) = (0,0,f_1,f_2)$ where $f_1$, $f_2$ are some suitable functions. Moreover, $\om_5^{-1|_{(0,0,x_6, y_6)}}d^2(J \circ\psi_5^{-1})|_{(0,0,x_6, y_6)} $ descends to the associated $L_p^\perp /L_p$ as 
$ \left(\begin{smallmatrix}
0 & -1 \\ 1 & 0
\end{smallmatrix} \right) $.
This matrix has eigenvalues $\pm i$, so the rank one points in $J^{-1}(3)$ in $U_5$ are also nondegenerate and of elliptic-regular type.

The chart $(U_5, \psi_5)$ contains all possible rank one points in $J^{-1}(3)$ except those with $0=p_4$ for which we have to use the chart $(U_4, \psi_4)$. Analogous calculations as above show that there is only one such point and that it is of elliptic-regular type.
\end{proof}


\section{{\bf The proofs of \refmainTheorem\ and \refdoublePinchParam}} 


\label{section pinched}

\begin{proof}[Proof of \refmainTheorem]
In \refthOctagon, we showed that $(M, \om)$ is a $4$-dimensional compact symplectic manifold and that $F=(J, H)$ is, up to equivariant symplectomorphism, the toric system satisfying $F(M)=\De$ where $\De $ is the octagon from Figure \ref{fig_octagon}. In \refHtFt, we extended $(M, \om, F=(J, H))$ to a family $(M, \om, F_t=(J, H_t))$ of integrable systems with $F_0=F$. Since $J$ does not depend on $t$, its induced $\mbS^1$-action is unchanged under variation of $t$. Moreover, since $M$ is compact, $J:M \to \R$ is proper.

It remains to show that, apart from a finite number of transition times, the system $F_t=(J, H_t)$ is semitoric, i.e., singularities are nondegenerate and have no hyperbolic components and that, at the transition times, certain singular points change from elliptic-elliptic to focus-focus or vice versa.

In Section \ref{section posRankZero}, we deduced the eight fixed points of $F_t=(J, H_t)$ for $t \in\ ]0, 1]$ denoted by $A$, $B$, $C$, $D$, $P^{min}_t$, $P^{max}_t$, $Q^{min}_t$, $Q^{max}_t$ and their coordinates, hereby proving \refrankOpoints. In \refAeeffee, we showed that there are two transition times $0< t^- < \frac{1}{2} <t^+ <1$ where $A$, $B$, $C$, $D$ switch from being nondegenerate and elliptic-elliptic via a degeneracy at $t^-$ to being nondegenerate and focus-focus and then via a degeneracy at $t^+$ to being nondegenerate and elliptic-elliptic. In \refeeFixedPoints, we showed that $P^{min}_t$, $P^{max}_t$, $Q^{min}_t$, $Q^{max}_t$ are nondegenerate and elliptic-elliptic for all $t \in\ ]0,1]$.

\refellRegVertical\ proves that the rank one points of $F_t=(J, H_t)$ on $M \setminus (J^{-1}(0) \cup J^{-1}(3))$ are nondegenerate and of elliptic-regular type for all $t \in\ ]0,1]$. \refellRegHorizontal\ shows that the rank one points in $J^{-1}(0)$ and $J^{-1}(3)$ are nondegenerate and of elliptic-regular type for all $t \in\ ]0,1]$. This finishes the proof of \refmainTheorem.
\end{proof}

It remains to prove \refdoublePinchParam, i.e., that, at $t= \frac{1}{2}$, the focus-focus points $A$ and $B$ lie both in $F_{\frac{1}{2}}^{-1}(1,0)$ and that the focus-focus points $C$ and $D$ lie both in $F_{\frac{1}{2}}^{-1}(2,0)$ and that this gives the fibres the form of a double pinched torus.

\begin{proof}[Proof of \refdoublePinchParam]
Zung \cite{zung1, zung2} showed that a focus-focus fibre containing exactly $n$ focus-focus points consists of a chain of $n$ spheres where each of the spheres intersects transversally two others and that the intersection points are given by the $n$ focus-focus points. Such a fibre is said to have the form of an $n$-pinched torus. In the special case $n = 1$, the fibre is a sphere with one point of self-intersection and is said to be a single pinched torus.

The exact coordinates of the fixed points $A$, $B$, $C$, $D$ are calculated in \refrankOpoints\ and \refAeeffee\ proves that they are of focus-focus type for $t^- < t <t^+$ where $0< t^- < \frac{1}{2} < t^+ < 1$. We find $F_{\frac{1}{2}} = (J, \frac{\gamma}{2}X)$ and calculate $F_{\frac{1}{2}}(A)= (1, 0)=F_{\frac{1}{2}}(B) $ and $F_{\frac{1}{2}}(C)= (2, 0)=F_{\frac{1}{2}}(D) $. Thus, according to Zung \cite{zung1, zung2}, the fibres over $(1,0)$ and $(2, 0)$ have the form of double pinched tori.

We now parametrise the fibre $ F_{\frac{1}{2}}^{-1}\left(1,0\right)$ as follows.
Given any $z = [z_1, \dots, z_8] \in F_{\frac{1}{2}}^{-1}\left(1,0\right) \subset M$, we have $J(z) = 1$ and thus $\abs{z_1}^2 = 2$. Moreover, we find $0=X(z) =  \mfR(\overline{z_2} \: \overline{z_3} \: \overline{z_4} \: z_6 \: z_7 \: z_8) = 0$. Since the last two coordinate entries of $A$ are zero, $A$ lies in the chart $(U_7, \psi_7)$. Thus let us determine the points of $ F_{\frac{1}{2}}^{-1}\left(1,0\right)$ that lie in $U_7$. 
Recall 
$$
\psi_7^{-1}: \R^4 \to U_7, \qquad \psi_7^{-1}(x_7, y_7, x_8, y_8) = [x_1, 0, x_2, 0, \ldots, x_6, 0, x_7, y_7, x_8, y_8]
$$
where
\begin{align*}
x_1 &= \sqrt{2 - \vert z_7 \vert^2 + \vert z_8 \vert^2}, &&& x_3 &= \sqrt{6 - \vert z_7 \vert^2}, &&& x_5 &= \sqrt{4 + \vert z_7 \vert^2 - \vert z_8 \vert^2}, \\
x_2 &= \sqrt{6 - 2\vert z_7 \vert^2 + \vert z_8 \vert^2}, &&& x_4 &= \sqrt{8 - \vert z_8 \vert^2},  &&& x_6 &= \sqrt{2 + 2\vert z_7 \vert^2 - \vert z_8 \vert^2}.
\end{align*}
Thus we may consider $z_1, \ldots, z_6$ as lying in $\mathbb{R}^+$. Consider $z_7 = x_7 + iy_7$ and $z_8 = x_8 + iy_8$. From $\abs{z_1}^2 = 2$ we conclude $z_1 = x_1 = \sqrt{2}$ and thus $\vert z_7 \vert^2 = \vert z_8 \vert^2$.
The equation $\mfR(\overline{z_2} \: \overline{z_3} \: \overline{z_4} \: z_6 \: z_7 \: z_8) = 0$ becomes $z_2 z_3 z_4 z_6 (x_7x_8 - y_7y_8) = 0$ and since $z_1, \ldots, z_6 \neq 0$ in $U_7$, we deduce $x_7x_8 = y_7y_8$.
Keeping the above deductions in mind, we calculate
\begin{equation*}
x_7^2 \vert z_8 \vert^2 = x_7^2 x_8^2 + x_7^2 y_8^2 = y_7^2 y_8^2 + x_7^2 y_8^2  y_8^2 \vert z_7 \vert^2 
= y_8^2 \vert z_8 \vert^2.
\end{equation*}
If $\vert z_8 \vert^2 \neq 0$, this implies $x_7 = \pm y_8$ and thus $y_7 = \pm x_8$, where we need to choose the same sign to satisfy $x_7x_8 = y_7y_8$. Hence, we get two possible values for $z_7$, namely
$$ z_7 = y_8 + ix_8 = i\ \overline{z_8} \quad \mbox{ or } \quad z_7 = -y_8 - ix_8 = -i\ \overline{z_8}. $$
If $\vert z_8 \vert^2 = 0$, then $z_8 = 0$ and by $\abs{z_7}^2 = \abs{z_8}^2$ also $z_7 = 0$. Then $z_7 = \pm i\ \overline{z_8}$ still holds, but does no longer give two different values. 
Moreover, the manifold equations \eqref{manifold eqn} imply $\vert z_8 \vert^2 = \vert z_7 \vert^2 \leq 6$.  
Note that the inequality $\vert z_8 \vert^2 < 6$ is strict for points in $U_7$, since otherwise $z_2 = 0$ or $z_3 = 0$.

It remains to consider the points that do not lie in $U_7$, i.e., at least one of the values $z_1, \ldots, z_6$ is zero. We check the following cases:
\begin{itemize}
\item 
$z_1 = 0$ is not possible due to $J(z) = \frac{1}{2}\abs{z_1}^2 =1$.
\item 
If $z_2 = 0$ we still have $z_1 \neq 0$, but $z$ lies in $U_2$ and thus we have $z_1 = \sqrt{2}$. We use the manifold equations \eqref{manifold eqn} to find the other coordinates of the point and get
$$ z = [\sqrt{2}, 0, 0, \sqrt{2}, 2, 2\sqrt{2}, \sqrt{6}, \sqrt{6}] $$
which are the coordinates of the focus-focus point $B$.
\item 
If $z_3 = 0$, then $\vert z_1 \vert^2 = 2$ and the manifold equations imply $\vert z_5 \vert^2 = 4$, $\vert z_7 \vert^2 = 6$ and eventually $z_2 = 0$, so we recover the focus-focus point $B$ again.
\item 
If $z_4 = 0$ we calculate $\vert z_1 \vert^2 = 2$, $\vert z_5 \vert^2 = 4$, $\vert z_7 \vert^2 = 8$ which leads to the contradiction $\vert z_3 \vert^2 = -2$, so $z_4 \neq 0$ must be true.
\item 
If $z_5 = 0$ then $\vert z_1 \vert^2 = 6$ which contradicts $J(z) =\frac{1}{2}\abs{z_1}^2= 1$, so $z_5 \neq 0$ must be true.
\item 
If $z_6 = 0$ then the manifold equations \eqref{manifold eqn} imply $\vert z_1 \vert^2 = 2$, $\vert z_5 \vert^2 = 4$ which lead to the contradiction $\vert z_7 \vert^2 = -2$, so $z_6 \neq 0$ must be true.
\end{itemize}
Therefore the only point in $F_{\frac{1}{2}}^{-1}\left(1,0\right)$ which does not lie in $U_7$ is the focus-focus point $B$. This point can be included in the previous description of the fibre when we allow $\vert z_8 \vert = \sqrt{6}$. Indeed, both the positive and negative choice of the sign of $z_7$ results in the point $B$.
Altogether we conclude that the fibre $F_{\frac{1}{2}}(1,0)$ is of the form
$$ 
\left\{ \left. \left[\sqrt{2}, \sqrt{6 - \vert z_8 \vert^2}, \sqrt{6 - \vert z_8 \vert^2}, \sqrt{8 - \vert z_8 \vert^2},\ 2, \sqrt{2 + \vert z_8 \vert^2}, \ \pm i \ \overline{z_8}, \ z_8 \right] \in M \ \right| \vert z_8 \vert \in \left[0, \sqrt{6}\right] \right\}. 
$$
Concerning the reformulation in polar coordinates, we remark that $r = 0$ recovers the point $A$ and $r = \sqrt{6}$ gives the point $B$. For all other values of $r$, the fibre splits up in two parts, one for the plus and one for the minus sign in the seventh coordinate. The rotation related to the $\mbS^1$-action induced by $J$ is described by the parameter $\theta$, so indeed, we get a 2-torus pinched at the points $A$ and $B$, as visualised in Figure \ref{Fig_double_pinched}.
\end{proof}


\newpage

{\small
  \noindent
  \\
  Annelies De Meulenaere\\
  University of Antwerp\\
  Department of Mathematics\\
  Middelheimlaan 1\\
  B-2020 Antwerpen, Belgium
  
  \hspace{-7mm}
   \begin{tabular}{ll}
   {\em E\--mail}: & \texttt{demeulenaere.annelies {\em AT} gmail.com} \\
   &  \texttt{Annelies.DeMeulenaere {\em AT} student.uantwerpen.be}
  \end{tabular}

  \vspace{5mm}

{\small
  \noindent
  \\
  Sonja Hohloch\\
  University of Antwerp\\
  Department of Mathematics\\
  Middelheimlaan 1\\
  B-2020 Antwerpen, Belgium\\
  {\em E\--mail}: \ \ \  \texttt{sonja.hohloch {\em AT} uantwerpen.be}

\end{document}